\documentclass[a4paper,11pt,twoside]{article}

\usepackage{amssymb, amsmath, amsfonts, amsthm, amscd}
\usepackage[left=2.75cm, right=3.25cm]{geometry} 

\usepackage[english]{babel}
\usepackage{verbatim} 

\input xy
\xyoption{all}

\newtheorem{theorem}{Theorem}[section]

\newtheorem{lemma}[theorem]{Lemma}
\newtheorem{proposition}[theorem]{Proposition}
\newtheorem{corollary}[theorem]{Corollary}
\newtheorem*{thma}{Theorem}

\newtheorem*{coronotnumb}{Corollary}

\theoremstyle{definition}

\newtheorem*{definition}{Definition}

\newenvironment{rem}[1][Remark.]{\begin{trivlist}
\item[\hskip \labelsep {\bfseries #1}]}{\end{trivlist}}

\renewcommand{\Re}{ \operatorname{Re} }
\renewcommand{\Im}{ \operatorname{Im} }

\newcommand{\Lie}{\mathfrak}

\newcommand{\son} { \operatorname{SO} }
\newcommand{\spin}{ \operatorname{Spin} }

\newcommand{\tor} { \operatorname{\tau} }
\newcommand{\normaltor} { \mathcal{T} }

\newcommand{\sln} { \operatorname{SL} }
\newcommand{\ort} { \operatorname{O} }
\newcommand{\su}  { \operatorname{SU}(2) }
\newcommand{\gln} { \operatorname{GL} }
\newcommand{\psl} { \operatorname{PSL} }
\newcommand{\Hol} { \operatorname{Hol} }
\newcommand{\Vol} { \operatorname{Vol} }
\newcommand{\Isom}{ \operatorname{Isom} }

\newcommand{\supp}{ \operatorname{supp} }

\newcommand{\lift}{ \tilde }

\newcommand{\cs}{\operatorname{cshape}}

\newcommand{\tr}{ \operatorname{trace} }
\newcommand{\length}{ l }

\newcommand{\pclsp}{ \operatorname{\mu_{sp}} }
\newcommand{\plsp}{ \operatorname{\mu_{lsp}} }

\newcommand{\diam}{ \operatorname{diam} }

\newcommand{\cmplx}{ \mathbf{C} }
\newcommand{\hyp}{ \mathbf{H} }
\newcommand{\real}{ \mathbf{R} }
\newcommand{\integer}{ \mathbf{Z} }

\newcommand{\disk}{\overline{D}}

\newcommand{\modtwo}{ \mathbf{Z}/2\mathbf{Z} }

\newcommand{\cohom}{ \operatorname{H} }

\newcommand{\idest}{i.e.\;}

\title{Higher-dimensional Reidemeister torsion invariants for cusped hyperbolic $3$-manifolds}
\author{Pere Menal-Ferrer \and Joan Porti 
\thanks{ Both authors partially supported by the 
Spanish Micinn through grant MTM2009-0759 and by 
the Catalan AGAUR through grant SGR2009-1207. 
The second author received the prize “ICREA Acad\`emia” 
for excellence in research, funded by the Generalitat de Catalunya.}
} 

\begin{document}
\maketitle
\begin{abstract}
	For an oriented finite volume hyperbolic $3$-manifold $M$ 
	with a fixed spin structure $\eta$, we consider a sequence of invariants 
	$\{\normaltor_n(M; \eta)\}$. 
	Roughly speaking, $\normaltor_n(M; \eta)$ is the Reidemeister torsion of $M$ with respect to the 
	representation given by the composition of the lift of the holonomy representation 
	defined by $\eta$, and the $n$-dimensional, irreducible, complex representation of $\sln(2,\cmplx)$.
	In the present work, we focus on two aspects of this invariant:
	its asymptotic behaviour and its relationship with the complex-length spectrum of the manifold.
	Concerning the former, we prove that for suitable spin structures, $\log |\normaltor_n(M; \eta)| \sim -n^2 \frac{\Vol M}{4\pi}$,
	extending thus the result obtained by W.\;M\"uller for the compact case in \cite{Mul}.
	Concerning the latter, we prove that the sequence $\{|\normaltor_n(M; \eta)|\}$
	determines the complex-length spectrum of the manifold up to complex conjugation. 
\end{abstract}

\section{Introduction}

Let $M$ be an oriented, complete, hyperbolic three-manifold of finite volume. 
The hyperbolic structure of $M$ yields the holonomy representation:
\[
\Hol_M \colon \pi_1(M, p) \to \Isom^+ \hyp^3,
\]
where $\Isom^+ \hyp^3$ denotes the orientation preserving isometry group of hyperbolic 3-space $\hyp^3$. 
Using the upper half-space model, $\Isom^+ \hyp^3$ is naturally identified 
with $\psl(2,\cmplx) = \sln(2,\cmplx)/\{\pm 1\}$. 
It is known that $\Hol_M$ can be lifted to $\sln(2,\cmplx)$;
moreover, such lifts are in canonical one-to-one correspondence with spin structures on $M$. 
Thus, attached to a fixed spin structure $\eta$ on $M$, we get a representation
\[
\Hol_{(M,\eta)} \colon \pi_1(M, p) \to \sln(2,\cmplx).
\]
On the other hand, for all $n > 0$ there exists a unique (up to isomorphism)
$n$-dimensional, complex, irreducible representation of the Lie group $\sln(2,\cmplx)$, say:
\[
\varsigma_n \colon \sln(2,\cmplx) \to \sln(n,\cmplx).
\]
Hence, composing $\Hol_{(M,\eta)}$ with $\varsigma_n$ we get the following representation:
\[
\rho_n \colon \pi_1(M, p) \to \sln(n,\cmplx).
\]
This representation will be called the \emph{canonical n-dimensional representation}
of the spin-hyperbolic manifold $(M,\eta)$. 

Roughly speaking, the Reidemeister torsion invariants that we want to study are those coming from $\rho_n$.
The first issue that arises in trying to define the Reidemeister torsion 
concerns the cohomology groups $\cohom^*(M; \rho_n)$ 
(\idest the cohomology groups of $M$ in the local system defined by $\rho_n$).
If all these groups vanish, then it makes sense to consider the Reidemeister torsion $\tau(M; \rho_n)$;
however, if some of them are not trivial, then a choice of bases for $\cohom^*(M; \rho_n)$ is required.

An important case for which $\rho_n$ is acyclic (\idest $\cohom^*(M; \rho_n) = 0$) 
is when $M$ is closed. This is a particular case of Raghunathan's vanishing theorem.
For $M$ closed, the invariant $\tau(M; \rho_n)$ has been considered by W.\;M\"uller in \cite{Mul},
and for $n = 3$ by J.\;Porti in \cite{Porti}.

In general, the representation $\rho_n$ does not need to be acyclic for a cusped manifold $M$.
Therefore, we need to choose bases in (co)homology to define $\tau(M; \rho_n)$.
Obviously, if we want an invariant of the manifold, these bases must be chosen in a somehow canonical way.
Unfortunately, we do not know if this is possible. 
J.\;Porti proved in \cite{Porti} that for $n = 3$ a natural choice of bases can be done once a basis for 
$\cohom^1(\partial \overline M; \integer)$ is chosen.
Using the same approach, we will prove the following result: 
given non-trivial cycles $\{ \theta_i \}$ in $\cohom_1(\partial \overline M; \integer)$,
one for each connected component of $\partial \overline M$, 
there is a canonical family of bases of $\cohom^*(M; \rho_n)$ 
such that any member of this family yields the same Reidemeister torsion,
say $\tau (M; \rho_n; \{\theta_i\})$. Moreover, we will show that for $ k > 1$ the following quotients
are independent of the choices $\{\theta_i\}$:
\begin{eqnarray*}
	\normaltor_{2k+1}(M, \eta) & := & \frac{\tau( M; \rho_{2k + 1}; \{ \theta_i \})}
	{\tau( M; \rho_3; \{ \theta_i \} )} \in \cmplx^*/\{\pm 1\} ,\\
	\normaltor_{2k  }(M, \eta) & := & \frac{\tau( M; \rho_{2k}; \{ \theta_i \} )}
	{\tau( M; \rho_2; \{ \theta_i \})} \in \cmplx^*/\{\pm 1\}.
\end{eqnarray*}
Thus, for all $n \geq 4$, $\normaltor_n(M,\eta)$ is an invariant of the spin-hyperbolic manifold $(M,\eta)$.
Notice that if $n$ is odd the quantity $\normaltor_n(M,\eta)$ is independent of the spin structure
(this is an immediate consequence of the fact that an odd 
dimensional irreducible complex representation of $\sln(2,\cmplx)$ factors through $\psl(2,\cmplx)$),
and hence we will denote it simply by $\normaltor_{2k + 1}(M)$.
The invariant $\normaltor_n(M,\eta)$ will be called the \emph{normalized $n$-dimensional Reidemeister torsion}
of the cusped spin-hyperbolic manifold $M$. We will also refer to these invariants
as the higher-dimensional Reidemeister torsion invariants. 
These invariants are the focus of study of the present paper.

\begin{rem}
	It is possible to assign a well defined sign to $\normaltor_n(M,\eta)$:
	if $n$ is even, this can be done for $\tau(M;\rho_n)$ (see Turaev's book \cite{Turaev});
	if $n$ is odd, this can be done because, roughly speaking,
	the sign indeterminacy of $\tau(M;\rho_n)$ is the same for $\tau(M;\rho_3)$. 
	In spite of this, we will work up to sign in general, as our main results concern just the modulus of $\normaltor_n(M,\eta)$.
\end{rem}

To simplify the exposition in this introduction, we will restrict ourselves to the odd-dimensional case.
Thus we do not need any spin structure on $M$.
Our main result concerns the asymptotic behaviour of $\{\normaltor_{2k + 1}(M)\}$.
\begin{thma}
	Let $M$ be an oriented, complete, finite-volume, hyperbolic $3$-manifold. Then 
	\[
	\lim_{k\to \infty} \frac{\log|\normaltor_{2k+1}(M)|}{(2k +1)^2} = - \frac{\Vol(M)}{4\pi}.
	\] 
\end{thma}

For $M$ compact, this result was established by W.\;M\"uller in \cite{Mul}.
To explain our approach to this theorem, we need to discuss M\"uller's result.

Let us assume that $M$ is closed. According to M\"uller's Theorem on the equivalence between
Reidemeister torsion and Ray-Singer torsion for \emph{unimodular} representations (see \cite{MulUni}),
we have
\[
|\tau(M; \rho_n)| = \text{Tor}(M; \rho_n),
\]
where $\text{Tor}(M; \rho_n)$ is the Ray-Singer torsion of $M$ with respect to $\rho_n$.
For a hyperbolic manifold and a unitary representation $\rho$, 
D.\;Fried established in \cite{Fried1} and \cite{Fried2} a deep relationship between $\text{Tor}(M,\rho_n)$ and the twisted Ruelle zeta function.
The twisted Ruelle zeta function of $M$ and $\rho$ is formally defined by
\begin{equation} \label{intro:Ruelle}
	R_{\rho} (s) = \prod_{\varphi \in \mathcal{PC}(M)} \det\left( \text{Id} - \rho_n(\varphi) e^{-s\length(\varphi)}\right),
\end{equation}
where $\mathcal{PC}(M)$ denotes the set of oriented, prime, closed geodesics in $M$, and $\length(\varphi)$ is the length of $\varphi$
(we are using the identification between $\mathcal{PC}(M)$ and the set of hyperbolic conjugacy classes of $\pi_1 M$,
so the expression appearing inside the product above makes sense).
D.\;Fried proved that, for any representation $\rho$, the function 
$R_{\rho}(s)$ admits a meromorphic extension to the whole plane; moreover,
if $\rho$ is assumed to be acyclic and unitary, then $|R_{\rho} (0)| = {\text{Tor}(M,\rho_n)}^{2}$. 
The work of U.\;Br\"ocker \cite{Brocker} and A.\;Wotzke \cite{Wotzke} shows that a similar 
result also holds for a compact hyperbolic manifold and representations of its fundamental group
arising from representations of $\Isom^+\hyp^n$.
In our particular case, the result is the following.  
\begin{thma}[\textbf{A.\;Wotzke, \cite{Wotzke}}]
	Let $(M,\eta)$ be a compact spin-hyperbolic $3$-manifold. 
	Then, for $n > 1$, $R_{\rho_n}(s)$ admits a meromorphic extension to the whole complex plane and
	\[
	|R_{\rho_n}(0)| = {\operatorname{Tor}(M; \rho_n)}^{2}.
	\]
\end{thma}

U.\;Br\"ocker established in \cite{Brocker} a functional equation for $R_{\rho_n}(s)$
involving the volume of the manifold. Using this equation and other related material,
M\"uller has recently established in \cite{Mul} the following formula for $|\tau(M; \rho_n)|$, which involves the volume  
of the closed manifold $M$ and some related Ruelle zeta functions $R_k(s)$,
\begin{equation} \label{intro:RuelleMuller}
	\log \left| \frac{\tor(M, \rho_{2k +1}) }{\tor(M, \rho_5)} \right|  =  
		 \sum_{j=3}^k \log{|R_{2j}(j)| } - \frac{1}{\pi} \Vol M \left( k(k+1) -6 \right).
\end{equation}
One of the advantages of this formula is that the Ruelle zeta functions $R_k(s)$ are evaluated
inside the corresponding region of convergence, and hence they have an expression similar to that of Equation \eqref{intro:Ruelle}.
The result about the asymptotics of the torsion is then deduced by showing that the sum appearing in the right hand side 
of Equation \eqref{intro:RuelleMuller} is uniformly bounded on $k$.

In trying to adapt M\"uller's proof to the non-compact case,
some difficulties arise, the main one being the fact that the Ray-Singer torsion is \emph{a priori} 
not defined for non-compact manifolds.
Nevertheless, the terms appearing in Equation \eqref{intro:RuelleMuller} still make sense for cusped manifolds.
Thus this equation is meaningful for such manifolds also;
we prove that this true in Section \ref{chapter:asymptotic}.
Roughly speaking, our proof will consist in approximating the manifold $M$ by 
the compact manifolds $\{M_{p/q}\}$ obtained by performing Dehn fillings on $M$.
Then we will get a formula relating $\normaltor_{2k+1}(M)$ and $\normaltor_{2k+1}(M_{p/q})$ in Section \ref{chapter:surgery}.
This will be done using a Mayer-Vietoris argument. As a by-product of this formula, 
the behaviour of the higher-dimensional Reidemeister torsion invariants under Dehn filling will be established as well.

The other thing we must take into account concerns the limit of the 
Ruelle zeta functions of the manifolds $M_{p/q}$ as $(p,q)$ goes to infinity.
Our main tool to deal with this will be the continuity of the complex-length spectrum, which we briefly discuss now.
\begin{definition}
	The \emph{prime complex-length spectrum} of $M$, denoted $\pclsp M$,
	is the measure on $\cmplx$ defined by
	\[
	\pclsp M = \sum_{ \varphi \in \mathcal{PC}(M)} \delta_{ e^{ \lambda(\varphi) } },
	\]
	where $\lambda$ is the complex-length function of $M$, and $\delta_x$ denotes the Dirac measure centered at $x$.
	In other words, $\pclsp M$ is the image measure of the counting measure in $\mathcal{PC}(M)$ 
	under the exponential of the complex-length function. 
\end{definition}
\begin{rem}
	The complex-length spectrum is usually regarded as a collection of complex numbers and multiplicities.
	This is of course equivalent to our definition;
	however, we think that regarding it as a measure puts some questions in a natural context. 
\end{rem}

We can consider the prime complex-length spectrum as a map from $\mathcal M$, 
the set of complete oriented hyperbolic $3$-manifolds of finite volume, to 
$M(\cmplx\setminus \overline D)$, the set of measures on the exterior of the unit disc $\overline D$.
Both spaces are endowed with natural topologies: the former with the geometric topology, and the latter 
with the topology of weak convergence. 
Using standard techniques from hyperbolic geometry we will prove
the continuity of this map in Section \ref{chapter:cmplxlength}.
\begin{thma}
	The map $\pclsp \colon \mathcal M \to M(\cmplx\setminus \overline D)$
	which assigns to every finite volume complete oriented hyperbolic $3$-manifold its complex-length spectrum
	is continuous.
\end{thma}

With this formalism, Equation \eqref{intro:RuelleMuller} can be expressed in terms of the complex-length spectrum measure.
Using some complex analysis, we will prove that  
if we know all the values $\{|\normaltor_{2k +1} (M)|\}_{k\geq N}$,
for some $N \geq 4$, then we also know the values of the following integrals
\[
M_k = \int_{|z| > 1} ( z^{-k} + {\bar z}^{-k} ) \operatorname{d}\pclsp M(z), \quad k \geq N.
\]
Using the Cauchy transform we will prove that for this kind of measure 
this information is enough to recover the measure up to complex conjugation, that is,
we do not know $\pclsp M$ but 
\[
\pclsp M + \overline {\pclsp M},
\]
where $\overline{\pclsp M}$ denotes the image measure of $\pclsp M$ under complex conjugation.
As a consequence, we will obtain the following result.
\begin{thma}
	Let $M$ be an oriented complete hyperbolic $3$-manifold of finite volume.
	For all $N \geq 4$, the sequence of values $\{|\normaltor_{2k +1} (M)|\}_{k \geq N}$ determines
	the complex-length spectrum of $M$ up to complex conjugation.
\end{thma}
\begin{rem}
	This theorem may be regarded as a geometric interpretation of the 
	information encoded in the higher-dimensional Reidemeister torsion invariants.
\end{rem}
As a particular case, if $M$ admits an orientation-reversing isometry 
(this is for instance the case of the complement of the figure eight knot), 
then $\overline {\pclsp M} = \pclsp M$, and hence the sequence $\{|\normaltor_{2k +1} (M)|\}_{k\geq N}$
determines the complex-length spectrum completely. 

Using Wotzke's Theorem we obtain the following corollary of this theorem.
\begin{coronotnumb}
	Let $M$ be an oriented \emph{compact} hyperbolic $3$-manifold.
	Knowing the invariants $|\normaltor_{2k +1} (M)|$ for all $k \geq N \geq 4$
	is equivalent to knowing the complex-length spectrum of $M$ up to complex conjugation.
\end{coronotnumb}

\section{Spin-hyperbolic three-manifolds}

The aim of this section is to review and establish some facts and constructions 
concerning a spin-hyperbolic $3$-manifold. 
The definition of the object under consideration is quite obvious.

\begin{definition}
	A \emph{spin-hyperbolic} $3$-manifold is a pair $(M, \eta)$ where $M$ is an oriented hyperbolic $3$-manifold 
	and $\eta$ a spin structure on $M$.
\end{definition}

The first subsection reviews the relation between spin structures and lifts of the holonomy representation;
although this material is well known (see for instance \cite{culler}), we think it is worth outlining it here
in an elementary and self-contained way.
In the second subsection we give the definition of the $n$-dimensional canonical representation of a spin-hyperbolic $3$-manifold;
some basic results about irreducible finite-dimensional complex representations of $\sln(2,\cmplx)$ are also recalled.

\subsection{Lifts of the holonomy representation}
\label{section:lifts}

Let $M$ be a connected, oriented, hyperbolic $3$-manifold which is not necessarily complete.
We will use the following definition of a spin structure, see \cite{Kirby}. 
The $\son(3)$--principal bundle of orthonormal positively-oriented frames on $M$ is denoted by $P_{\son(3)} M$.

\begin{definition}
	A \emph{spin structure} on $M$ is a (double) cover of $P_{\son(3)} M$ by a $\spin(3)$--principal bundle over $M$. 
\end{definition}

This definition is equivalent to saying that a spin structure on $M$ 
is a double cover of $P_{\son(3)} M$ such that the preimage of any fiber of $P_{\son(3)} M$ is connected.
One can deduce from this observation that there is a natural identification between the set of spin structures on $M$ 
and the following set:
\[
\left \{ \alpha \in \cohom^1(P_{\son(3)} M; \modtwo) \mid i^*(\alpha) = 1 \in \cohom^1(\son(3); \modtwo) \right \}.
\]

On the other hand, the hyperbolic structure of $M$ defines a canonical flat $\Isom^+ \hyp^3$--principal bundle over $M$,
see \cite{Thurston3d}.
Let us recall how it is defined. Let $\hyp^3$ be hyperbolic space of dimension three with a fixed orientation.
Consider an $(\Isom^+ \hyp^3, \hyp^3)$--atlas on $M$ defining the hyperbolic structure.
Thus we have local charts $\phi_i \colon U_i \to \hyp^3$ covering $M$
such that the changes of coordinates are restrictions of orientation-preserving isometries of $\hyp^3$. 
We can assume that the local charts preserve the fixed orientations on both $M$ and $\hyp^3$.
Let $\psi_{ij}$ be the change of coordinates from 
$(\phi_j, U_j)$ to $(\phi_i, U_i)$, that is,
\[
\psi_{ij}\colon U_i\cap U_j \to \Isom^+ \hyp^3, \quad \psi_{ij} \circ \phi_j = \phi_i.
\]
The analyticity of the elements of $\Isom^+ \hyp^3$ implies that $\psi_{ij}$ is a locally constant map. 
Since these maps also satisfy the cocycle condition  
$\psi_{ij} \circ \psi_{jk} = \psi_{ik}$,
they define a flat $\Isom^+ \hyp^3$--principal bundle over $M$, 
\[ 
\Isom^+ \hyp^3 \to P_{\Isom^+ \hyp^3} M \overset{\pi}{\to} M.
\]
Let us fix a base point $p\in M$. Given $u\in P_{\Isom^+ \hyp^3} M$ with $\pi(u) = p$, 
it makes sense to consider the holonomy representation of this principal bundle,
\[
\Hol_u \colon \pi_1(M, p) \to \Isom^+ \hyp^3.
\]
By definition, if $\sigma \colon [0,1] \to M$ is a loop based at $p$, 
$\Hol_u(\sigma)$ is the unique element of $\Isom^+ \hyp^3$ such that 
\[
\tilde{\sigma}(1) \cdot \Hol_u(\sigma) = \tilde{\sigma}(0),
\]
where $\tilde{\sigma}(t)$ is the horizontal lift of $\sigma(t)$ starting at $u$.
It can be checked that this holonomy agrees, up to a conjugation, with the holonomy given in terms of the developing map.
In other words, for some suitable initial choices, we have $\Hol_u = \Hol_M$.

In what follows, we will identify $\Isom^+ \hyp^3$ with $\psl(2,\cmplx)$.
\begin{proposition} \label{prop: liftssl2}
	There is a canonical one-to-one correspondence between the following sets:
	\begin{enumerate}
		\item The set of covers of $P_{\psl(2,\cmplx)} M$ by $\sln(2,\cmplx)$--principal bundles over $M$. 
		\item The set of lifts of $\Hol_M$ to $\sln(2,\cmplx)$.
	\end{enumerate}
\end{proposition}
\begin{proof}
	Let us assume that we have chosen a base point
	$u\in P_{\psl(2,\cmplx)} M$ with $\pi(u) = p \in M$ such that $\Hol_u = \Hol_M$.
	Let $P_{\sln(2,\cmplx)} M$ be an $\sln(2,\cmplx)$--principal bundle over $M$ covering 
	$P_{\psl(2,\cmplx)} M$. Take one of the two points 
	$\tilde{u} \in P_{\sln(2,\cmplx)}M$ that projects to $u$, and consider 
	the corresponding holonomy representation $\Hol_{\tilde{u}}$. 
	It is clear that $\Hol_{\tilde{u}}$ is a lift of $\Hol_u$; 
	moreover, it is independent of the choice of the base point $\tilde{u}$, 
	for the other choice is obtained by conjugating it by $-\operatorname{Id} \in \sln(2, \cmplx)$.
	This gives a well-defined correspondence between 
	the set of covers of $P_{\psl(2,\cmplx)} M$ by $\sln(2,\cmplx)$--principal bundles over $M$
	and the set of lifts of $\Hol_M$ to $\sln(2,\cmplx)$.
	Finally, this correspondence is one-to-one because we can recover the flat bundle from its holonomy representation. 
\end{proof}

Next, we want to embed the frame bundle $P_{\son(3)} M$ into $P_{\psl(2,\cmplx)} M$.
To that end, identify $\psl(2,\cmplx)$ with $P_{\son(3)} \hyp^3$
by fixing a positively-oriented frame $R_O \in P_{\son(3)} \hyp^3$ based at $O \in \hyp^3$.
Notice that this gives a concrete embedding of $\son(3)$ into $\psl(2,\cmplx)$ 
as the isometry group of the tangent space at $O$ with fixed basis $R_O$.
Now let $u \in P_{\psl(2,\cmplx)} M$ and $p = \pi(u)$. 
A local chart $(\phi_j, U_j)$ of the hyperbolic structure containing $p$
gives a local trivialization $U_j \times \psl(2,\cmplx)$ of $P_{\psl(2,\cmplx)} M$,
with respect to which the point $u$ is written as a pair $(p, g)\in U_j \times \psl(2,\cmplx)$. 
We will say that $u$ \emph{is based at} $p \in M$  if $g \in \psl(2,\cmplx) \cong P_{\son(3)} \hyp^3$ 
is a frame based at $\phi_j(p)$. It can be checked that this definition does not 
depend on the choice of the local chart $(\phi_j, U_j)$,
and that we have the following identification:
\[
\left \{ u \in P_{\psl(2,\cmplx)} M \mid u \text{ is a frame based at } \pi(u)  \right \} \cong P_{\son(3)} M.
\]
Thus we have obtained a concrete embedding $P_{\son(3)} M \hookrightarrow P_{\psl(2,\cmplx)} M $,
which is easily seen to be compatible with the actions of the respective structural groups 
$\son(3)$ and $\psl(2,\cmplx)$.
In other words, we have an explicit reduction of the structural group with respect to the fixed embedding $\son(3)\subset \psl(2,\cmplx)$.
Although this embedding depends on the choices that we have done, 
it must be pointed out that its homotopy class does not.

\begin{proposition}
	\label{prop:spin_lifts}
	There is a canonical one-to-one correspondence between the following sets:
	\begin{enumerate}
		\item The set of covers of $P_{\psl(2,\cmplx)} M$ by $\sln(2,\cmplx)$--principal bundles over $M$. 
		\item The set of spin structures on $M$.
	\end{enumerate}
\end{proposition}
\begin{proof}
	The set of spin structures on $M$ is canonically identified with the following set:
	\[
	\left \{ \alpha \in \cohom^1(P_{\son(3)} M; \modtwo) \mid i^*(\alpha) = 1 \in \cohom^1(\son(n); \modtwo) \right \}.
	\]
	The same argument shows that the set of covers of $P_{\psl(2,\cmplx)} M$ by $\sln(2,\cmplx)$--principal bundles over $M$
	is identified with 
	\[
	\left \{ \alpha \in \cohom^1( P_{\psl(2,\cmplx)} M; \modtwo) \mid i^*(\alpha) = 1 \in \cohom^1(\sln(2,\cmplx); \modtwo) \right \}.
	\]
	The result then follows from the fact that the map $P_{\son(3)} M \hookrightarrow P_{\Isom^+ \hyp^3} M$ defined above,
	whose homotopy class is canonical, is a homotopy equivalence, for $\son(3) \simeq \psl(2, \cmplx)$.
\end{proof}

\begin{corollary}
	The holonomy representation of a hyperbolic $3$-manifold 
	can be lifted to $\sln(2,\cmplx)$.
	The number of such lifts is $|\cohom^1(M;\modtwo)|$. 
\end{corollary}
\begin{proof}
	An oriented $3$-manifold admits $|\cohom^1(M;\modtwo)|$ 
	different spin structures.
\end{proof}

\section{Positive spin structures}

Let $M$ be a complete, oriented, hyperbolic $3$-manifold of finite volume.
Thus $M$ is the interior of a compact manifold whose boundary consists of 
tori $T_1,\dots, T_k$.

\begin{definition}
	We will say that a spin structure $\eta$ on $M$ is \emph{positive} 
	on $T_i$ if for all $g \in \pi_1 T_i$ we have:
	\[
	\tr \Hol_{(M,\eta)}(g) = +2.
	\]
	Otherwise, we will say that $\eta$ is \emph{nonpositive} on $T_i$.
\end{definition}

By \cite{MenalPorti},  $\eta$ is {nonpositive} on $T_i$ if and only if  
$H^*(T_i;\Hol_{(M,\eta)})$ is trivial.

\begin{definition}
A spin structure $\eta$ is \emph{acyclic}
if   $H^*(T_i;\Hol_{(M,\eta)})$ 
is trivial for each connected component $T_i$ of the boundary. Equivalently, $\eta$ is
nonpositive on each $T_i$.
 \end{definition}

The aim of this subsection is to prove the existence of spin structures that are acyclic.
Let $T^2$ be  a peripheral torus. 
We can assume that $T^2$ is a horospheric cross-section, and that 
\[
\Hol_M (\pi_1 T^2) = \left \langle
	\left[ \begin{pmatrix}  
		1 & 1 \\ 
		0 & 1
	\end{pmatrix} \right], 
	\left[ \begin{pmatrix}  
		1 & \tau \\ 
		0 & 1
	\end{pmatrix} \right]
	\right \rangle < \psl(2, \cmplx), \quad \text{ with } \Im \tau > 0.
\]

Let $P_{\son(3)} T^2 \subset P_{\psl(2,\cmplx)} T^2$ be the restriction of 
$P_{\son(3)} M \subset P_{\psl(2,\cmplx)} M$ to $T^2$, and let $\Hol_{T^2}$ be the restriction of 
$\Hol_M$ to $\pi_1 T^2$.
Using the Euclidean structure of $T^2$ and the outward normal vector of $T^2$,
we can construct a canonical (up to homotopy) section $s$ of the bundle $P_{\son(3)} T^2$
as follows: fix $p \in T^2$ and define $s(p) \in P_{\son(3)} T^2$ 
as any frame based at $p$ whose third component is equal to the outward normal vector at $p$;
for all $q \in T^2$, define $s(q)$ as the parallel transport (with respect to the Euclidean structure) of $s(p)$ along a curve joining $p$ and $q$ on $T^2$.
This yields a well-defined section which is canonical up to homotopy.
Thus we have a canonical trivialization $P_{\son(3)} T^2 \cong T^2 \times \son(3)$,
and hence a distinguished spin structure $T^2 \times \spin(3)$.
All other spin structures arise as quotients of the form
\[
\eta_{\alpha} = \left(\widetilde{T^2}\times \spin(3) \right) / \pi_1 T^2, 
\]
where $\alpha \in \cohom^1(T^2; \modtwo) = \operatorname{Hom}(\pi_1 T^2; \{ \pm 1\})$,
and the action of $\sigma \in \pi_1 T^2$ on $\spin(3)$ is by multiplication by $\alpha(\sigma) \operatorname{Id}$.
Therefore, spin structures of $P_{\son(3)} T^2$ are in canonical one-to-one correspondence with 
$\cohom^1(T^2;\modtwo)$. 

A similar argument proves the following result.

\begin{lemma}
	Let $\alpha \in \cohom^1(T^2; \modtwo) = \operatorname{Hom}(\pi_1 T^2; \{ \pm 1\})$, $\eta_\alpha$ be the associated spin structure on $T^2$,
	and $\Hol_{(T^2,\eta_\alpha)}$ be the corresponding lift of the holonomy representation.
	Then we have
	\[
	\alpha(\sigma) = \operatorname{sgn} \tr \Hol_{(T^2,\eta_\alpha)}(\sigma), \quad\text{for all } \sigma \in \pi_1 T^2.
	\]
\end{lemma}

Now we can prove the existence of acyclic spin structures.
\begin{proposition} \label{prop:existenceacyclicspin}
	Let $M$ be an oriented, complete, hyperbolic $3$-manifold of finite volume.
	For each boundary component $T_i$ take a closed simple curve $\gamma_i$.
	Then there exists a spin structure $\eta$ on $M$ such that  
	\[
	\tr{\Hol_{(M, \eta)}} ( [\gamma_i]) = -2, 
	\]
	where $[\gamma_i]$ denotes the conjugacy class of $\pi_1 (M, p)$ defined by $\gamma_i$.
\end{proposition}
\begin{proof}
	Let $N$ be the manifold obtained by performing a Dehn filling along each 
	of the curves $\{\gamma_i\}$. Fix a spin structure $\eta$ on $N$.
	We claim that the restriction of $\eta$ to $M$ gives the required spin structure.

	Assume that $\gamma$ is one of the curves $\gamma_i$, and that it is contained in a horospheric cross-section $T^2$.
	We can assume also that $\gamma$ is a closed geodesic with respect to the Euclidean structure of $T^2$. 
	Let $P_{\spin(3)} T^2 \to P_{\son(3)} T^2$ be the corresponding $\spin(3)$--bundle over $T^2$ defined by $\eta$,
	and $\alpha \in \cohom^1(T^2; \modtwo)$ the associated cohomology class.
	Consider the canonical section $s\colon T^2 \to P_{\son(3)} T^2$ constructed above using as starting frame one whose
	first vector is tangent to $\gamma$. Then the closed curve $s \circ \gamma$ can be lifted to $P_{\spin(3)} T^2$
	if and only if $\alpha(\gamma) = 1$. On the other hand, if $s \circ \gamma$ could be lifted to $P_{\spin(3)} T^2$,
	then such a lift could be extended to the added disk bounding $\gamma$ 
	(there is no obstruction in doing it because $\pi_1 \spin(3) = \{1\}$), and hence
	$s \circ \gamma$ could be extended to that disk, which is not possible by construction.
	Thus $\alpha(\gamma) = -1$, and the preceding lemma implies the result.
\end{proof}

As a corollary of the proof of the Proposition \ref{prop:existenceacyclicspin}, we obtain the following result.
\begin{corollary} \label{coro:spinextension}
	Let $\gamma \subset \partial M$ be a simple closed curve non-homotopically trivial in $\partial M$,
	and $M_\gamma$ be the manifold obtained by performing a Dehn filling along $\gamma$.
	A spin structure $\eta$ on $M$ extends to a spin structure on $M_\gamma$ if and only if 
	\[
	\tr \Hol_{(M,\eta)} (\gamma) = -2.
	\]
\end{corollary}

The following corollary of Proposition \ref{prop:existenceacyclicspin} gives a sufficient condition to guarantee
the existence of acyclic spin structures.
\begin{corollary} \label{coro:spinnonpositive}
	Assume that for each boundary component $T_i$ of $M$, 
	the map 
	\[
	\cohom_1 (T_i; \modtwo) \to \cohom_1(M; \modtwo)
	\]
	induced by the inclusion has non-trivial kernel.
	Then all spin structures on $M$ are nonpositive on each $T_i$ (i.e. acyclic). 
	In particular, if $M$ has only one cusp, all spin structures on $M$ are acyclic.
\end{corollary}
\begin{proof}
	If the hypothesis holds, then for each $T_i$ there exists a closed simple curve $\gamma_i \in T_i$
	that is zero in $\cohom_1(M; \modtwo)$. Take a spin structure on $M$ such that
	$\tr{\Hol_{(M, \eta)}} ( [\gamma_i]) = -2$, for all $\gamma_i$. Now let $\eta'$ be another spin structure on $M$, 
	and $\alpha \in \cohom^1(M; \modtwo)$ be the cohomology class relating $\eta$ and $\eta'$. Then, using multiplicative notation, we have
	\[
	\Hol_{(M, \eta')} ( \gamma_i) = \alpha(\gamma_i) \Hol_{(M, \eta)} ( \gamma_i). 
	\]
	Since $[\gamma] \in \cohom_1(M; \modtwo)$ is zero, we have $\alpha(\gamma_i) = 1$,
	and hence $\Hol_{(M, \eta')} ( \gamma_i)$ has trace $-2$, as we wanted to prove.
	The rest of the result follows from the fact that in any compact $3$-manifold $M$ 
	the map 
	\[
	i_* \colon \cohom_1(\partial M; \modtwo) \to \cohom_1( M; \modtwo)
	\]
	induced by the inclusion $i\colon \partial M \to M$ has a non-trivial kernel.
\end{proof}

\subsection{The canonical $n$-dimensional representation}
\label{section:sl2rep}
Irreducible, complex, finite-dimensional representations of $\operatorname{SL}(2,\mathbf C)$
are well known: for all $n$ there is exactly one irreducible representation of dimension $n$ which 
is given by $V_n = \operatorname{Sym}^{n-1} V_2$, 
the $(n- 1)$-th symmetric power of the standard representation $V_2 = \mathbf{C}^2$.
We use the convention that $\operatorname{Sym}^{0} V_2$ is the base field. 

\begin{definition}
	Let $(M,\eta)$ be a spin-hyperbolic $3$-manifold with holonomy representation
	$\Hol_{(M,\eta)}$. We define the \emph{canonical $n$-dimensional representation} of $M$
	as the composition of $\Hol_{(M,\eta)}$ with $V_n$.
\end{definition}

The decomposition into irreducible factors of the tensor product of two representations of 
$\sln(2,\cmplx)$ is given by the Clebsch-Gordan formula (see \cite[\S 11.2]{FultonHarris}).

\begin{theorem}[Clebsch-Gordan formula]
	\label{thm:CG}
	For nonnegative integers $n$ and $k$ we have:
	\[
	V_{n} \otimes V_{n+k} = \bigoplus_{i = 0}^{n-1} V_{2(n-i) + k - 1}.
	\]
\end{theorem}

\begin{lemma} 
	\label{lemma:pairing}
	Let $V$ be a finite-dimensional complex representation of $\operatorname{SL}(2,\mathbf C)$. 
	Then there exists a nondegenerate $\cmplx$--bilinear invariant pairing 
	\[
	\phi\colon V \times V \to \cmplx.
	\] 
	Moreover, if $V$ is irreducible, then there exists, up to multiplication by nonzero 
	scalars, a unique $\cmplx$--bilinear invariant pairing, which \emph{a fortiori} is non-degenerate.
\end{lemma}
\begin{proof}
	On one hand, the natural pairing between $V^*$ and $V$ always yields 
	a nondegenerate $\cmplx$--bilinear invariant map. 
	From the classification of irreducible representations 
	of $\operatorname{SL}(2,\cmplx)$, we deduce that
	$V^*$ is isomorphic to $V$, and hence the first part of the lemma is proved.
	On the other hand, invariant bilinear maps are in one-to-one correspondence
	with fixed vectors of $V^* \otimes V^*$. Thus the second assertion follows
	from the Clebsch-Gordan formula, which shows that 
	$(V_n\otimes V_n)^*\cong V_n\otimes V_n$ has a unique irreducible factor of
	dimension $1$, on which $\sln(2,\cmplx)$ acts trivially.
\end{proof}
\begin{rem}
	Roughly speaking, the $\cmplx$--bilinear invariant pairing on $V_n =\operatorname{Sym}^{n-1} V_2$
	is the $(n -1)$-th symmetric power of the determinant. 
	To be precise, let $\operatorname{S}(V_2)$ be the symmetric algebra on $V_2$, that is,
	\[
	\operatorname{S}(V_2) = \bigoplus_{i\geq 0} \operatorname{Sym}^i V_2.
	\]  
	With respect to a fixed basis $(e_1, e_2)$ of $V_2$, the determinant is given by:
	\[
	\operatorname{det} = e_1^*\otimes e_2^* - e_2^*\otimes e_1^*,
	\]
	where $(e_1^*, e_2^*)$ is the dual basis of $(e_1,e_2)$.
	The determinant thus can be regarded as an element of $\operatorname{S}(V^*) \otimes \operatorname{S}(V^*)$.
	This latter vector space is an algebra in a natural way, and hence it makes sense to consider the power
	$\operatorname{det}^n$. Notice that $\operatorname{det}^n \in \operatorname{Sym}^n(V^*) \otimes \operatorname{Sym}^n(V^*)$,
	so $\operatorname{det}^n$ defines a bilinear pairing on $V_{n+1}$. On the other hand, it can be checked that we have:
	\[
	g\cdot \operatorname{det}^n = (g\cdot \operatorname{det})^n, \quad \text{for all } g \in \sln(2,\cmplx).
	\]
	Hence, $\operatorname{det}^n$ is $\sln(2,\cmplx)$--invariant, for so is $\operatorname{det}$,
	and Lemma \ref{lemma:pairing} implies that this pairing is nondegenerate. 
	Notice also that $\operatorname{det}^n$ is alternating for $n$ odd and symmetric for $n$ even.
\end{rem}

From Lemma \ref{lemma:pairing} we get the following result (see \cite[Sec. 2.2]{goldman}), 
which will be used very often in the sequel. 
\begin{corollary}
	\label{corollary:PD}
	Poincar\'e duality with coefficients in $\rho_n$ holds. 
\end{corollary}

\section{Higher-dimensional Reidemeister torsion}
\label{chapter:def}

In this section we define the $n$-dimensional normalized Reidemeister torsion
for a complete spin-hyperbolic $3$-manifold of finite volume and an integer $n \geq 4$. 
We will refer to these invariants as the higher-dimensional Reidemeister torsion invariants.

Let $(M,\eta)$ be a spin-hyperbolic $3$-manifold, and $\rho_n$ be its
canonical $n$-dimensional representation.
We want to define the Reidemeister torsion of $M$ with respect to the 
representation $\rho_n$. However, to do that we need either $M$ to 
be $\rho_n$-acyclic (\idest the groups $\cohom^*(M ; \rho_n)$ are all trivial),
or, if it does not happen, to fix bases on (co)homology.

If $M$ is compact, then, as a particular case of Raghunathan's vanishing theorem (see \cite{MenalPorti}),
the cohomology groups $\cohom^*(M ; \rho_n)$ are all trivial. 
Thus for $M$ closed the Reidemeister torsion $\tau(M; \rho_n)$ is defined. 

On the other hand, if $M$ is non-compact, these groups need not to be trivial.
Thus we need to choose bases in (co)homology in that case. 
Of course, if we want to get an invariant of the manifold we must choose bases
in a somehow canonical way. Unfortunately, we do not know how to do this.
Nevertheless, we have at least the following results. 
Their proofs will be given in Section \ref{section:odddim}.

\begin{rem}
	In the whole present section we will restrict ourselves to finite-volume manifolds.
	Thus $M$ is the interior of a compact manifold $\overline M$ such that 
	\[
	\partial \overline M = T_1 \cup \cdots \cup T_l,
	\]
	where each connected component $T_i$ is homeomorphic to a torus $T^2$. 
\end{rem}

\begin{proposition} \label{prop:thetatorsion}
	Let $n > 0$. For each connected boundary component $T_i$ of $M$ such that $\cohom^0(T_i; \rho_n)$ is not trivial,
	fix a non-trivial cycle $\theta_i \in \cohom_1( T_i; \integer)$.
	Then there exists a canonical family of bases for the homology groups $\cohom_*(M; \rho_n)$ 
	such that any basis of this family determines the same Reidemeister torsion, 
    which we will denote as:
    \begin{equation} \label{eq:thetatorsion}
        \tau( M; \rho_n; \{\theta_i\}).
    \end{equation}
\end{proposition}
The next proposition will allow us to define a new invariant from $\tau( M; \rho_n; \{\theta_i\})$ 
that does not depend on the choices $\{\theta_i\}$ (i.e.\; an invariant solely of the manifold).

\begin{proposition} \label{prop:normaltor}
	Let $n > 0$. For each connected boundary component $T_i$ of $M$ such that $\cohom^0(T_i; \rho_n)$ is not trivial,
	fix a non-trivial cycle $\theta_i \in \cohom_1( T_i; \integer)$. 
    Then for $k > 0$ the following quantities are independent of the cycles $\{\theta_i\}$:
	\begin{eqnarray*}
		\normaltor_{2k +1}(M, \eta) & := & \frac{\tau( M; \rho_{2k + 1}; \{ \theta_i \})}
		{\tau( M; \rho_3; \{ \theta_i \} )} \in \cmplx^*/\{\pm 1\} ,\\
		\normaltor_{2k}(M, \eta) & := & \frac{\tau( M; \rho_{2k}; \{ \theta_i \} )}
		{\tau( M; \rho_2; \{ \theta_i \})} \in \cmplx^*/\{\pm 1\}.
	\end{eqnarray*}
\end{proposition}

\begin{definition}
	Let $(M, \eta)$ be a complete spin-hyperbolic $3$-manifold of finite volume.
	For $n \geq 4$, the invariant $\normaltor_n(M, \eta)$ defined in this
	proposition will be called the \emph{normalized $n$-dimensional Reidemeister torsion} 
	of the spin-hyperbolic manifold $(M,\eta)$.
	If $n= 2k + 1$ is odd, $\normaltor_{2k +1}(M; \eta)$ is independent of $\eta$, and will be denoted by 
	$\normaltor_{2k +1}(M)$.
\end{definition}

The rest of this section is devoted to the proof of Propositions~\ref{prop:thetatorsion} and \ref{prop:normaltor}.
To that end, we will analyse the groups $\cohom_*(M; \rho_n)$.

\subsection{Cohomology of the boundary}

Let $T_j$ be a connected component of $\partial \overline M$ (recall that we are assuming that $M$ has finite volume),
and $U_j \cong T_j \times [0,\infty)$ be the corresponding cusp. 
It is well known that $T_j$ can be identified with the set of rays contained in $U_j$,
and that this endows $T_j$ with a canonical similarity structure;
in particular, $T_j$ has a canonical holomorphic structure.
Let us consider the canonical projection from $U_j$ to $T_j$ which sends a point in $U_j$ to the ray it belongs to;
denote this projection as
\[
\pi_j \colon U_j \to T_j.
\]

Let $E_n$ be the flat vector bundle over $\overline{M}$ defined by the representation $\rho_n$.
To compute $\cohom^*(T_i; \rho_n)$ we will interpret it as $\cohom^*(T_i; E_n)$, that is 
the cohomology of the de Rham complex 
\[
( \Omega^*( T_i ; E_n), d_\nabla),
\]
where $d_\nabla$ denotes the covariant differential defined by the flat connection on $E_n$.
This complex is isomorphic to the complex $( \Omega^*( \widetilde{T_i}; V_n)^{\pi_1 T_i}, d)$
of equivariant $V_n$--valued differential forms on $\widetilde{T_i}$ with the usual exterior differential. 

On the other hand, $E_n$ is a holomorphic vector bundle with respect to the holomorphic structure of $T_i$.
This yields the following canonical decomposition:
\[
\Omega^1( T_i ; E_n) = \Omega^{1,0}( T_i ; E_n) \oplus \Omega^{0,1}( T_i ; E_n),
\]
where $\Omega^{1,0}( T_i ; E_n)$ and $\Omega^{0,1}( T_i ; E_n)$ are the spaces of 
$E_n$-valued $1$-forms of type $(1,0)$ and $(0,1)$ respectively. 
Let us denote as $\cohom^{r,s}(T_i ; E_n)$ the projection of $\Omega^{r,s}( T_i ; E_n)\cap \operatorname{Ker} d$
onto $\cohom^1( T_i ; E_n)$, with $(r,s) = (0,1), (1,0)$.

\begin{proposition} \label{prop:basiscohombdry}
	Assume that $\cohom^0(T_i; E_n) \neq 0$. Then,
	\[
	\cohom^1(T_i ; E_n) = \cohom^{0,1}(T_i ; E_n) \oplus \cohom^{1,0}(T_i; E_n),
	\]
    with $\operatorname{dim}_{\mathbf{C}} \cohom^{0,1}(T_i ; E_n) = \operatorname{dim}_{\mathbf{C}} \cohom^{1,0}(T_i; E_n) = 1$.
\end{proposition}
\begin{proof}
	We can assume that for all $\gamma \in \pi_1 T_i$ we have:
	\[
	\Hol_M(\gamma) = 
	\left[ 
	\begin{pmatrix}
	      1 & a(\gamma) \\
	      0 & 1 
      \end{pmatrix} \right] \in \psl(2,\cmplx).
	\]
	This choice of the holonomy representation gives a complex coordinate $z$ on $\widetilde{T_i}$.
	Identifying $V_n$ with the space of $(n-1)$-th degree homogeneous polynomials in the variables $X$ and $Y$,
	we define the following two forms on $\Omega^1( \widetilde{T_i}; V_n)$,
	\[
	\alpha = d\bar z \otimes X^{n-1}, \quad \beta = dz \otimes (zX + Y)^{n-1}.
	\]
	Let us check that these forms are equivariant. Let $\gamma \in \pi_1 T_i$, 
	and denote by $L_{\gamma}$ the action of $\gamma$ on $\widetilde{T_i}$. 
	Notice that $L_\gamma(z) = z+ a(\gamma)$. Hence, on one hand, we have:
	\begin{align*}
		L_{\gamma}^*(\alpha) & = d( \bar z + \bar{a}(\gamma) ) \otimes X^{n-1} = \alpha,\\
		L_{\gamma}^*(\beta)  & = dz\otimes\left( (z+a(\gamma))X + Y \right)^{n-1},
	\end{align*}
	and on the other hand:
	\begin{align*}
		\rho(\gamma) \alpha & = dz \otimes ( \gamma \cdot X)^{n-1} = dz \otimes (\epsilon X)^{n-1},\\
		\rho(\gamma) \beta  & = dz \otimes \left( z \gamma \cdot X + \gamma \cdot Y \right)^{n-1} 
		= dz \otimes \left( \epsilon(z X + a(\gamma) X + Y )\right)^{n-1},
	\end{align*}
	where $\epsilon = \pm 1$ is the sign of the trace of $\gamma$ determined by
	the lift of the holonomy representation. 
	If $n$ is odd, these two forms are clearly equivariant.
	If $n$ is even, then the condition that $\cohom^0( T_i; E_n)$ is not trivial is equivalent to say
	that $\Hol_{(M,\eta)} (\sigma)$ has trace $2$ for all $\sigma \in \pi_1 T_i$; hence, $\epsilon = 1$, and
	the two forms are equivariant.
	Since $\alpha$ and $\beta$ are closed forms, they define cohomology classes in 
	$\cohom^1(T_i; E_n)$, and hence $[\alpha] \in \cohom^{0,1}(T_i; E_n)$ and $[\beta] \in \cohom^{1,0}(T_i; E_n)$.
	To conclude the proof, it remains to prove that 
	$[\alpha]$ and $[\beta]$ are linearly independent, as $\dim_\cmplx \cohom^1( T_i ; \rho_m) = 2$.
	This is equivalent to say that $[\alpha]\wedge[\beta] \in \cohom^2( T^2 ; \cmplx)$ is not zero.
	A simple computation shows that 
	\[
	\alpha \wedge \beta = \phi\left( X^{n-1}, (zX + Y)^{n-1} \right) d\bar{z} \wedge dz = d\bar{z} \wedge dz,
	\]
	where $\phi$ is the non-degenerate $\sln(2,\cmplx)$--invariant pairing of $V_n$, see Section \ref{section:sl2rep}.
	This shows that $[\alpha]\wedge[\beta]$ is not zero, and hence the two classes must be linearly independent.
\end{proof}
\begin{rem}
	Although considering the induced holomorphic structure on the tori $T_i$ may seem a little bit
	unnatural, it yields at least a \emph{canonical} decomposition 
	of the cohomology group $\cohom^1(T_i ; E_n)$, which is all we need.
\end{rem}

Next we want to characterize the image of the map induced by the inclusion
\[
i^* \colon \cohom^1( \overline{M}; E_n) \to \cohom^1( \partial \overline{M}; E_n).
\]
Although this description will not be complete, it will be enough to give bases for the homology groups
$\cohom_*(M; \rho_n)$. Before analysing the general case, let us discuss briefly the case $n = 3$. 

The representation $V_3$ is the adjoint representation of $\sln(2; \cmplx)$,
and the cohomology group $\cohom^1(M; E_3)$ has a geometrical interpretation in terms of infinitesimal 
deformations of the complete hyperbolic structure. 
The vector bundle $E_3$ is identified with the bundle of germs of Killing vector fields on $M$;
under this identification, it can be checked that the global section $X^2 \in \Omega^0(T_i; E_3)$ 
(we are using here the same notation as in the proof of Proposition~\ref{prop:basiscohombdry}) corresponds to the vector field 
$\frac{\partial}{\partial z_j}$. With this description, the $1$-form $d\bar z_j \otimes \frac{\partial}{\partial z_j}$
is a $(0,1)$-form that takes values in the vector bundle of holomorphic fields.
According to the theory of deformations of complex manifolds, this cohomology class describes the 
deformations of the holomorphic structure of $T_j$ by deformations of the defining lattice;
in particular, it gives a deformation of the Euclidean structure through Euclidean structures.
On the other hand, a non-trivial deformation of the complete hyperbolic structure is encoded 
by a cohomology class $\omega \in \cohom^1(M; E_3)$, 
and $i^*(\omega)$ encodes the corresponding deformation of the similarity structure in each torus.
Since this deformation cannot be through Euclidean structures on all tori 
(otherwise it will yield a different complete hyperbolic structure on $M$, contradicting thus the Mostow-Prasad rigidity),
then, for some
$T_j$, the restriction of $i^*(\omega)$ to $T_j$ can not be contained in $\cohom^{0,1}(T_j; E_3)$.
This shows that we have the following decomposition:
\begin{equation} \label{eq:decompcohom3}
\cohom^1(\partial \overline M; E_3) =  \operatorname{Im}i^* \oplus \bigoplus_{j=1}^k \cohom^{0,1}(T_j^2; E_3).
\end{equation}

We will prove that this decomposition holds also for $n \geq 2$. 
Since we do not have an interpretation of the cohomology group $\cohom^1(M; E_n)$ 
in geometrical terms such as deformations, 
we proceed in a different way.
Our key tool will be Theorem 2.1 of \cite{MenalPorti}, which states that a class $\omega \in \cohom^1(M; E_n)$
cannot be represented by a square-integrable form, with respect to a suitable inner product on $E_n$. 
Let us recall the definition of the inner product on $E_n$. 
Choose any $\su$--invariant inner product $\langle \cdot , \cdot \rangle$ on $V_n$
(we are considering $\su$ as a subgroup of $\sln(2,\cmplx)$). 
Identify $\hyp^3$ with $\sln(2, \cmplx)/\su$, and let $p \in \hyp^3$ be the class 
of the identity. Define an inner product on the trivial vector bundle $\hyp^3 \times V_n$ by
\[
\left\langle (q,w_1), (q, w_2) \right\rangle_q = \langle g w_1, g w_2 \rangle, \quad \text{where } g\cdot q = p.
\]
Then it induces an inner product on the vector bundle 
$E_n = \hyp^3 \times_{\pi_1(M,p)} V_n$.

\begin{lemma}
\label{lemma:l2class}
	Assume that $\cohom^0(T_j; E_n) \neq 0$.
	Then there exists a form $\alpha_j \in \Omega^{0,1}( T_j; E_n) $ 
	representing a non-trivial element in $\cohom^{0,1}( T_j; E_n) $ such that
	$\pi_j^*( \alpha_j ) \in \Omega^1( U_j; E_n)$ is $L^2$.
\end{lemma}
\begin{proof}
	Let us work in the model of the half-space $\hyp^3 = \cmplx \times (0,\infty)$.
	If $(z,t) = (x,y, t) \in \hyp^3$, the metric is given by 
	\[
	g = \frac{1}{t^2} (dx^2 + dy^2 + dt^2).
	\]
	Proceeding as in the proof of Proposition \ref{prop:basiscohombdry},
	we obtain the form  $\alpha = d\bar{z} \otimes X^{n-1}$.
	We will be done if we prove that $\pi_j^*(\alpha)$ is $L^2$.
	To compute the norm of $d\bar{z} \otimes X^{n-1}$, we may 
	assume that the cusp $U_j$ is isometric to $\cmplx \times [1, \infty) / (\Hol_M \pi_1 T^2)$.
	Thus we have:
	\[
	| d\bar{z} \otimes X^{n-1}|_{(w,t)} = |d\bar z|_{(w,t)} |X^{n-1}|_{(w,t)}.
	\]
	On one hand,
	\[
	|d\bar z|^2_{(w,t)} = |dx|^2_{(w,t)} + |dy|^2_{(w,t)} = 2 t^2.
	\]
	On the other hand, by definition of the metric of $E_n$, it can checked that
	\[
	\vert X^{n-1}\vert^2_{(w,t)}= t^{1-n} |X^{n-1}|^2,
	\]
	where $|X^{n-1}|$ is the norm of $X^{n-1}$ in $V_n$ with respect to the fixed hermitian metric.
	Therefore, if $R$ is a fundamental domain for $T^2$, we get
	\[
	\int_{U_j}  |d\bar{z} \otimes X^{n-1} |^2 d\Vol_{U_j} =
	2 |X^{n-1}|^2 \int_{R\times[1,\infty]} \frac{ t^{3 - n} }{t^3} dx dy dt = C \int_1^\infty t^{- n} dt < \infty,
	\]
	and the lemma is proved.
\end{proof}
Now we can prove that decomposition \eqref{eq:decompcohom3} holds for all $n \geq 2$.

\begin{proposition} \label{prop:generatorimage}
	Assume that $T_1,\dots,T_r$ are all the connected components of $\partial \overline M$
	such that $\cohom^0(T_j; E_n) \neq 0$. Then we have the following decomposition:
	\[
	\bigoplus_{j=1}^r \cohom^1(T_j; E_n) = \operatorname{Im}i^* \oplus \bigoplus_{j=1}^r \cohom^{0,1}(T_j; E_n).
	\]
\end{proposition}
\begin{proof}
    On one hand, by Proposition \ref{prop:basiscohombdry} we have:
    \[
        \operatorname{dim}_{\mathbf{C}} \cohom^1(T_j; E_n)  =  2, \quad 
        \operatorname{dim}_{\mathbf{C}} \cohom^{0,1}(T_j; E_n) =  1. 
    \]
    On the other hand, Theorem 0.1 of \cite{MenalPorti} implies $\operatorname{dim}_{\mathbf{C}} \operatorname{Im}i^* = r$.
    Therefore it is enough to prove that $\operatorname{Im}i^* \cap \bigoplus_{j=1}^r \cohom^{0,1}(T_j; E_n) = 0$.
	Let $[\omega] \in \cohom^1(M; E_n)$ such that $ i^*([\omega]) \in \bigoplus_{j=1}^k \cohom^{0,1}(T_j^2; E_n)$.
	Let us work with the cusps $ U_j \cong T_j \times (0, \infty)$, and assume that they are disjoint.
	Let $\alpha_j$ be the forms given by Lemma~\ref{lemma:l2class}.  
	Then 
	\[
	\omega = \lambda_j \pi_i^*(\alpha_j) + df_j, \quad \text{on } U_j,
	\]
	for some $\lambda_j \in \cmplx$ and $f_j \in \Omega^0( U_j; E_n)$. 
	Let $F \in \Omega^0( M; E_n)$ such that $F_{|T_j\times [1, \infty)} = f_j$ and vanishing outside the cusps. 
	By Lemma~\ref{lemma:l2class}, $\omega - dF$ is $L^2$, and hence the class $[\omega]$ has an $L^2$ representative.
    Finally, Theorem 2.1 of \cite{MenalPorti} implies that $[\omega] = 0$, as we wanted to prove. 
\end{proof}

\subsection{The homology groups $\cohom_*( M ; \rho_n)$ }
\label{section:odddim}

The aim of this subsection is to prove Propositions~\ref{prop:thetatorsion} and~\ref{prop:normaltor} concerning 
the existence of a distinguished family of bases for the groups $\cohom_*( M ; \rho_n)$.

We will use the following construction for the homology of a finite  CW--complex $X$ in the local system 
defined by a representation $\rho \colon \pi_1 (X, p) \to \operatorname{GL}(V)$.
Consider the right action of $\pi_1(X, p)$ on $V$, so that $\gamma \in \pi_1(X, p)$ maps $v\in V$ to $\rho(\gamma)^{-1} v$.
We will write $V_\rho$ to emphasize the fact that $V$ is a $\pi_1 (X, p)$--right module through $\rho$.
Let $C_*(\widetilde X; \integer)$ denote the complex of singular chains on the universal covering,
in which $\pi_1(X, p)$ acts on the left by deck transformations, and let 
\[
C_*( X; V_\rho) = V_\rho \otimes_{\cmplx [\pi_1(X, p)]} C_*(\widetilde X; \integer).
\]
Then $\cohom_*(X; \rho)$ is the homology of the following complex of $\cmplx$-vector spaces,
\[
\left( C_*(X; V_\rho), \operatorname{Id} \otimes \partial_* \right).
\]
We will use the Kronecker pairing between homology and cohomology with twisted coefficients.
To define it we need an invariant and non-degenerated bilinear map 
\[
\phi \colon V \times V \to \cmplx.
\]
If $X$ is a differentiable manifold, then the Kronecker pairing can be defined 
at the level of smooth chains and forms as follows:
\begin{align*}
	C_r(X; V_\rho) \times \Omega^r(\widetilde X; V_\rho )^{\pi_1 X} & \longrightarrow  \cmplx \\
	\left( v_\theta \otimes \theta, \omega \otimes v_{\omega} \right)
	& \longmapsto  \int_{\theta} \phi(v_\theta,v_{\omega}) \, \omega.
\end{align*}

The Kronecker pairing does not depend on the different choices, but on the respective classes in cohomology and homology,
and it is natural and non-degenerate.

We want to prove the following result from which Proposition \ref{prop:thetatorsion} will be easily deduced.

\begin{proposition} \label{prop:defbasodd_wj}
	Let $T_1, \dots, T_r$ be the boundary components of $\overline M$ that are \emph{not} $\rho_n$-acyclic.
	Let $G_j < \pi_1 (M, p)$ be some fixed realization of the fundamental group of $T_j$ as a subgroup of $\pi_1(M, p)$.
	For each $T_j$ choose a non-trivial cycle $\theta_j \in \cohom_1(T_j; \integer)$, 
	and a non-trivial vector $w_j \in V_n$ fixed by $\rho_n (G_j)$.
	If $i_j \colon T_j \to M$ denotes the inclusion, then we have: 
	\begin{enumerate}
		\item A basis for $\cohom_1( M; \rho_n)$ is given by
			\[
			\left( i_{1*}( [w_1 \otimes \theta_1]),\dots, i_{r*}( [w_r \otimes \theta_r])  \right).
			\]
			
		\item Let $[T_j ] \in \cohom_2(T_j ; \integer)$ be a fundamental class of $T_j$.
			A basis for $\cohom_2( M; \rho_n)$ is given by
			\[
			\left( i_{1*}( [w_1 \otimes T_1]),\dots, i_{r*}( [w_r \otimes T_r])  \right).
			\]
	\end{enumerate}
\end{proposition}

\begin{proof} 
	Let $[\alpha_j]$ and $[\beta_j]$ be generators of $\cohom^{0,1}(T_j^2; E_n)$ and $\cohom^{1,0}(T_j^2; E_n)$ respectively.
	We claim that the Kronecker pairing $([ w_j \otimes \theta_j], [\alpha_k])$ is zero for all $j,k$,
	and $([ w_j \otimes \theta_j], [\beta_k])$ is zero if and only if $j\neq k$.
	We can assume that $k = j$. Let us fix $T_j$. 
	Proceeding as in the proof of Proposition \ref{prop:basiscohombdry}, we may assume that
	$w_j = X^{n-1}$, $\alpha_j = d\bar{z} \otimes X^{n-1}$ and $\beta_j = dz \otimes (zX + Y)^{n-1}$.
	We have
	\begin{equation} \label{eqn:w0}
		( [ w_j \otimes \theta_j], [\beta_j]) = \int_{\theta_j} \phi\left( X^{n -1}, (zX + Y)^{n-1} \right) dz
		= \int_{\theta_j} \phi\left(X^{n-1}, Y^{n-1}\right) dz = \int_{\theta_j} dz\neq 0.
	\end{equation}
	On the other hand, since $\phi( X^{n-1}, X^{n-1}) = 0$,
	$( [ w_j \otimes \theta], [\alpha_j]) = 0$. This proves the claim.
	
	Let us prove now the first assertion. Assume that we have: 
	\[
	\sum_{j= 1}^r \lambda_j i_* [ w_j \otimes \theta_j] = 0, \quad \text{with } \lambda_j \in \cmplx. 
	\]
	The naturality and the non-degeneracy of the Kronecker pairing imply that this is equivalent to
	\[
	\sum_{j= 1}^r \lambda_j \left( w_j \otimes \theta_j, i^*(\omega) \right) = 0,\quad \text{for all } [\omega] \in \cohom^1(M; E_n),
	\]
	where $(\cdot,\cdot)$ denotes the Kronecker pairing.
	By Proposition \ref{prop:generatorimage}, each $\beta_j$ is uniquely written as 
	\[
	\beta_j = \gamma_j + \sum_{k = 1} ^r \mu_j^k \alpha_k, \quad \text{with } \gamma_j \in \operatorname{Im} i^* \text{ and } \mu_j^k \in \cmplx.
	\]
	Moreover, $(\gamma_1,\dots,\gamma_r)$ is a basis of $\operatorname{Im} i^*$. 
	The preceding discussion then implies $\lambda_j = 0$ for all $j$. The first assertion is thus proved.

	Let us prove Assertion 2. The long exact sequence in homology for the pair $(\overline M, \partial \overline M)$
	shows that the inclusion $\partial \overline M \subset \overline M$ yields an isomorphism
	\[
	i_* \colon \cohom_2( \partial \overline M; E_n) = \bigoplus_{j=1}^r \cohom_2( T_j; E_n) \to \cohom_2(M; E_n).
	\]
	Thus it is enough to prove that $[w_j \otimes T_j]$ is not zero.
	This can be proved using Poincar\'e duality $\operatorname{PD}$.
	Indeed, if we identify $\cohom^0(T_j ;E_n)$ with the subspace of $V_n$ of invariant vectors, then it can be checked that
	\[
	\operatorname{PD}(w_j) = [w_j \otimes T_j].
	\]
\end{proof}

\begin{proof}[Proof of Proposition~\ref{prop:thetatorsion}] 
  The choice of basis on homology provided by Proposition~\ref{prop:defbasodd_wj}
  defines a Reidemeister torsion $\tau(M;\rho_n;\{\theta_i\})$ that depends only on $M$, $\rho_n$ and $\{\theta_i\}$. 
  Notice that the choice of the fundamental class $[T_ j]\in \cohom_2(T_j ; \integer)$
  is canonical. In addition, the invariant subspace $V_n^{\rho_n (G_j)}$ is one-dimensional
  \cite{MenalPorti}, therefore different choices of $w_j\in V_n^{\rho_n (G_j)}$ differ by a multiple.  
  The torsion is homogeneous on the $w_j$, moreover each $w_j$ appears once in the basis for $H_1(M;\rho_n)$ and once in the basis for $H_2(M;\rho_n)$,
  hence the effect of replacing $w_j$ by a multiple is cancelled (see~\cite{MilnorTor,Turaev} for the behaviour of torsion under change of bases in homology).
\end{proof}

Next we compute the dependence of the vectors $i_{1*}( [w_j \otimes \theta_1])$ (introduced 
in Proposition~\ref{prop:defbasodd_wj}) on the cycles $\{\theta_j\}$.
Before doing this, let us recall the following definition.

\begin{definition}
	Let $\theta, \theta' \in \cohom_1(T_j; \integer)$ be two non-trivial cycles in a boundary component $T_j$ of $\overline M$.
	Using the natural identification between $\cohom_1(T_j; \integer) \cong \pi_1 T_j$, 
	let us assume that 
	\[ 
	\Hol_M (\theta) = 
	\left[ \begin{pmatrix}
				1 & a (\theta) \\
				0 & 1 
	\end{pmatrix} \right] \in \psl(2,\cmplx)
	\]
	for some $a(\theta) \in \cmplx^*$, $i = 1, 2$. Then define the \emph{cusp shape} of the pair $(\theta,\theta')$
	as
	\[
	\cs(\theta, \theta') = \frac{a(\theta)}{a(\theta')}.
	\]
	Notice that $\cs(\theta, \theta')$ is well defined, because $a\colon \pi_1 T_j \to \cmplx$ is unique up to
	homothety. 
\end{definition}

\begin{proposition} \label{prop:defbasodd_thetaj}
    With the same notation used in Proposition \ref{prop:defbasodd_wj}, we have:
	\[
    	i_{j*}([ w_j \otimes \theta_j] ) = \cs(\theta_j, \theta_j') i_{j*}([ w_j \otimes \theta_j']),
	\]
    where each $\theta_j'$ is any non-trivial cycle in $\cohom_1(M; \integer)$.
\end{proposition}

\begin{proof} 
	The result follows from Equation \eqref{eqn:w0} because
    \[
    \int_{\theta_j} dz = a(\theta_j).
    \]
\end{proof}

\begin{proof}[Proof of Proposition~\ref{prop:normaltor}]
  By the behaviour of the torsion under change of bases in homology \cite{MilnorTor,Turaev},
  Proposition~\ref{prop:defbasodd_thetaj} implies that the quotient $\tau(M;\rho_n;\{\theta_j\})/\tau(M;\rho_n;\{\theta_j'\})$
  is independent of $n$, and Proposition~\ref{prop:normaltor} follows.
\end{proof}

\section{Behaviour under hyperbolic Dehn filling}
\label{chapter:surgery}

The aim of this section is to analyse the behaviour of the $n$-dimensional Reidemeister 
torsion under hyperbolic Dehn surgery. Before discussing it, we need to fix some notation. 

Throughout this section $M$ will denote an oriented complete hyperbolic 
3-manifold of finite volume with $l$ cusps. 
For each connected boundary component $T_i$ of $M$ we 
fix two closed simple oriented curves $a_i, b_i$ in $T_i$ generating $\cohom_1( T_i; \integer)$. 
We define the following sets:
\begin{eqnarray*}
	\mathcal{A}               & = & \{ (p,q) = (p_1,\dots,p_l,q_1,\dots,q_l) \in \integer^{l}\times \integer^l \mid \gcd(p_i,q_i) = 1 \}, \\
	\mathcal{A}_M 		  & = & \{ (p,q) \in \mathcal{A} \mid M_{p/q} := M_{p_1/q_1,\dots,p_l/q_l} \text{ is hyperbolic } \}.\\
\end{eqnarray*}
\begin{rem}
	We may regard $\mathcal{A}$ as a directed set with respect to the following preorder:
	\[
	(p,q) \leq (p', q') \Leftrightarrow  (p_i)^2 + (q_i)^2 \leq (p_i')^2 + (q_i')^2 \text{ for all } i = 1,\dots,l.
	\]
	The hyperbolic Dehn surgery theorem by Thurston \cite{ThurstonNotes} implies that $\mathcal{A}_M$ is also a directed subset 
	of $\mathcal{A}$ , namely any two elements of $\mathcal{A}_M$ have a common greater element. 
	The limit of an $\mathcal{A}_M$-net $\{x_{p/q}\}$ in some topological space, whenever it exists, will be denoted by:
	\[
	\lim_{(p,q)\to \infty} x_{p/q}.
	\]
\end{rem}

In analysing the relation between the $n$-dimensional torsion invariants of $M$ with those of $M_{p/q}$,
some issues arise. In order to discuss them, we distinguish two cases according to the parity of $n$.

We consider first the case $n = 2k + 1$, with $k > 0$. In that case we find two difficulties.
The first one is that we need some extra data in order to define the torsion invariant for $M$ 
(we must choose non-trivial cycles $\theta_i \in \cohom_i(T_i; \integer)$),
whereas for $M_{p/q}$ this is already defined.
The second one is due to the following result proved in \cite[p.~110]{Porti}
(notice that our torsion is the inverse of the one considered in \cite{Porti}),
\[
\lim_{(p,q) \to \infty} \left| \tor_3( M_{p/q}) \right| = 0.
\]
The proof of this limit also works for any odd number $n \geq 3$.
Moreover, the asymptotic growth of these sequences does not depend on the dimension $n$.
These facts suggest that the question above should be formulated in terms of normalized torsions.
In that case, we will prove the following result.

\begin{proposition} \label{prop:oddtordeform}
	The set of cluster points of the following net in $\cmplx/\{ \pm 1\}$, 
	\[
	\left \{ \normaltor_{2k + 1} (M_{p/q}) \right\}_{ (p,q) \in \mathcal{A}_M},
	\]
	is the segment joining the origin and the point $2^{2(k - 1)l} \normaltor_{2k+1}(M)$.
\end{proposition}

Let us analyse now the even-dimensional case $n = 2k$, for $k >0$. 
In this case, the main difficulty comes from the fact that we need a spin structure to define the 
$n$-dimensional torsion invariant. Hence, we somehow need a way to relate spin structures on $M$ with those of $M_{p/q}$.
To that end, for a fixed spin structure $\eta$ on $M$, we define the following set
\[
\mathcal{A}_{M,\eta}  = \left\{ (p,q) \in \mathcal{A}_M \mid \eta \text{ can be extended to } M_{p/q} \supset M \right\}.
\]
\begin{rem}
	Notice that if $\eta$ can be extended to $M_{p/q}$ then the extension is unique
	(this follows from the fact if a spin structure on $\partial D^2$ can be extended to $D^2$, 
	then the extension is unique).
	In such case the extension will be denoted by $\eta_{p/q}$.
\end{rem}

Using Corollary \ref{coro:spinextension}, we get easily the following characterization of $\mathcal{A}_{M,\eta}$.

\begin{proposition} \label{prop:spin_ext}
	For each $T_i$ let $\epsilon_{a_i}, \epsilon_{b_i} = \pm 1$ be the sign of the trace of $\Hol_{(M,\eta)} (a_i)$ 
	and $\Hol_{(M,\eta)} (b_i)$ respectively. Then $(p,q) \in \mathcal{A}_{M,\eta}$ if and only if 
	\[
	\epsilon_{a_i}^{p_i}\epsilon_{b_i}^{q_i} = -1, \quad \text{for all } i = 1,\dots,l.
	\]
\end{proposition}

\begin{definition}
	We will say that a spin structure $\eta$ on $M$ is \emph{compactly isolated} if $\mathcal{A}_{M,\eta}$ is empty;
	otherwise, we will say that $\eta$ is \emph{compactly approximable}.
\end{definition}

As a corollary of  Proposition~\ref{prop:spin_ext}  we get the following result.
\begin{corollary}
	A spin structure $\eta$ of $M$ is compactly approximable if and only if it is acyclic.
\end{corollary}

\begin{rem}
	If $\eta$ is compactly approximable, Proposition \ref{prop:spin_ext} implies that $\mathcal{A}_{M,\eta}$
	is infinite; in particular, $\mathcal{A}_{M,\eta}$ is a directed set as well.
	The terminology introduced in the previous definition is coherent with the geometric topology of the space $\mathcal{MS}$ 
	of spin-hyperbolic $3$-manifolds, see Section \ref{chapter:cmplxlength}. 
	For instance, if $\eta$ is compactly approximable then the net of compact spin-hyperbolic manifolds
	$\{ (M_{p/q}, \eta_{p/q}) \}_{(p,q) \in \mathcal{A}_{M,\eta}}$ converges to $(M, \eta)$ in $\mathcal{MS}$.
\end{rem}

If $\eta$ is compactly approximable, then $\cohom^*(M;\rho_{2k}) = 0$ for all $k > 0$,
and hence it makes sense to consider the Reidemeister torsion $\tau(M; \rho_{2k})$.
On the other hand, for all $(p,q) \in \mathcal{A}_{M,\eta}$ we have the $2k$-dimensional canonical representation 
of the spin-hyperbolic manifold $(M_{p/q},\eta_{p/q})$:
\[
\rho_{2k}^{p/q} \colon \pi_1 M_{p/q} \to \sln(2k, \cmplx).
\]
The compactness of $M_{p/q}$ guarantees the acyclicity of this representation. 
Hence, it also makes sense to consider $\tor(M_{p/q}; \rho^{p/q}_{2k})$.
We will prove the following result.

\begin{proposition} \label{prop:eventordeform}
	Let $\eta$ be a compactly approximable (or acyclic) spin structure on $M$. 
	The set of cluster points of the following net in $\cmplx/\{ \pm 1\}$, 
	\[
	\left \{ \pm \tor(M_{p/q}; \rho^{p/q}_{2k}) \right \}_{ (p,q) \in \mathcal{A}_{M,\eta} },
	\]
	is the segment joining the origin and  $ \pm 2^{2kl} \tor(M; \rho_{2k})$.
\end{proposition}

The proof of both propositions will be based on surgery formulas for the torsion,
which will be deduced from the Mayer-Vietoris formula. 
These formulas involve the spin complex lengths of the core
geodesics added on the Dehn filling.
The results above then will follow essentially from the fact that 
the cluster point set of the imaginary part of the spin complex lengths of the added geodesics in $M_{p/q}$, 
as $(p,q)$ varies in $\mathcal{A}_{M,\eta}$, is $\real/\langle 4\pi\rangle$, see \cite{Meyer}.

The rest of this section is organized as follows. 
The first subsection is a brief account of the deformations of the holonomy representation of $M$.
The second and third subsections contain the proofs of Propositions \ref{prop:eventordeform} 
and \ref{prop:oddtordeform} respectively.

\subsection{Deformations}

Consider a family of continuous local deformations of the complete hyperbolic structure of $M$
given by:
\[
\Hol_M \colon U \times \pi_1 M \to \psl(2,\cmplx), \quad  U \subset \cmplx^l,
\]
with $U$ an open ball containing the origin, and with 
\[
\Hol_M(0, \gamma ) = \Hol_M (\gamma), \quad \text{for all } \gamma \in \pi_1 (M).
\]
The open set $U$ is usually called Thurston's slice, and is a double branched covering of a 
neighborhood of the variety of characters of $M$ around the complete hyperbolic structure. 
If we fix a boundary component $T_i$, then we can assume that
\[
\Hol_M (u, a_i) =
\left[ 
\begin{pmatrix}
	e^{u_i/2} & 1 \\ 
	0       & e^{-u_i/2}
\end{pmatrix}
\right]
,
\quad
\Hol_M (u, b_i) =
\left[
\begin{pmatrix}
	e^{v_i(u)/2} & \tau_i(u) \\ 
	0       & e^{-v_i(u)/2}
\end{pmatrix}
\right],
\]
where $v_i(u)$ and $\tau_i(u)$ are analytic functions on $u$ which are related by
\[
\sinh \frac{v_i(u)}{2}= \tau_i(u) \sinh \frac{u_i}{2}.
\]
This last equation follows by imposing that the two matrices commute.

By Thurston's hyperbolic surgery theorem \cite{ThurstonNotes} (see also \cite{Kapovich, BP}), for $(p,q)$ large enough,
the holonomy representation of the complete hyperbolic structure of 
$M_{p/q}$ is given at some value of $u$, say $u^{p/q}$. More concretely,
we have the following commutative diagram,
\[
\xymatrix{
\pi_1 (M) \ar[rd]^{\Hol_M(u_{p/q},\cdot)} \ar@{->>}[d]_{i_*^{p/q}} & \\
\pi_1 (M_{p/q}) \ar[r]_{\Hol_{M_{p/q}}}& \psl(2;\cmplx)
}
\]
where $i_*^{p/q}$ is the induced morphism on the fundamental groups by the inclusion 
\[
i^{p/q} \colon M \hookrightarrow M_{p/q}.
\]
The map $i^{p/q}_*$ is surjective with kernel the normal subgroup generated 
by the curves $\{a_i^{p_i} b_i^{q_i}\}$ (here we are identifying $\cohom_1(T_i ; \integer)$ with $\pi_1 T_i$,
and the latter group with a subgroup of $\pi_1 M$), we have the so-called Dehn filling equations
\begin{equation} \label{eqn:DehnFill}
	p_i u^{p/q}_i + q_i v_i(u^{p/q}) = 2\pi i, \quad \text{for all } i = 1,\dots,l.
\end{equation}
Moreover, we also have:
\[
\lim_{(p,q)\to \infty} u^{p/q} = 0.
\]

\subsection{Even-dimensional case} \label{sec:thm_def}

Let us retain the notation used in the previous subsection.
Fix a spin structure $\eta$ on $M$, and 
consider the lift of the whole family of representations $\Hol_M(u, \cdot)$ starting at $u = 0$ with $\Hol_{(M,\eta)}$.
By continuity, all these lifts are also group morphisms.
Thus we obtain a family of representations
\[
\Hol_{(M,\eta)} \colon U \times \pi_1 M \to \sln(2, \cmplx).
\]
The representation $\Hol_{(M,\eta)}(u_{p/q}, \cdot)$ of $\pi_1 M$ needs no longer 
yield a representation of $\pi_1 M_{p/q}$. 
We can characterize this condition in terms of spin structures. 

\begin{lemma} \label{lemma:inducedrepdeform}
	The representation $\Hol_{(M,\eta)} (u^{p/q}, \cdot )$ yields a representation of $\pi_1 M_{p/q}$ if and only if $(p,q) \in \mathcal{A}_{M, \eta}$.
\end{lemma}
\begin{proof}
	$\Hol_{(M,\eta)} (u^{p/q}, \cdot )$ yields a representation of $\pi_1 M_{p/q}$ 
	if and only if 
	\[
	\Hol_{(M,\eta)}( u_{p/q}, a_i^{p_i} b_i^{q_i}) = \text{Id} \in \sln(2,\cmplx),\quad \text{for all } i = 1,\dots,l.
	\]
	By Proposition \ref{prop:spin_ext}, $(p,q) \in \mathcal{A}_{M, \eta}$ if and only if 
	\[
	\epsilon_{a_i}^{p_i}\epsilon_{b_i}^{q_i} = -1 \quad \text{for all } i = 1,\dots,l,
	\]
	where $\epsilon_{a_i}, \epsilon_{b_i} = \pm 1$ is the sign of the trace of 
	$\Hol_{(M,\eta)} (0, a_i)$ and $\Hol_{(M,\eta)} (0, b_i)$ respectively. 
	On the other hand, for a fixed $i$ we can assume that 
	\[
	\Hol_{(M,\eta)} (u, a_i) = \epsilon_{a_i}(u) 
	\begin{pmatrix}
		e^{u_i/2} & 1 \\ 
		0       & e^{- u_i/2} 
	\end{pmatrix},
	\quad
	\Hol_M (u, b_i) = \epsilon_{b_i}(u) 
	\begin{pmatrix}
		e^{v_i(u)/2} & \tau_i(u) \\ 
		0       & e^{-v_i(u)/2}
	\end{pmatrix}.
	\]
	By continuity, for $u$ close to $0$, $\epsilon_{a_i}(u) = \epsilon_{a_i}$ and $\epsilon_{b_i}(u) = \epsilon_{b_i}$.
	Thus, Equation \eqref{eqn:DehnFill} yields:
	\[
	\left( \Hol_{(M,\eta)} (u^{p/q}, a_i) \right)^{p_i} \left( \Hol_{(M,\eta)} (u^{p/q}, b_i)\right)^{q_i} = -\epsilon_{a_i}^{p_i}\epsilon_{b_i}^{q_i} \operatorname{Id},
	\]
	The result then follows immediately.
\end{proof}

Now let $\eta$ be a \emph{compactly approximable}  spin structure on $M$.
Consider the composition of $\Hol_{(M,\eta)}(u, \cdot )$ 
with the $2k$-dimensional irreducible representation of $\sln(2, \cmplx)$ 
\[
\rho_{2k} \colon U \times  \pi_1 M  \rightarrow \sln(2k, \cmplx).
\]
Since $\eta$ is acyclic, for $u = 0$ the representation $\rho_{2k} (u, \cdot)$ is acyclic.
The following more or less well-known result then implies that $\rho_{2k}( u, \cdot )$ is also acyclic for $u$ close to $0$.

\begin{proposition} \label{prop:semicont}
	Let $X$ be a finite $CW$--complex, and consider a continuous family of representations 
	\[
	\rho \colon U \times \pi_1(X, x_0) \to \gln(n, \cmplx),
	\]
	where $U$ is some space of parameters.
	For a fixed $m\geq 0$, define the map $F\colon U \to \integer$ by $F(u) = \dim \cohom_m(X; \rho_u)$,
	where $\rho_u := \rho(u, \cdot)$.
	Then $F$ is upper semicontinuous, that is, 
	\[
	\limsup_{u \to u_0} F(u) \leq F(u_0), \quad \text{for all } u_0 \in U.
	\]
\end{proposition}
\begin{proof}
	The idea is that the rank of a matrix, viewed as a map from the space of matrices to $\integer$, is a lower semicontinuous function.
	The details are as follows.
	The homology groups $\cohom_*(X; \rho_u)$ can be defined as the homology groups of the complex
	\[
	\left( V\otimes_{\rho(u)} C_*(\widetilde X;\mathbf Z), \operatorname{Id}\otimes\partial_* \right).
	\]
	Let us fix $(w_1,\dots, w_n)$ a basis of $V$.
	Let $\{e^j_1,\dots, e^j_{i_j} \}$ be the cells of $X$ of dimension $j$, 
	and let $\{\lift{e}^j_1,\dots, \lift{e}^j_{i_j} \}$ be fixed lifts of these cells to $\widetilde X$.
	Then the set $\{ w_i\otimes \lift{e}^j_{k} \}$ gives a basis of $V\otimes_{\rho_u} C_j(\widetilde X;\mathbf Z)$.
	With respect to these bases, the boundary map $\partial_j(u)$ is written as a matrix $A_j(u)$
	whose entries depend continuously on $u$.
	Then we have 
	\[
	F(u) = \dim \operatorname{Ker} A_j(u) - \operatorname{rank} A_{j+1}(u).
	\]
	Since the rank of a matrix is lower semicontinuous, the dimension of the kernel is upper-semicontinuous,
	and hence $F(u)$ is upper semicontinuous.
\end{proof}

\begin{rem}
	Proposition~\ref{prop:semicont} is a special case of the semicontinuity theorem, cf.~ \cite{Hartshorne},
	which establishes 
	the upper semicontinuity of the dimension function of some cohomology groups in a more general context.
\end{rem}

Let us put $\rho_{2k}(u) := \rho_{2k}(u, \cdot)$.
The  proposition above shows that it makes sense to consider $\tor(M; \rho_{2k}(u))$ for $u$ close to $0$.
On the other hand, for $(p,q)\in \mathcal{A}_{M,\eta}$ large enough, Lemma \ref{lemma:inducedrepdeform}
implies that we have the following commutative diagram:
\[
\xymatrix{
\pi_1 (M) \ar[rd]^{\rho_{2k}(u_{p/q})} \ar@{->>}[d]_{i_*^{p/q}} & \\
\pi_1 (M_{p/q}) \ar[r]^{\rho_{2k}^{p/q}} & \sln(2k;\cmplx)
}
\]
Since $M_{p/q}$ is compact, the representation $\rho_{2k}^{p/q}$ is acyclic.
Therefore, it also makes sense to consider $\tor(M_{p/q}; \rho_{2k}^{p/q} )$.
The following lemma gives the relationship between these two quantities.

\begin{lemma} \label{lemma:evensurgeryform}
	Let $\gamma_1, \dots, \gamma_l$ be the core geodesics added on the $(p,q)$-Dehn filling $M_{p/q}$, 
	and $\lambda_{p/q}$ be the spin-complex-length function with respect to the spin-hyperbolic structure $\eta_{p/q}$.
	Then we have 
	\[
	\tor(M_{p/q}; \rho_{2k}^{p/q} ) = \pm \tor(M; \rho_{2k}(u_{p/q}) ) 
	\prod_{j=0}^{k-1} \prod_{i = 1}^l \left(e^{(\frac{1}{2} + j)\lambda_{p/q}(\gamma_i)} - 1\right)
	\left(e^{-(\frac{1}{2} + j)\lambda_{p/q}(\gamma_i)} - 1\right).
	\]
\end{lemma}
\begin{proof}
	By induction, we can assume that $M$ has only one cusp. 
	We will apply the Mayer-Vietoris sequence to the decomposition $M_{p/q} = M\cup N(\gamma)$,
	where $N(\gamma)$ is a tubular neighbourhood of the core geodesic $\gamma$ added on the Dehn filling.
	We must show first that all the involved spaces are $\rho_{2k}^{p/q}$-acyclic.
	We already know it for $M$. Since $\Hol_{M_{p/q}}(\gamma)$ has no fixed vector other than $0$, 
	$\cohom^0(\gamma; \rho_{2k}^{p/q})$ is trivial, and hence so is $\cohom^1(\gamma; \rho_{2k}^{p/q})$;
	this proves that $N(\gamma)\simeq \gamma$ is acyclic. The same argument shows that 
	$\cohom^r(\partial N(\gamma); \rho_{2k}^{p/q})$ is trivial for $r = 0, 2$, 
	which implies (Euler characteristic argument) that this holds for $r = 1$ as well.
	The Mayer-Vietoris sequence then yields the surgery formula (cf.~\cite[Thm.~3.2]{MilnorTor}, \cite[Ch.~VIII]{TuraevPM} or \cite[Prop.~0.11]{Porti}) 
	\[
	\tor(M_{p/q}; \rho_{2k}^{p/q} ) \tor( \partial N(\gamma); \rho_{2k}^{p/q} ) 
	= \tor(M; \rho_{2k}^{p/q}) \tor( \gamma; \rho_{2k}^{p/q}).
	\]
	The torsion of the torus $\partial N(\gamma)$ is $\pm 1$, as it is the Reidemeister torsion 
	of an even-dimensional manifold, see \cite{MilD}. Finally, $\tor( \gamma; \rho_{2k}^{p/q})$
	is the determinant of $\rho_{2k}^{p/q}(\gamma)- \operatorname{Id}$ (cf.~\cite[Lemma~1.3.3]{TuraevKT}), 
	and since $\rho_{2k}^{p/q}(\gamma)$ is the $(2k -1)$-symmetric power of  the holonomy of $\gamma$, that has eigenvalues $e^{\pm  \lambda_{p/q}(\gamma)/2}$, we get
	\[ 
	\tor( \gamma; \rho_{2k}^{p/q}) = 
	\prod_{j=0}^{k-1} \left(e^{(\frac{1}{2} + j)\lambda_{p/q}(\gamma)} - 1\right)\left(e^{-(\frac{1}{2} + j)\lambda_{p/q}(\gamma)} - 1\right).
	\]
\end{proof}

Now we can prove Proposition \ref{prop:eventordeform}.
\begin{proof}[Proof of Propostion \ref{prop:eventordeform} ]
	The formula of Lemma \ref{lemma:evensurgeryform} can be written as  
	\[
	\frac{ \tor(M_{p/q}; \rho_{2k}^{p/q} ) } {\tor(M; \rho_{2k}(u_{p/q}) ) } = 
	2^{2kl} \prod_{j=0}^{k-1} \prod_{i= 1}^l \frac{1 - \cosh\left( (\frac{1}{2} + j)\lambda_{p/q}(\gamma_i)\right)}{2}.
	\]
	Since $\tor(M; \rho_{2k}(u_{p/q}) )$ converges to $\tor(M; \rho_{2k})$ as $(p,q)$ goes to infinity 
	(this can be proved in the same way as Proposition \ref{prop:semicont}),
	to prove the result we may restrict our attention to the product of the right hand side of the equation above.
	Consider the map defined by
	\begin{eqnarray*}
		F \colon  [0, \infty) \times [0, 4\pi] & \longrightarrow & \cmplx \\
		(t, \theta) & \longmapsto & \prod_{j=0}^{k-1} \frac{1 - \cosh\left( (\frac{1}{2} + j)(t + \theta\operatorname{i})\right) }{2}.
	\end{eqnarray*}
	The image of $\{0\}\times [0, 4\pi]$ under $F$ is $[0,1]$, since $F(\{0\}\times [0, 4\pi]) \subset [0,1]$,
	$F(0, 0) = 0$ and $F(0, 2\pi) = 1$. 
	The result then follows from the fact that the cluster point set of the following net: 
	\[
	(\lambda_{p/q}(\gamma_1),\dots,\lambda_{p/q}(\gamma_l))_{ (p,q)\in \mathcal{A}_{M,\eta} }
	\]
	is $A^l$, with $A = \{ z \in \cmplx \mid \Re(z) = 0,\quad 0 \leq \Im(z) \leq 4\pi \}$, see \cite{Meyer}.
\end{proof}

\subsection{Odd-dimensional case}

We will use the same notation as in the previous subsections.
Throughout this subsection we will assume that $n = 2k + 1$ and $k > 0$.

\begin{lemma}
\label{lemma:invariant}
Let $T_j$ be a fixed boundary component of $\partial \overline M$. 
Assume that 
\[
\Hol_M (u, a_j) = 
\left[ 
\begin{pmatrix}
	e^{u_j/2} & 1 \\ 
	0       & e^{-u_j/2}
\end{pmatrix}
\right]
,
\quad
\Hol_M (u, b_j) =
\left[
\begin{pmatrix}
	e^{v_j(u)/2} & \tau_j(u) \\ 
	0       & e^{-v_j(u)/2}
\end{pmatrix}
\right],
\]
where $a_j, b_j$ are generators of the fundamental group of $T_j$.
For $u =(u_1,\dots,u_l) \in U \subset \cmplx^l$ such that $u_j\neq 0$, consider the following vector
\[
w_j(u) := X^k  \left( X - 2 \sinh\frac{u_j}{2} Y \right)^k \in V_{2k+1} \cong S_{2k}[X,Y],
\]
where $S_n[X,Y]$ is the space of homogeneous polynomials of degree $n$ in the variables $X,Y$.
Then, for $u$ close to $0$ with $u_j \neq 0$, the vector $w_j(u)$ is $\rho_{2k+1}(u)$-invariant.
Moreover, the map
\begin{eqnarray*}
	\Omega^*( T_j; \cmplx) & \to     & \Omega^*(T_j; E_{\rho_{2k+1}(u)} ) \\
	\omega                       & \mapsto & \omega \otimes w_j(u)
\end{eqnarray*}
induces isomorphisms in de Rham cohomology.
\end{lemma}
\begin{proof}
	Let $\Hol_{(M,\eta)}$ be a lift of the holonomy representation.
	The matrices $\Hol_{(M,\eta)}(a_j)$ and $\Hol_{(M,\eta)}(b_j)$ diagonalize and commute; hence, there 
	exists a basis $(e_1, e_2)$ of $\cmplx^2$ that simultaneously diagonalize them. 
	It can be checked that we can take \[
	 e_1 = X, \qquad e_2 = X - 2 \sinh\frac{u_j}{2} Y.
	\]
	The vector $e_1^k e_2^k \in V_{2k+1}$ is then independent of the chosen lift and 
	invariant by both $\Hol_{(M,\eta)}(a_j)$ and $\Hol_{(M,\eta)}(b_j)$. 
	This shows that $w_j(u)$ is $\rho_{2k+1}(u)$-invariant, and the first part of the lemma is proved.

	For the second part, notice that the vector $w_j(u)$ 
	gives a parallel nowhere-vanishing section of the flat vector bundle
	$E_{\rho_{2k+1}(u)}$. On the other hand, the $\sln(2, \cmplx)$-invariant pairing
	(see Section \ref{section:sl2rep})
	\[
	\phi \colon V_n \times V_n \to \cmplx, 
	\]
	defines a non-degenerate symmetric bilinear form on $E_{\rho_{2k+1}(u)}$. We have, 
	\[
	\phi\left( w_j(u), w_j(u) \right) = 2 \left( -2\sinh\frac{u_j}{2}\right)^{k}.
	\]
	Therefore, for $\sinh\frac{u_j}{2}\neq 0$, we have a decomposition 
	${E_{\rho_{2k+1}(u)}}_{|T_j} = L \oplus L^\perp$, 
	where $L$ is the line bundle defined by 
	$w_j(u)$, and $L^\perp$ is the orthogonal complement with respect to $\phi$.
	Note that both sub-bundles are flat, so we have
	\[
	\cohom^*( T_j; E_{\rho_{2k+1}(u)} ) = \cohom^*( T_j; L) \oplus \cohom^*( T_j; L^\perp).
	\]
	The line bundle $L$ is trivialized using the section $w_j(u)$. Therefore, tensorization by $w_j(u)$
    yields an isomorphism $\cohom^0( T_j; L) \cong  \cohom^0( T_j; \cmplx) \cong \mathbf{C}$.
    This shows that $\cohom^0( T_j; L^\perp)$ is trivial, 
    for $\cohom^*( T_j; E_{\rho_{2k+1}(u)} ) \cong \mathbf{C}$ 
    (this can be deduced from the upper semicontinuity of cohomology).
    Thus we have proved the last assertion of the lemma in degree $0$.
	The lemma then follows by Poincar\'e duality and an Euler characteristic argument.
\end{proof}

\begin{proposition}
	There exists a neighbourhood of the origin $W \subset U$  such that for all $u \in W$,
	\[
		\dim_\cmplx \cohom_1( M; \rho_{2k + 1}(u)) = \dim_\cmplx \cohom_2( M; \rho_{2k + 1}(u)) = l,
	\]
	where $l$ is the number of connected components of $\partial \overline M$.
\end{proposition}
\begin{proof}
	By Poincar\'e duality and an Euler characteristic argument, we deduce that
	\[
	\dim_\cmplx \cohom_1( \overline M; \rho_{2k + 1}(u)) = \dim_\cmplx \cohom_1( \overline M, \partial \overline M; \rho_{2k + 1}(u)).
	\]
	The long exact sequence of the pair $(\overline M, \partial \overline M)$ yields the following short exact sequence,
	\[
	\cohom_1( \overline M, \partial \overline M; \rho_{2k + 1}(u)) \to \cohom_0( \partial \overline M; \rho_{2k + 1}(u)) \to 0.
	\]
	Therefore,
	\[
	\dim_\cmplx \cohom_1( \overline M; \rho_{2k + 1}(u)) \geq \dim_\cmplx \cohom_0( \partial \overline M; \rho_{2k + 1}(u)) =
	\sum_{j= 1}^l \dim_\cmplx \cohom_0( T_j; \rho_{2k + 1}(u)).
	\]
	The vector space $\cohom_0( T_j; \rho_{2k + 1}(u))$ has dimension $1$.
	Indeed, if $u_j = 0$ this is clear by a direct inspection, 
	and if $u_j \neq 0$, this follows from Lemma \ref{lemma:invariant}.
	Hence,
	\[
	\dim_\cmplx \cohom_1( M; \rho_{2k + 1}(u)) \geq l, \quad \text{for } u \in U.
	\]
	Since $\dim_\cmplx \cohom_1( M; \rho_{2k + 1}(0)) = l$, the upper semicontinuity of the dimension function
	(Proposition \ref{prop:semicont}) implies the result. 
\end{proof}

\begin{proposition} \label{prop:odddimbasesparam}
	Let $\{\theta_j\}$ be a collection of nontrivial cycles with $\theta_j \in \cohom_1(T_j; \integer)$.
	Then there exists a neighbourhood of the origin $W \subset U$ such that for all $u \in W$ the following assertions hold:
	\begin{enumerate}
		\item A basis of $\cohom_1(M; \rho_{2k + 1}(u) )$ is given by
			\[
			\left( i_* [ w_1(u) \otimes \theta_1], \dots, i_* [ w_l(u) \otimes \theta_l] \right).
			\]
		\item A basis of $\cohom_2(M; \rho_{2k + 1}(u) )$ is given by
			\[
			\left( i_* [ w_1(u) \otimes T_1 ], \dots, i_* [ w_l(u) \otimes T_l] \right).
			\]
	\end{enumerate}
	In both cases, the vectors $w_j(u)$ are the ones given by Lemma \ref{lemma:invariant},
	$[T_j] \in \cohom_2( T_j; \integer)$ is a fundamental class of $\partial M$,
	and $i_*$ is the map induced in homology by the inclusion $i\colon \partial \overline M \to \overline M$.
\end{proposition}
\begin{proof}
	Proposition \ref{prop:defbasodd_wj} shows that the two assertions are true for $u = 0$.
	The result then follows proceeding as in the proof of Proposition \ref{prop:semicont}.
\end{proof}

It makes sense therefore to consider $\tau( M; \rho_{2k + 1}(u_{p/q}); \{ \theta_j \}) $, 
the Reidemeister torsion of $M$ with respect to the representation $\rho_{2k + 1}(u_{p/q})$ and the bases in homology
associated to the family of non-trivial cycles $\{\theta_j\}$ given by the Proposition \ref{prop:odddimbasesparam}.
We want to get a surgery formula for $\tau( M; \rho_{2k + 1}(u_{p/q}); \{ \theta_j \})$.
It turns out that it is easier to work with the bases given by the following lemma.
\begin{lemma} \label{lemma:oddsurgeryform}
	For sufficiently large $(p, q)$, a basis of $\cohom_1(M; \rho_{2k + 1}(u_{p/q}) )$ is given by,
		\begin{equation} \label{eqn:basisodd} 
	\left( i_*^{p/q} [ w_1(u_{p/q}) \otimes (p_1 a_1 +q_1 b_1)], \dots, i_*^{p/q} [ w_l(u_{p/q}) \otimes (p_l a_l +q_l b_l)] \right),
	\end{equation}
	where $i_*^{p/q}$ is the map induced is the inclusion $i\colon M \to M_{p/q}$.
\end{lemma}
Notice that this lemma is not a consequence of the previous proposition, because the open set $U$ of Proposition~\ref{prop:odddimbasesparam}
depends on the $\theta_j$, that here are replaced by $p_j a_j +q_j b_j$, which are not constant.
\begin{proof}
	This is a Mayer-Vietoris argument as in Lemma \ref{lemma:evensurgeryform}.
	We have the decomposition
	\[
	M_{p/q} = M \cup N, \quad \text{with } N  = \bigcup_{j= 1}^l N(\gamma_j),
	\]
	where $\{N(\gamma_j)\}$ is a collection of disjoint tubular neighbourhoods of the core geodesics $\gamma_j$ added in the Dehn filling.
	By compactness, $M_{p/q}$ is $\rho_n(u_{p/q})$-acyclic. The Mayer-Vietoris exact sequence then gives an isomorphism
	\[
	\cohom_*( \partial M; \rho_n(u_{p/q})) \cong \cohom_*(M;\rho_n(u_{p/q})) \oplus \cohom_*( N ; \rho_n(u_{p/q})).
	\]
	The group $\cohom_*( T_j ; \cmplx)$ is isomorphic to $\cohom_*( T_j; \rho_n(u_{p/q}) )$ via 
	tensorization $w_j(u_{p/q}) \otimes - $ (this is the homological
	counterpart of Lemma \ref{lemma:invariant}).
	The same isomorphism also holds true for $N(\gamma_j) \simeq \gamma_j$.
	Since $[p_j a_j + q_j b_j] \in \cohom_1( N(\gamma_j); \integer)$ is zero by construction, the vectors described
	in \eqref{eqn:basisodd} must be linearly independent.
\end{proof}

The surgery formula is now easily obtained.

\begin{lemma} \label{eqn:torodd}
	Let $\gamma_1, \dots, \gamma_l$ be the core geodesics added on the $(p,q)$-Dehn filling $M_{p/q}$, 
	and $\lambda_{p/q}$ be the complex-length function of $M_{p/q}$.
	Then we have 
	\begin{equation*}
		\tau(M_{p/q}; \rho_{2k+1}^{p/q}) = 
		\tau(M; \rho_{2k+1}(u_{p/q}), \{p_j a_j + q_j b_j\} ) 
		\prod_{j = 1}^k \prod_{i=1}^l ( e^{j\lambda_{p/q}(\gamma_i)}-1 ) (e^{-j \lambda_{p/q}(\gamma_i)}-1 ).
	\end{equation*}
\end{lemma}
\begin{proof}
    This is again a Mayer-Vietoris argument. We use the same notation as in Lemma \ref{lemma:oddsurgeryform} and its proof.
    We have	$M_{p/q} = \overline M \cup N$. The formula for the torsion is
	\[
	\tau(M_{p/q}; \rho_{2k+1}^{p/q}) \tau ( \partial \overline M; \rho_{2k +1}^{p/q} ) = 
	\tau( \overline M; \rho_{2k+1}^{p/q},  \{p_ja_j + q_j b_j\} ) \tau ( N; \rho_{2k +1}^{p/q}) \tau( \mathcal{H}_*),
	\]
	where $\tau( \mathcal{H}_*)$ is the torsion of the Mayer-Vietoris complex computed using 
	the bases that has been chosen to compute the involved torsions in the decomposition  (cf.~\cite[Thm.~3.2]{MilnorTor}, \cite[Ch.~VIII]{TuraevPM} or \cite[Prop.~0.11]{Porti}).
	To compute the torsions we choose bases in homology as follows.
	For $ \cohom_*(T_j; \rho_{2k +1}^{p/q} )$,
	we take in degree $0$, $[w_j(u_{p/q}) \otimes \sigma_j]$ (recall that we are using the notation of Lemma \ref{lemma:oddsurgeryform}), 
    where $\sigma_j$ is a generator of $\cohom_0(T_j; \integer)$,
	in degree $1$, $[w_j(u_{p/q}) \otimes (p_j a_j + q_j b_j)]$, and in degree $2$, $[w_j(u_{p/q}) \otimes T_j]$.
	For $ \cohom_*( N(\gamma_j); \rho_{2k +1}^{p/q} )$, we take in degree $0$, $[w_j(u_{p/q}) \otimes i_{2,*}(\sigma_j)]$,
	and in degree $1$, $[w_j(u_{p/q}) \otimes i_{2,*}(\gamma)]$,
	where $i_{2,*}$ is the map induced by the inclusion $i_2 \colon \partial M = \partial N \to N$,
	and $\gamma \in \cohom_1(\partial \overline M; \integer)$ is such that 
	$i_{1,*}(\gamma) \in \cohom_1(M; \rho_{2k + 1}(u_{p/q}) )$ is zero 
	(notice that such a curve always exists and $i_{2,*}(\gamma) \in \cohom^1( D^2 \times S^1; \integer )$ is homologous to the core geodesic).
	With respect to these bases, we have $\tau( \mathcal{H}_*) = 1$, since the isomorphism $i_{1,*} + i_{2,*}$ appearing in the Mayer-Vietoris
	sequence is represented by the identity matrix. On the other hand, the torsion of 
	$\partial \overline M$ is $\pm1$, as it is an even-dimensional manifold.
	Thus we have
	\[
	\tau(M_{p/q}; \rho_{2k+1}^{p/q}) = \tau(M; \rho_{2k+1}^{p/q},  \{p_j a_j + q_j b_j\}) \prod_{j=1}^l \tau ( \gamma_j; \rho_{2k +1}^{p/q}).
	\] 
	Finally, a computation as in Lemma~\ref{lemma:evensurgeryform} gives
	\[
	\tau( \gamma_j; \rho_{2k+1}^{p/q}) = \prod_{h=1}^k \left( e^{h\lambda_{p/q}(\gamma_j)}-1 \right) \left( e^{-h \lambda_{p/q}(\gamma_j)}-1 \right).
	\]
\end{proof}

Let us normalize torsions in the formula of Lemma \ref{eqn:torodd}. 
Thus we get:
\[ 
\normaltor_{2k +1}(M_{p/q}) = 
\frac{\tau(M; \rho_{2k+1}(u_{p/q}), \{p_j a_j + q_j b_j\} )}{\tau(M; \rho_3(u_{p/q}), \{p_j a_j + q_j b_j\} )}
\prod_{j = 2}^k \prod_{i=1}^l \left( e^{j\lambda_{p/q}(\gamma_i)}-1 \right) \left( e^{-j \lambda_{p/q}(\gamma_i)}-1 \right).
\]
Let us focus on the quotient of torsions appearing in the right hand side of this equation.
We shall write down a formula relating the torsion of $M$ with respect to the basis $\{ a_j \}$ and 
$\{ p_j a_j + q_j b_j \}$.
To that end, let $A_{2k+1}(p,q)$ be the change of basis matrix from the basis $\{ [w_j(u_{p/q})\otimes a_j] \}$
to $\{ [w_j(u_{p/q})\otimes (p_j a_j + q_j b_j)] \}$. Then the change of basis formula for the torsion yields:
\[
\tau( M; \rho_{2k + 1}(u_{p/q}), \{ p_j a_j + q_j b_j \}) \det A_{2k+1}(p,q) = \tau( M; \rho_{2k + 1}(u_{p/q}), \{ a_j \}) . 
\]
This equation implies
\begin{equation} \label{equation:normaltor}
\frac{\tau( M; \rho_{2k + 1}(u_{p/q}), \{ p_j a_j + q_j b_j \})}{\tau( M; \rho_3(u_{p/q}), \{ p_j a_j + q_j b_j \})} 
= \frac{\tau( M; \rho_{2k + 1}(u_{p/q}), \{ a_j \})}{\tau(M; \rho_3(u_{p/q}), \{ a_j \})} \frac{\det A_3(p,q)}{ \det A_{2k+1}(p,q)}.
\end{equation}
On one hand, working as in in Proposition \ref{prop:semicont}, it can be checked that
\[
\lim_{u\to 0} \tau(M; \rho_{2k+1}(u), \{a_j\} ) = \tau(M; \rho_{2k+1}(0), \{a_j\} ) = \tau(M; \rho_{2k+1}, \{a_j\} ).
\]
Hence, 
\begin{equation} \label{eq:toraj}
	\lim_{(p,q) \to \infty} \frac{\tau(M; \rho_{2k+1}(u_{p/q}), \{a_j\} )}{\tau(M; \rho_3(u_{p/q}); \{ a_j\})} = \normaltor_{2k +1}(M).
\end{equation}
On the other hand, we have the following result.
\begin{lemma} \label{lemma:limApq}
	For any $k \geq 3$,
	\[
	\lim_{(p,q) \to \infty} \frac{\det A_{2k + 1}(p,q) }{\det A_3(p,q) } = 1.
	\]
\end{lemma}
\begin{proof}
	We have 
	\[
	A_{2k + 1}(p,q) = \operatorname{diag}(p) + \operatorname{diag}(q) B_{2k + 1}(u_{p/q}),
	\]
	where $B_{2k +1}(u)$ is the change of basis matrix from the basis $\{ [w_j(u)\otimes a_j] \}$
	to the basis $\{ [w_j(u)\otimes b_j] \}$. Working as in Proposition \ref{prop:semicont}, it can be checked that 
	$B_{2k +1}(u)$ depends analytically on $u$. Note that at $u = 0 $ we have
	\[
	B_{2k +1}(0) = \operatorname{diag}(\cs(b_1,a_1),\dots,\cs(b_l,a_l)).
	\]
	Let us write $P = \operatorname{diag}(p)$, $Q = \operatorname{diag}(q)$ and $C = B_{2k +1}(0)$.
	Notice that $C$ is independent of $k$. The lemma will follow easily once we have proved the following 
	equality: 
	\[
	\lim_{(p,q) \to \infty} \frac{\det (P  + Q C) } { \det (P + Q B_{2k +1}(u_{p/q})) } = 1.
	\]
	We have
	\[
	\frac{\det (P  + Q C) } { \det (P + Q B_{2k +1}(u_{p/q}) ) } = 
	\frac{\det (Q^{-1} P  + C) } { \det (Q^{-1} P + B_{2k +1}(u_{p/q})) }.
	\]
	Let us put $D = Q^{-1} P  + C$ and $ E(u_{p/q}) = B_{2k +1}(u_{p/q}) - C$. Then we have
	\begin{eqnarray*}
		\frac{\det (P  + Q C) } { \det (P + Q B_{2k +1}(u_{p/q})) } & = & \frac{\det D } { \det ( D + E_{2k +1}(u_{p/q})) } \\
		& = & \frac{1} { \det ( \operatorname{Id} + D^{-1} E_{2k +1}(u_{p/q}) )}.
	\end{eqnarray*}
	If $D = (d_{ij})$ then we have
	\[
	|d_{jj}| = |p_j/q_j + \cs(a_j, b_j)| > |\Im \cs(a_j, b_j)| > 0. 
	\]
	Therefore, the entries of the diagonal matrix $D^{-1}$ are bounded, and hence 
	\[
	\lim_{(p,q) \to \infty} D^{-1} E_{2k +1}(u_{p/q}) = \lim_{(p,q) \to \infty} D^{-1} (B_{2k +1}(u_{p/q}) - B_{2k +1}(0) ) = 0.
	\]
\end{proof}
Finally, taking limits in Equation \eqref{equation:normaltor}, 
and using Equation \eqref{eq:toraj} and Lemma \ref{lemma:limApq}, we get:
\[ 
\lim_{(p,q) \to \infty} \frac{\tau(M; \rho_{2k+1}(u_{p/q}), \{p_ja_j + q_jb_j\} )}{\tau(M; \rho_3(u_{p/q}); \{ p_ja_j + q_jb_j\})} = 
\normaltor_{2k +1}(M). 
\]
Just for future references, we summarize the preceding results in the following lemma.
\begin{lemma} \label{lemma:odddimnormaltor}
	With the notation above, for $k > 1$  we have
	\[ 
	\normaltor_{2k + 1}(M_{p/q} ) = \frac{\det A_3(p,q) }{\det A_{2k +1}(p,q) }  
	\frac{\tau(M; \rho_{2k+1}(u_{p/q}), \{a_j \} )}{\tau(M; \rho_3(u_{p/q}), \{a_j \} )}
	 \prod_{j=2}^k \prod_{i=1}^l ( e^{j\lambda_{p/q}(\gamma_i)}-1 ) ( e^{-j \lambda_{p/q}(\gamma_i)}-1 ).
	\]
	Moreover, 
	\begin{eqnarray*}
		\lim_{(p,q) \to \infty} \frac{\det A_{2k + 1}(p,q) }{\det A_3(p,q) } & = & 1, \\
		\lim_{(p,q) \to \infty} \frac{\tau(M; \rho_{2k+1}(u_{p/q}), \{a_j \} )}{\tau(M; \rho_3(u_{p/q}), \{a_j \} )}  & = & 
		\normaltor_{2k +1}(M) . 
	\end{eqnarray*}
	
\end{lemma}

\begin{proof}[Proof of Proposition \ref{prop:oddtordeform}]
	By Lemma \ref{lemma:odddimnormaltor}, the result is reduced to prove that 
	the set of cluster points of the following net 
	\[
	\left\{ 
	 \prod_{j=2}^k \prod_{i=1}^l (e^{j\lambda_{p/q}(\gamma_i)}-1) (e^{-j \lambda_{p/q}(\gamma_i)}-1) \right\}_{ (p,q) \in \mathcal{A}_M} 
	\]
	is $[0, 4^{(k-1)l}]$, which may be proved in the same way as in the even-dimensional case
	(see Proposition \ref{prop:eventordeform}).
\end{proof}

\section{Complex-length spectrum}
\label{chapter:cmplxlength}

The aim of this section is to prove the continuity of the complex-length spectrum 
in a sense that we shall precise.

\subsection{Closed geodesics in a hyperbolic manifold}
Although the material of this subsection is well known, we think it is worth reviewing it
for the sake of completeness.

Let $M$ be an oriented, complete, hyperbolic $3$-manifold, 
and $\Hol_M$ be its holonomy representation.
Let us consider $\mathcal{C}(M)$ the set of closed (constant-speed) geodesics in $M$ up to orientation-preserving reparametrisation.
We will describe $\mathcal{C}(M)$ as the following quotient set,
\[
\mathcal{C}(M)  =  \left\{ \varphi \colon S^1 \to M \mid \varphi \text{ is a geodesic} \right\} / S^1.
\]
The action of $S^1$ on a closed geodesic is given by translation on the parameter.
We are interpreting $S^1$ as $\real/\integer$. 
If $k \in \integer$ and $\varphi \colon S^1 \to M$ is a closed geodesic, 
$k \varphi$ will denote the closed geodesic $t \mapsto \varphi(kt)$.

\begin{definition}
	A closed geodesic $\varphi$ is said to be \emph{prime} if $\varphi \neq k \psi$
	for any $k > 1$ and any closed geodesic $\psi$ 
	(\idest $\varphi$ is prime if it traces its image exactly once).
	A class $[\varphi] \in \mathcal{C}(M)$ is said to be prime if $\varphi$ is prime.
	The set of prime classes of $\mathcal{C}(M)$ will be denoted by $\mathcal{PC}(M)$.
\end{definition}

We will also need the group theoretic definition of primality.
\begin{definition}
	Let $G$ be a group. An element $g \in G$ is said to be \emph{prime} if $g \neq h^k$ 
	for all $h\in G$ and $k>1$ (note that we are excluding the identity from this definition).
	If $\operatorname{C}(G)$ denotes the set of conjugacy classes of $G$, then
	$[g] \in \operatorname{C}(G)$ is said to be prime if $g$ is prime.
\end{definition}

The identification between the set $\operatorname{C} \left( \pi_1(M, p) \right)$
of conjugacy classes of $\pi_1 (M, p)$ and loops in $M$ up to free homotopy
yields a natural map 
\[
\psi \colon \mathcal{C}(M) \to \operatorname{C}\left(\pi_1(M, p)\right).
\]
Let $\operatorname{HypC}(\pi_1 (M, p))$ be the set of hyperbolic conjugacy 
classes of $\pi_1 (M, p )$, that is,
\[
\operatorname{HypC}(\pi_1 (M, p)) = \left\{ [\gamma] \in \operatorname{C}\left(\pi_1(M, p)\right) 
\mid \Hol_M(\gamma)\text{ is of hyperbolic type} \right\}.
\]

The following result is well known, and is easily deduced from the fact that an isometry of $\hyp^3$ 
of hyperbolic type has exactly one axis.

\begin{proposition} \label{prop:bij_geo_conj}
	The natural map $\psi \colon \mathcal{C}(M) \to \operatorname{C}( \pi_1 (M, p ) )$ is a bijection
	onto the set $\operatorname{HypC}(\pi_1 (M, p))$. 
	Moreover, $[\varphi] \in \mathcal{C}(M)$ is prime if and only if so is $\psi([\varphi])$.
\end{proposition}

It is useful to endow the set $\mathcal{C}(M)$ with the (quotient) supremum metric.
More explicitly, if $[\varphi_1], [\varphi_2] \in \mathcal{C}(M)$ then its distance is defined by 
\[
d([\varphi_1], [\varphi_2]) = \min_{s\in S^1} \max_{t \in S^1} \left\{ d(\varphi_1(t+s),\varphi_2(t)) \right\}.
\]
The following observation is an immediate consequence of the previous proposition,
and will be used quite often in the subsequent subsections.

\begin{proposition}
	\label{prop:dist_bb}
	Let $[\varphi], [\varphi']$ be two distinct elements of $\mathcal{C}(M)$,
	and let $m$ be the minimum of the injectivity radius at $\varphi$.
	Then $d([\varphi], [\varphi'] ) \geq m$.
\end{proposition}
\begin{proof}
	Assume that $d([\varphi], [\varphi'] ) = m' < m$.
	With suitable parametrisations, we have that for all
	$t\in S^1$, $d(\varphi (t), \varphi' (t) ) \leq m'$.
	By the hypothesis on the injectivity radius, there exists a unique minimizing geodesic joining 
	$\varphi (t)$ and $\varphi'(t)$. Therefore, we can define a free homotopy 
	from $\varphi$ to $\varphi'$, which contradicts Proposition \ref{prop:bij_geo_conj}.
\end{proof}

This subsection ends with an estimate on the growth of the number of closed 
geodesics in function of their length. The following estimate, 
though not the best possible (see for instance \cite{Margulis}, \cite{CoornKnieper}), 
has the advantage of being explicit. Its proof is very close to the proof of Lemma 5.3
in \cite{CoornKnieper}.

\begin{lemma}
	\label{lemma:growth_geo}
	Let $M$ be a complete hyperbolic $3$-manifold. 
	For a compact domain $K \subset M$ define 
	\[ 
	\mathcal{P}_K(t) = \# \left\{ \varphi \in \mathcal{C}(M) \mid \varphi(S^1) \cap K \neq \emptyset, \length(\varphi) \leq t \right\}.
	\]
	Then, $\mathcal{P}_K(t) \leq C e^{2 t}$, with $C = \pi \frac{e^{8\diam K}}{\Vol K}$.
\end{lemma}
\begin{proof}
	Let $M = \hyp^3 / \Gamma$, with $\Gamma$ a subgroup of $\Isom \hyp^3$,
	and let $\pi \colon \hyp^3 \to M$ denote the covering projection.
	Pick a point $p\in\mathbf H^3$ with $\pi(p)\in K$ and  
	consider the Dirichlet domain centred at $p$:
	$$
		  D(p)=\{x\in\mathbf H^3\mid d(x,\gamma (p))\leq d(x,p), \ \forall \gamma\in\Gamma\}.
	$$
	The intersection $\tilde K=D(p)\cap \pi^{-1}(K)$ is a fundamental domain for $K$, which means that 
	$\pi^{-1}(K) = \bigcup_{\gamma\in\Gamma}\gamma(\tilde K)$ and 
	$\Vol (\gamma_1 (\tilde K)\cap \gamma_2(\tilde K))=0$,
	for all  $\gamma_1\neq\gamma_2 \in\Gamma$.
%
	Moreover, $\diam\tilde K\leq 2\diam K$ and $\Vol\tilde K=\Vol K$.
        Now let $\varphi \in \mathcal{C}(M)$ intersecting $K$.
	Then there exists an isometry $\gamma \in \Gamma$ of hyperbolic type representing 
	$\varphi$ whose axis intersects $\tilde K$. We claim that 
	$$
	  \gamma(\tilde K)\subset B(p,4\diam K + \length (\varphi)).
	$$
	To prove this inclusion, we pick a point  $q\in\tilde K$  that lies in the axis of $\gamma$. For any 
	$q'\in \tilde K$,
	$$
		d(p,\gamma(q'))\leq d(p,q)+d(q,\gamma(q)) + d(\gamma(q), \gamma(q')) \leq 4\diam K + d(q, \gamma(q))
	$$
	and $d(q, \gamma(q))= \length (\varphi)$, which proves the claim. Hence, for any geodesic contributing to $\mathcal{P}_K(t) $, there is a hyperbolic isometry
	whose axis is a lift of this geodesic and such that $\gamma(\tilde K)\subset B(p,4\diam K +t)$.
	In addition, $\Vol (\gamma_1(\tilde  K)\cap \gamma_2(\tilde K))=0$, for all $\gamma_1\neq\gamma_2\in\Gamma$.
	Thus we get the inequality:
	$$
	\mathcal{P}_K(t) \Vol K = \mathcal{P}_K(t) \Vol \tilde K\leq  \Vol B_p(4\diam K + t ) \leq \pi e^{8\diam K + 2t}. 
	$$
	We have used that the volume of a ball of radius $R$ in $\hyp^3$ is less than $\pi e^{2R}$.
\end{proof}

\subsection{Complex-length spectrum}

Any closed geodesic $\varphi \in\mathcal{C}(M)$ has attached 
two geometric invariants: its length and its geometric torsion. 
Recall that the geometric torsion of $\varphi$ is defined as
the oriented angle between an orthogonal vector to 
$\varphi$ and the parallel transport of it along $\varphi$.
In terms of the holonomy representation, these 
two invariants are the translation distance and the rotational part of the 
corresponding hyperbolic isometry. More explicitly, if $[\gamma] \in \text{HypC}(\pi_1 (M, p))$,
then
\[
\Hol_M (\gamma)
\sim
\left[\left(\begin{array}{cc} 
	e^{ \lambda/2} & 0 \\ 
	0 & e^{-\lambda/2} 
\end{array}\right)\right] \in \psl(2,\cmplx),
\quad \Re(\lambda) > 0,
\]
$\Re(\lambda)$ is the length of the corresponding closed geodesic, and $\Im(\lambda)$ 
its geometric torsion. The parameter $\lambda$ is called the \emph{complex length} of $\gamma$, 
and it is only well defined up to $2\pi i$.
We will regard this as a function
\[
\begin{array}{rcl}
	\lambda \colon \mathcal{C}(M) & \to &  \cmplx /\langle 2\pi i\rangle \\
	\varphi & \mapsto & \lambda(\varphi) = \length(\varphi) + i \operatorname{torsion}(\varphi).
\end{array}
\]
To avoid the $2\pi i$ indeterminacy, we will work with the exponential of this map.

\begin{definition}
	The \emph{(prime) complex-length spectrum} of $M$, denoted as $\pclsp M$,
	is the measure on $\cmplx$ defined by
	\[
	\pclsp M = \sum_{ \varphi \in \mathcal{PC}(M)} \delta_{ e^{ \lambda(\varphi) } },
	\]
	where $\delta_x$ is the Dirac measure centered at $x$.
	In other words, $\pclsp M$ is the image measure of the counting measure in $\mathcal{PC}(M)$ 
	under the exponential of the complex-length function. 
	The \emph{(prime) length spectrum} of $M$, denoted as $\plsp M$,
	is the measure on $\real$ defined by
	\[
	\plsp M = \sum_{ \varphi \in \mathcal{PC}(M) } \delta_{ \length{(\varphi)} }.
	\]
\end{definition}
Thus we have:
\[
\#\{ \varphi \in \mathcal{PC}(M) \mid a < \length(\varphi) < b\} = \pclsp M \{ z \in \cmplx \mid e^a < |z| < e^b\}.
\]
\begin{rem}
	The prime complex-length spectrum is usually regarded as a collection of numbers and multiplicities.
	This is of course equivalent to the definition made above; however, we think that some of the results 
	that we will present in what follows are better expressed in these terms. 
\end{rem}

The following properties of $\pclsp M$ are immediately implied by Lemma \ref{lemma:growth_geo}
and the fact that a closed geodesic cannot be contained in a cusp.

\begin{proposition}
	Assume that $M$ has finite volume. The following assertions then hold:
	\begin{enumerate}
		\item The measure $\pclsp M$ is locally finite with discrete support.
			In particular, it is a Radon measure on the complex plane.
		\item Let $N_1,\dots, N_j$ be cusps of $M$ in such a way that $K = M\setminus \bigcup_{1 \leq j \leq n} N_j$
			is compact. Then for all $R > 1$,
			\[
			\pclsp M \big( \{ |z| \leq R \} \big) \leq C_M R^2,
			\]
			where $C_M = \pi \frac{e^{8\diam K}}{\Vol K}$.
	\end{enumerate}
\end{proposition}
Next we want to analyse the complex-length spectrum as a map 
\[
M \mapsto \pclsp M.
\]
The domain of this map will be the set $\mathcal{M}$ of all
(isometry classes of) oriented, complete, hyperbolic $3$-manifolds of finite volume.
This set is naturally endowed with the geometric topology, which is briefly discussed in next subsection.
On the other hand, the target of this map will be $M( \cmplx \setminus \disk )$,
the vector space of $\cmplx$-valued Radon measures defined on the complement of the closed unit disk 
$\disk$. We will endow $M( \cmplx \setminus \disk )$ with the topology
of the weak convergence. Thus a sequence $\{\mu_n\}$ converges weakly to $\mu$
in $M( \cmplx \setminus \disk )$ if for every continuous function $f$ with compact support contained in $\cmplx \setminus \disk$,
we have:
\[
\lim_{n\to \infty} \int_{|z| > 1 } f(z) d\mu_n(z) = \int_{ |z| > 1} f(z) d\mu(z).
\]
The aim of the rest of this subsection is essentially to prove that this map is continuous.

\begin{theorem} \label{thm:cont_cmplxlsp}
	The map $\pclsp \colon \mathcal{M} \to M( \cmplx \setminus \disk )$ is continuous.
\end{theorem}

\begin{rem}
	If we consider the space $M(\cmplx)$ instead of $M( \cmplx \setminus \disk )$, 
	then Theorem~\ref{thm:cont_cmplxlsp} is no longer true.
	For instance, let $M \in \mathcal{M}$ be a one-cusped manifold, and $M_{p/q}$ be the
	manifold obtained by a hyperbolic $(p,q)$-Dehn filling. Then $\{M_{p/q}\}_{(p,q)}$
	converges to $M$ as $(p, q)$ goes to infinity. However, the sequence of the corresponding measures 
	do not even converge. To see this, let $\pm \varphi_{p/q}$ be the two (oriented) core prime geodesics
	added in the Dehn filling. Then the length of $\varphi_{p/q}$ goes to zero, and the geometric torsion 
	is dense in $\real/2\pi\integer$, which implies that this sequence of measures does not converge.
	Restricting our attention to $M( \cmplx \setminus \disk )$ we avoid these phenomena. 
	Nevertheless, this bad behaviour is the worst that can happen; this is expressed in the following result.
\end{rem}

\begin{theorem} \label{thm:cont_lsp}
	Let $M \in \mathcal{M}$ with $k > 0$ cusps, and $\{M_n\}$ be a sequence converging to $M$ 
	in $\mathcal{M}$. Assume that the number of cusps of $M_n$ is eventually constant 
	and is equal to $l$. Then the sequence of real-length spectrum measures $\{\plsp M_n\}$ converges weakly 
	in $M(\real)$ to the measure
	\[
	\plsp M + 2(k - l)\delta_0.
	\]
\end{theorem}

Theorems~\ref{thm:cont_cmplxlsp} and~\ref{thm:cont_lsp} will be proved in Section \ref{sec:continuity} after 
having discussed the geometric topology.

\subsection{The geometric topology}

Most of the material in this subsection is based on \cite{CanEpsGre}. 

Let $\mathcal{MF}$ be the set of (isometry classes of) oriented, complete, hyperbolic $3$-manifolds of finite volume
and with a baseframe.
Thus an element of $\mathcal{MF}$ is a pair $(M, E)$, where $E$ is an orthonormal frame based at some point $p$ in the oriented hyperbolic 
$3$-manifold $M$ of finite volume.

\begin{rem}
	Our notation differs from \cite{CanEpsGre}, where $\mathcal{MF}$ is defined without the finite volume restriction.
\end{rem}

If we fix a base frame on hyperbolic space $\hyp^3$, then the holonomy representation of a member of $\mathcal{MF}$
is unambiguously defined (\idest not only up to conjugation). Therefore, $\mathcal{MF}$ is in one-to-one
correspondence with the set of discrete torsion-free subgroups of $\psl(2,\cmplx)$ with \emph{finite co-volume}.
The latter set is endowed with the \emph{geometric topology}. 
We recall its definition in the general context of Lie groups, see \cite{ThurstonNotes}.

\begin{definition}
	A sequence $\{ \Gamma_n \}$ of closed subgroups of a Lie group $G$ 
	\emph{converges geometrically} to a group $\Gamma$ if the following
	conditions are satisfied:
	\begin{enumerate}
		\item Each $\gamma \in \Gamma$ is the limit of a sequence $\{\gamma_n \}$,
			with $\gamma_n \in \Gamma_n$.
		\item The limit of every convergent sequence $\{ \gamma_{n_j} \}$,
			with $\gamma_{n_j} \in \Gamma_{n_j}$, is in $\Gamma$ ($n_j$ 
			is an increasing sequence of natural numbers).
	\end{enumerate}
\end{definition}

Two related spaces are $\mathcal{MB}$ and $\mathcal{M}$. The former is obtained by forgetting the frame,
but retaining the basepoint, and the latter by forgetting both the frame and the basepoint.
Both sets are endowed with the quotient topology given by the corresponding forgetful maps.

The following results are well known, and will play an important role in the 
following subsections. See \cite{CanEpsGre} for a proof.

\begin{lemma} \label{lemma:injradcont}
	Let $\operatorname{inj}_R(M,p)$ be the infimum of the injectivity radius 
	on the ball $B_R(p) \subset M$. Then for any $ R > 0$ the map
	$\operatorname{inj}_R \colon \mathcal{MB} \to (0, \infty)$ is continuous.
\end{lemma}

\begin{lemma} \label{lemma:bdddiameterthick}
	Let $\epsilon > 0 $ less than the Margulis constant.
	Let $\{M_n\}$ be a sequence converging to $M$ in $\mathcal{M}$. 
	Then there exists a uniform bound on the diameter of the thick parts
	$\{ M_{n,[\epsilon, \infty)}\}$.
\end{lemma}

\begin{theorem}[J{\o}rgensen]
	The map $\Vol \colon \mathcal{M} \to \real$ that assigns to each manifold
	its volume is continuous.
\end{theorem}

The following theorem due to Thurston describes how a non-trivial convergence sequence
in $\mathcal{M}$ is. We recall that we are assuming that all manifolds have finite volume.
\begin{theorem}[Thurston]
	Let $\{M_n\}$ be a sequence converging to $M$ in $\mathcal{M}$.
	Assume that $\{M_n\}$ is not eventually constant, and that $M$ has $k$ cusps.
	Then $M_n$ is obtained by hyperbolic Dehn surgery $M_{p_{1,n}/q_{1,n},\dots,p_{k,n}/q_{k,n}}$,
	with $p_{i,n}^2 + q_{i,n}^2 \to \infty$, as $n \to \infty$.
\end{theorem}

\begin{corollary} \label{coro:thurston}
	Let $\{(M_n, E_n)\}$ be a sequence converging to $(M, E)$ in $\mathcal{MF}$. 
	Then, for $n$ large enough, we have a commutative diagram, 
	\[
	\xymatrix{
	\pi_1 (M, p) \ar[r]^<<<<<{\rho_n} \ar[d]^{i_*^{n}} & \psl(2;\cmplx) \\
	\pi_1 (M_n, p_n) \ar[ur]_{\Hol_{M_n}}&
	}
	\]
	Moreover, the sequence of representations $\{\rho_n\}$ converges to $\Hol_M$
	both algebraically (that is, for all $\sigma \in \pi_1 (M, p)$ the sequence $\{\rho_n(\sigma)\}$ converges to $\Hol_M(\sigma)$),
	and geometrically (that is, the sequence of discrete groups $\{ \rho_n(\pi_1(M,p)) \}$ converges geometrically
	to $\Hol_M(\pi_1(M, p))$).
\end{corollary}

It can be proved that if a sequence $\{(M_n, p_n)\}$ converges to $(M, p)$ in $\mathcal{MB}$,
then it also converges to $(M, p)$ in the pointed Hausdorff-Gromov sense, see \cite{CanEpsGre}.

Next we want to give the following \emph{ad hoc} definition concerning the convergence 
of geodesics. 
\begin{definition}
	With the previous notation, we will say that a sequence of parametrised closed geodesics 
	$\{ \varphi_n \colon [0, 1] \to M_n \}$
	converges to $\varphi \colon [0,1] \to M$ if for all $n$ there is a lift 	
	of $\varphi_n$ (with respect to the covering map $\pi_n$)
	\[
	\lift{\varphi}_n \colon [0, 1] \to \hyp^3,
	\]
	such that the sequence of maps $\{\lift{\varphi}_n \}$
	converges pointwise to a lift of $\varphi$ (with respect to the covering map $\pi$). 
\end{definition}
\begin{rem}
	This definition coincides with the more general (and natural) definition of convergence 
	of maps $\{f_n \colon X_n \to Y_n \}$, where $\{X_n\}$ and $\{Y_n\}$ are sequences of compact metric
	space converging in the Hausdorff-Gromov sense to $X$ and $Y$ respectively, see \cite{GrovPet}. 
\end{rem}

With this definition, it is quite obvious that the limit of parametrised closed geodesics is
also a geodesic whose length is the limit of the lengths of the converging geodesics. 

\begin{definition}
	We will say that a sequence $\{ \varphi_n \}$ of closed geodesics, with $\varphi_n \in \mathcal{C}(M_n)$,
	converges to $\varphi \in \mathcal{C}(M)$ if for all $n$ we can choose parametrisations of $\varphi_n$
	converging to a parametrisation of $\varphi$ (in the sense of this definition).
\end{definition}

Again the following result holds in a more general context, see \cite{GrovPet}. 
Its proof in our case is quite obvious.

\begin{theorem} [Ascoli-Arzela, Grove-Petersen.]
	Let $R > 0$ and $\{ \varphi_n \}$ be a sequence of closed geodesics with $\varphi_n \subset B_R(p_n) \subset (M_n, p_n)$.
	If there exists a common upper bound on the lengths of $\{ \varphi_n \}$, then $\{\varphi_n \}$ has a converging subsequence.
\end{theorem}

\subsection{Proof of the continuity} 
\label{sec:continuity}
In this subsection we want to prove the continuity of the complex-length spectrum as 
a map from $\mathcal{M}$ to $M( \cmplx \setminus \disk )$.
An obvious observation is that we can assume that this map is defined 
from $\mathcal{MF}$ to $M( \cmplx \setminus \disk )$,
since the topology of $\mathcal{M}$ is the quotient topology coming from the forgetful map 
$\mathcal{MF} \to \mathcal{M}$.

Hereafter $\{ (M_n, E_n) \}$ will denote a sequence converging to $\{ (M, E) \}$ in $\mathcal{MF}$.
In order to simplify notation, we will write $\mu_n$ and $\mu_\infty$ for $\pclsp M_n$
and $\pclsp M$, respectively.
We want to prove that the sequence of measures $\{\mu_n\}$ converges
to $\mu_\infty$ in $M( \cmplx \setminus \disk )$. 
Our first task is to translate this into geometrical terms.

Recall from last subsection, Corollary \ref{coro:thurston}, that we have a commutative diagram,
\[
\xymatrix{
\pi_1 (M, p) \ar[r]^<<<<<{\rho_n} \ar[d]^{i_*^{n}} & \psl(2;\cmplx) \\
\pi_1 (M_n, p_n) \ar[ur]_{\Hol_{M_n}}&
}
\]
Furthermore, the sequence of representations $\{\rho_n\}$ converges 
both algebraically and geometrically to $\Hol_M$. 

Let $\sigma \in \pi_1 (M, p)$ be a hyperbolic element.
The algebraic convergence of $\{\rho_n\}$ implies that
$\rho_n(\sigma)$ is also of hyperbolic type for large $n$ 
(it follows for instance from the fact that the set of hyperbolic isometries is open in $\psl(2,\cmplx)$).
As a consequence, for large $n$, the conjugacy class of $i_*^n(\sigma)$ defines a closed geodesic in $M_n$;
moreover, the complex length of $\rho_n(\sigma)$ is close to that of $\Hol_M(\sigma)$.

Let $0 < a < b$. Then, for large $n$,
the map $i_*^n \colon \pi_1(M, p) \to \pi_1(M_n, p_n)$ gives a well defined map
\[
\iota_{a,b,n} \colon \left\{ \varphi \in \mathcal{C}(M) \mid a < \length(\varphi) < b \right\} 
\to \left\{ \varphi \in \mathcal{C}(M_n) \mid a < \length(\varphi) < b \right\}.
\]

\begin{lemma} \label{lemma:equivalencewkconv}
	Assume that for all $0 < a < b$ \emph{not in the real-length spectrum} of $M$ there exists 
	$N(a,b)$ such that for all $n> N(a,b)$ the map $\iota_{a,b,n}$ 
	is a bijection when restricted to prime geodesics.
	Then $\{ \mu_n \}$ converges weakly to $\mu_\infty$.
\end{lemma}
\begin{proof}
	For two real numbers $ a < b$ put $D_{a,b} = \{ z \in \cmplx \mid e^a < |z| < e^b \}$. 
	Let $f$ be a continuous function with compact support contained in the exterior of the unit disk.
	Take $1 < a < b$ such that $\supp f \subset D_{a,b}$, 
	with both $a$ and $b$ not in the real length spectrum of $M$.
	Let $A = \{ \varphi_1,\dots,\varphi_k\}$ be the set of prime closed geodesics in $M$ with 
	complex length in $D_{a,b}$. 
	Therefore, we have
	\[
	\int_{|z| > 1} f(z) d\mu_\infty(z) = \int_{D_{a,b}} f(z) d\mu_\infty(z) = \sum_{i = 1}^k f( \lambda(\varphi_i) ).
	\]
	By hypothesis, for $n > N(a,b)$, we have
	\[
	\int_{|z| > 1} f(z) d\mu_n (z) = \int_{D_{a,b}} f(z) d\mu_n(z) = \sum_{i = 1}^k f( \lambda_n( \iota_{n,a,b}(\varphi_i))),
	\]
	where $\lambda_n$ is the complex-length function of $M_n$. The algebraic convergence implies 
	\[
	\lim_{n\to\infty} \lambda_n( \iota_{n,a,b}(\varphi_i)) = \lambda(\varphi_i),
	\]
	and the continuity of $f$ gives
	\[
	\lim_{n\to\infty} \int_{|z| > 1} f(z) d\mu_n (z) = \int_{|z| > 1} f(z) d\mu_\infty(z).
	\]
	Hence, $\mu_n$ converges weakly to $\mu_\infty$.
\end{proof}

Next we want to prove that the hypothesis of the previous lemma is satisfied.
Hereafter, $a$ and $b$ will denote two fixed positive real numbers not in the length spectrum
of $M$ with $a < b$. We will write $\iota_n$ instead of $\iota_{a,b,n}$.

The following lemma is an immediate consequence of the convergence of $\{(M_n, E_n)\}$ to $(M,E)$.
\begin{lemma}
	Let $\varphi \in \mathcal{C}(M)$. Then the sequence of closed geodesics
	$\{ \iota_n( \varphi) \}$ converges to $\varphi$.
\end{lemma}
\begin{proposition} \label{prop:injectiveiotannotprime}
	Let $\varphi_1, \varphi_2 \in \mathcal{C}(M)$.
	If $\varphi_1 \neq \varphi_2$ then, for $n$ large enough, 
	$\iota_n(\varphi_1) \neq \iota_n(\varphi_2)$. 
\end{proposition}
\begin{proof}
	We have $d(\varphi_1, \varphi_2) > 0$. The  lemma above then implies that 
	for large $n$ also 
	\[
	d( \iota_n(\varphi_1), \iota_n(\varphi_2) ) > 0.
	\]
\end{proof}

\begin{proposition}\label{prop:iotanprime}
	If $\varphi \in \mathcal{C}(M)$ is prime, then,
	for $n$ large enough, $\iota_n(\varphi)$ is also prime.
\end{proposition}
\begin{proof}
	Take $R > 0$ such that $\iota_n(\varphi) \subset B_R(p_n)$ for all $n$.
	If the lemma were false, then (up to a subsequence) for all $n$,
	$\iota_n(\varphi) = k_n \psi_n$ for some integer $k_n \geq 2$ and some $\psi_n \in \mathcal{PC}(M_n)$.
	By Lemma \ref{lemma:injradcont}, the injectivity radius on $B_R(p_n)$ is uniformly bounded from below away from zero;
	hence, $k_n$ must be bounded from above.
	Therefore, (up to a subsequence) for all $n$, $\iota_n(\varphi) = k \psi_n$, for some fixed $k \geq 2$.
	The geodesics $\{ \psi_n \}$ have bounded length and are contained in $B_R(p_n)$; hence, by Ascoli-Arzela
	(up to a subsequence) they converge to a geodesic $\psi$ which satisfies $\varphi = k \psi$,
	contradicting the primality of $\varphi$. 
\end{proof}

These two preceding results imply that, for large $n$, $\iota_n$ gives an injective map
\[
\left\{ \varphi \in \mathcal{PC}(M) \mid a < \length(\varphi) < b \right\} 
	\to \left\{ \varphi \in \mathcal{PC}(M_n) \mid a < \length(\varphi) < b \right\},
\]
Next we want to prove that, for a larger $n$, this map is surjective.
We will proceed by contradiction using an Arzela-Ascoli argument. 
Before doing this, we need to prove that we have a control 
on the set of prime closed geodesics in $M_n$ whose lengths are in $(a, b)$.
This is the content of the following result, which is just an application 
of the thick-thin decomposition of a complete finite-volume hyperbolic manifold. 

\begin{lemma} \label{lemma:prime_geo_ball}
	There exists $R > 0$ such that for all $n$ any \emph{prime} closed geodesic
	in $(M_n, p_n)$ of length in $(a, b)$ is contained in $B_R(p_n)$.
\end{lemma}
\begin{proof} 
	Let $\epsilon > 0 $ be less than $a/2$ and the Margulis constant.
	If necessary, take a smaller $\epsilon > 0$ to guarantee that $p_n \in M_{n,[\epsilon,\infty)}$.
	If $\varphi$ is a closed geodesic in $M_n$ of length $ \length(\varphi) > a$,
	then $\varphi$ must intersect the $\epsilon$-thick part $M_{n,[\epsilon,\infty)}$
	(otherwise $\varphi$ would be the core of a Margulis tube in $M_{n,(0,\epsilon)}$, 
	and the injectivity radius in that tube would be achieved by the curve $\varphi$,
	so $a/2 < \length(\varphi)/2 < \epsilon$, which is absurd). 
	The result then follows from the fact that the diameter of $M_{n,[\epsilon,\infty)}$
	is uniformly bounded on $n$.
\end{proof}

\begin{lemma} \label{lemma:surjectiveiotan}
	There exists $N$ such that for all $n > N$ the following holds:
	if $\varphi_n$ is a prime closed geodesic in $M_n$ of length
	$ a < \length(\varphi_n) < b$, then there exists
	a prime closed geodesic in $M$ of length $a < \length(\varphi) < b$
	with $\varphi_n = \iota_n(\varphi)$.
\end{lemma}
\begin{proof}
	Assume that the lemma is false. Up to a subsequence, for all $n$ there exists
	a prime closed geodesic $\varphi_n$ on $M_n$ with $\length(\varphi_n) \in (a, b)$ such that 
	$\varphi_n \neq \iota_n(\psi)$, for all $\psi \in \mathcal{PC}(M)$
	of length in $(a, b)$. Take the $R$ given by Lemma \ref{lemma:prime_geo_ball}.
	By the continuity of the injectivity radius, there exists a uniform lower bound 
	$\epsilon > 0$ on the injectivity radius on $B_{R}(p_n)$.
	Therefore, by Proposition \ref{prop:dist_bb}, for all $\psi \in \mathcal{PC}(M)$
	of length in $(a, b)$,
	\[
	d ( \iota_n(\psi), \varphi_n ) > \epsilon.
	\]
	Up to a subsequence, $\{ \varphi_n \}$ converges to a closed geodesic $\varphi$ in $M$.
	It is easily seen that $\varphi$ must be prime. 
	Since $\length(\varphi) \in (a, b)$ 
	(recall that $a$ and $b$ do not belong to the length spectrum of $M$),
	the inequality above gives 
	\[
	d ( \iota_n(\varphi), \varphi_n ) > \epsilon.
	\]
	It contradicts the fact that both $\{ \varphi_n\}$ and $\{\iota_n(\varphi) \}$ converge to $\varphi$.
\end{proof}

\begin{proof}[Proof of Theorem \ref{thm:cont_cmplxlsp}]
	Propositions \ref{prop:injectiveiotannotprime} and \ref{prop:iotanprime} prove that 
	$\iota_n$ is injective, and Lemma \ref{lemma:surjectiveiotan} states that $\iota_n$ 
	is surjective. Then Lemma \ref{lemma:equivalencewkconv} proves that $\{\mu_n\}$ converges to $\mu_\infty$ weakly.
\end{proof}

It remains to prove Theorem \ref{thm:cont_lsp}. 
In order to do it, we can assume that $M$ has $k$ cusps, and that
the sequence $\{M_n\}$ converging to $M$ in $\mathcal{M}$ is obtained by performing Dehn fillings on $k-l$  fixed 
cusps of $M$. 
We must prove that the sequence of (real) length spectrum measures $\{\plsp M_n\}$ converges in $M(\real)$ to 
\[
\plsp M + 2(k - l)\delta_0.
\]
By Theorem \ref{thm:cont_cmplxlsp}, it is enough to prove that there exists $ \delta > 0$ less than the 
length of the shortest geodesic in $M$ such that 
\[
\lim_{n\to \infty} \plsp M_n ( [0, \delta) ) = 2(k - l).
\]
In geometrical terms, it is equivalent to the following well known result.
\begin{lemma} \label{lemma:shortgeocore}
	Let $\{ \pm \varphi_n^1,\dots, \pm \varphi_n^{k-l} \}$ be the core geodesics (oriented and prime) 
	in $M_n$ added on the Dehn filling.
	Let $\delta_s$ be the length of the shortest geodesic in $M$, and $\delta \in (0, \delta_s)$.
	Then, for large $n$, the only prime closed geodesics in $M_n$ of length $ < \delta$ are the core geodesics.
\end{lemma}
\begin{proof}
	Take $\epsilon > 0$ less than both the Margulis constant and $\delta/2$. 
	Thus $M_{(0,\epsilon)}$ consists only of cusps.
	Since $\length ( \varphi_n^i)$ goes to zero as $n$ goes to infinity, for large $n$, 
	all the geodesics $\varphi_n^i$ are in $M_{n, (0,\epsilon)}$. 
	Let $T_n^i$ be the Margulis tube corresponding to $\varphi_n^i$, and 
	$\{ C_n^{k-l + 1}, \dots, C_n^k \}$ be the cusp components of $M_{n,(0,\epsilon)}$ corresponding to the 
	non-deformed cusps. Let 
	\[
	F_n =  T_n^1 \cup \cdots \cup T_n^{k-l} \cup C_n^{k-l + 1} \cup \cdots \cup C_n^k \subset M_{n,(0,\epsilon)}.
	\]
	For large $n$, $M_{n,[\epsilon, \infty)}$ is homeomorphic to $M_{[\epsilon, \infty)}$;
	in particular, $M_{n,(0,\epsilon]}$ has $k$ boundary components.
	It implies that, for large $n$, $F_n = M_{n,(0,\epsilon)}$, and the result follows.
\end{proof}

We will need the following improvement of Theorem \ref{thm:cont_cmplxlsp} in the following subsection.

\begin{proposition}
	Let $f \colon \cmplx \to \cmplx$ be a continuous function with $\supp{f}$ not necessarily compact
	but contained in $\cmplx\setminus \disk$. 
	Assume that there exists $\epsilon > 0$, and $K > 0$ such that
	\[
	|f(z)| \leq  \frac{K}{|z|^{2 + \epsilon}},
	\]
	for all $z \in \cmplx$. Then we have:
	\begin{enumerate}
		\item For any $M \in \mathcal{M}$, 
			\[
			\int_{|z| > 1} |f(z)| d\pclsp M (z) < \infty.
			\]
		\item If $\{M_n\}$ converges to $M$ in $\mathcal{M}$, then
			\[
			\lim_{n\to\infty} \int_{|z| > 1} f(z) d\pclsp M_n (z) =  \int_{|z| > 1} f(z) d\pclsp M(z).
			\]
	\end{enumerate}
\end{proposition}
\begin{proof}
	Let $\delta$ be the Margulis constant. Then for all $M \in \mathcal{M}$ any prime closed geodesic in $M$ 
	of length $\geq 2\delta$ intersects the thick part $M_{[\delta, \infty)}$. 
	Let $M \in \mathcal{M}$, and put $\mu = \pclsp M$. Fix $R \gg  1$. 
	By Lemma \ref{lemma:growth_geo}, we have
	\[
	\mu  (\{ e^{2\delta} \leq |z| \leq R \}) \leq C R^2,
	\]
	where $C = \pi \frac{ e^{8\diam M_{[\delta, \infty)}}}{\Vol M_{[\delta, \infty)} }$.
	Then we have, 
	\begin{eqnarray*}
		\int_{ |z| \geq R} |f(z)| d\mu (z)  & = &  \sum_{ k = 0}^\infty \int_{ R 2^k \leq |z| < R 2^{k+1} } |f(z)| d\mu(z) \\
		& \leq &  \sum_{ k = 0}^\infty \int_{  R 2^k \leq |z| < R 2^{k+1} } \frac{K}{|z|^{2 + \epsilon} } d\mu (z) \\
		& \leq &  \sum_{ k = 0}^\infty \frac{K}{(R 2^k)^{2 + \epsilon}} \int_{  R 2^k \leq |z| < R 2^{k+1} } d\mu (z) \\
		& \leq &  \sum_{ k = 0}^\infty \frac{K}{(R 2^k)^{2 + \epsilon}} {C(R2^{k+1})^2} \\
		&   =  &  \frac{KC}{R^\epsilon} \sum_{ k = 0}^\infty \frac{2^{2k+2}}{2^{k(2+\epsilon)}} = 
		\frac{4KC}{R^\epsilon} \frac{1}{1- \frac{1}{2^\epsilon}} = \frac{C'}{R^\epsilon},
	\end{eqnarray*}
	where $C'$ is a constant depending only on $C, K$, and $\epsilon$. The first assertion is then proved.
	Now let $\{M_n\}$ be a sequence converging to $M$ in $\mathcal{M}$. Let us put $\mu_n = \pclsp M_n$.
	Since both $\diam M_{n,[\delta, \infty)}$ and $\Vol M_n$ are uniformly bounded on $n$, 
	last inequality implies that there exists a constant $C''$ such that for all $n$
	\[
	\int_{ |z| \geq R} |f(z)| d\mu_n (z) \leq \frac{C''}{R^\epsilon}.
	\]
	Thus we have
	\begin{eqnarray*}
		\left| \int_{ |z| > 1} f(z) ( d\mu_n(z) - d\mu(z) ) \right| 
		& \leq & \left| \int_{ 1 < |z| < R} f(z) ( d\mu_n(z) - d\mu(z) ) \right| \\
		&    + & \left| \int_{ |z| \geq R} f(z)      ( d\mu_n(z) - d\mu(z) ) \right| \\
		& \leq & \left| \int_{ 1 < |z| < R} f(z) ( d\mu_n(z) - d\mu(z) ) \right|  + \frac{C'' + C'}{R^{\epsilon}}.
	\end{eqnarray*}
	Theorem \ref{thm:cont_cmplxlsp} shows that 
	\[
	\lim_{n\to \infty} \left| \int_{ |z| > 1} f(z) ( d\mu_n(z) - d\mu(z) ) \right| \leq \frac{C'' + C'}{R^{\epsilon}}.
	\]
	Since $R$ is arbitrary and independent of both $C$ and $C''$, the left hand side of this equation must vanish.
	This proves the proposition.
\end{proof}

\subsection{Spin-complex-length spectrum}

Let $(M,\eta)$ be an spin complete hyperbolic $3$-manifold,
and consider its holonomy representation,
\[
\Hol_{(M,\eta)} \colon \pi_1(M, p) \to \sln(2,\cmplx).
\]
If $\gamma \in \pi_1(M, p)$ is of hyperbolic type then,  
\[
\Hol_{(M,\eta)} (\gamma)
\sim
\left(\begin{array}{cc} 
	e^{ \lambda/2} & 0 \\ 
	0 & e^{-\lambda/2} 
\end{array}\right) \in \sln(2,\cmplx),
\quad \Re(\lambda) > 0.
\]
The \emph{spin complex length} of $\gamma$ is by definition the parameter $\lambda \in \cmplx/\langle 4\pi i\rangle$.
Hence, in contrast to the usual complex length, $e^{\lambda/2}$ is well defined 
(we have a well defined sign given by the lift of the holonomy). 
We propose the following definition.

\begin{definition}
	The (prime) spin-complex-length spectrum of $(M,\eta)$ is defined by 
	\[
	\pclsp (M,\eta) = \sum_{ \varphi \in \mathcal{PC}(M)} \delta_{ e^{ \lambda(\varphi)/2 } },
	\]
	where $\delta_x$ is the Dirac measure centered at $x$.
\end{definition}
\begin{rem}
	The image measure of $\pclsp (M,\eta)$ under the function $z\mapsto z^2$ is 
	$\pclsp M$.
\end{rem}
The results obtained for the length spectrum in the previous subsections extend in a natural way 
for the spin-complex-length spectrum, and their proofs will be omitted. 
To do that we must consider the space $\mathcal{MSF}$ of
spin-hyperbolic manifolds with a baseframe. In this case we have the identification between $\mathcal{MSF}$
and the space of discrete torsion-free subgroups of $\sln(2, \cmplx)$ with finite co-volume. We topologize $\mathcal{MSF}$
in such a way that this identification becomes a homeomorphism. The quotient spaces $\mathcal{MSB}$
and $\mathcal{MS}$ are then defined as in the non-spin case.

\begin{theorem}
	The map $\pclsp \colon \mathcal{MS} \to M( \cmplx \setminus \disk )$ is continuous.
\end{theorem}

As in the non-spin case, we can improve the continuity in the following sense.
Notice that the condition on the decay at infinity must be replaced, since the measure of the ball $B_R (0) \subset \cmplx$
under the measure $\pclsp (M,\eta)$ is equal to the measure of the ball $B_{R^2}(0)$ under the measure $\pclsp M$.

\begin{proposition} \label{prop:spclspimprovedcont}
	Let $f \colon \cmplx \to \cmplx$ be a continuous function with support contained in $|z| > 1$. 
	Assume that there exists $\epsilon > 0$, and $K > 0$ such that
	\[
	|f(z)| \leq  \frac{K}{|z|^{4 + \epsilon}},
	\]
	for all $|z| > 1$.
	If $\{ (M_n, \eta_n) \}$ converges to $(M,\eta)$ in $\mathcal{MS}$, then
	\[
	\int_{|z| > 1} |f(z)| d\mu (z), \int_{|z| > 1} |f(z)| d\mu_n (z) < \infty,
	\]
	and 
	\[
	\lim_{n\to\infty} \int_{|z| > 1} f(z) d\mu_n (z) =  \int_{|z| > 1} f(z) d\mu(z).
	\]
	Where $\mu = \pclsp (M,\eta)$ and $\mu_n = \pclsp (M_n,\eta_n)$.
\end{proposition}

\section{Asymptotic behavior}
\label{chapter:asymptotic}

\newtheorem*{thmwotzke}{ Theorem \cite{Wotzke}(Wotzke)}

The aim of this section is to establish the asymptotic behavior 
of the n-dimensional hyperbolic Reidemeister torsion.
More concretely, we will prove the following result. 

\begin{theorem} \label{thm:asymp_main}
	Let $M$ be a connected, complete, hyperbolic $3$-manifold of finite volume.
	Then 
	\[
	\lim_{k\to\infty} \frac{\log |\normaltor_{2k + 1}(M)|}{(2k +1)^2} = -\frac{\Vol(M)}{4\pi}.
	\]
	In addition, if $\eta$ is an acyclic spin structure on $M$, then 
	\[
	\lim_{k\to\infty} \frac{\log |\normaltor_{2k}(M, \eta)|}{(2k)^2} = -\frac{\Vol(M)}{4\pi}.
	\]
\end{theorem}

For a compact manifold, this theorem is due to M\"uller, see \cite{Mul}.
In this case, we can consider $\tor_n (M; \eta )$ for all $n$ (\idest there is no need to consider 
the normalized torsion $\normaltor_{n}(M, \eta)$). 

\begin{theorem}[W.\;M\"uller, \cite{Mul}] \label{th_mull}
	Let $(M,\eta)$ be a connected, closed, spin-hyperbolic $3$-manifold.
	Then we have: 
	\begin{equation*}
		\lim_{n\to\infty} \frac{\log |\tor_n (M; \eta )| }{n^2} = - \frac{\Vol(M)}{4\pi}.
	\end{equation*}
\end{theorem}

The proof given by M\"uller is based on the fact that 
the Reidemeister torsion coincides with the Ray-Singer analytic
torsion for a compact manifold.
Since \emph{a priori} the Ray-Singer torsion is not even defined for 
non-compact manifolds, it seems difficult to adapt 
M\"uller's proof to the non-compact case. 
Nevertheless, M\"uller's techniques are still powerful
in the non-compact case, and will play a crucial role
in our proof of Theorem \ref{thm:asymp_main}. 
Roughly speaking, our approach will consist in approximating 
the cusp manifold $M$ by compact manifolds obtained 
by hyperbolic Dehn filling; then we will
apply M\"uller's theorem to these compact manifolds
and the surgery formulas for the torsion stated in Section \ref{chapter:surgery}.
The continuity of the (spin-)complex-length spectrum established in Section \ref{chapter:cmplxlength} 
will allow us to handle this limit process.

The distribution of this section is as follows. The first subsection is an exposition of the
notions concerning the Ray-Singer analytic torsion and Ruelle zeta functions that will be needed
in the subsequent subsections; that subsection ends with Wotzke's theorem in dimension three, which 
gives the relationship between Ruelle zeta functions and the Reidemeister torsion invariants that we are studying.
In the second subsection, we will state the theorem by M\"uller from which he
deduces the asymptotic behaviour for the compact case.
That theorem establishes a formula for the Ray-Singer analytic torsion,
which will be the essential ingredient for the proof of Theorem \ref{thm:asymp_main} given in the last subsection.

\subsection{Ruelle zeta functions}

Let $M$ be a differentiable closed $n$-manifold with a Riemannian metric $g$.
Let us assume that we have an \emph{acyclic} orthogonal 
(or unitary) representation of the fundamental group 
\[\rho\colon \pi_1 M \to \ort(n).\]
The analytic Ray-Singer torsion $T(M; \rho)$, introduced by Ray and Singer 
in the seminal paper \cite{RaySin}, is a certain weighted alternating product of regularized 
determinants of the Laplacians 
\[
\Delta^q \colon \Omega^q(M;E_\rho)\to \Omega^q(M;E_\rho).
\]
A theorem proved in \cite{RaySin} states that the Ray-Singer 
torsion is independent of the chosen metric.
Hence, it is usually denoted simply as $T(M; \rho)$,
without making reference to the metric $g$.

In the paper mentioned above, Ray and Singer conjectured that the Reidemeister 
torsion $\tor(M;\rho)$ agrees with the 
analytic torsion $T(M; \rho)$. This conjecture was proved independently by Cheeger
and M\"uller in \cite{Ch} and \cite{MulAT} respectively. In \cite{MulUni}, 
M\"uller extended the definition of the analytic torsion 
to unimodular representations 
\[
\rho\colon \pi_1 M \to \sln(n, \cmplx).
\]
As in the orthogonal case, this definition requires a Riemannian metric,
but, in contrast to the orthogonal case,
this new analytic torsion is only metric independent for odd dimensions. 
In that paper, M\"uller also proved that both the analytic torsion 
and the Reidemeister torsion agree for an odd dimensional closed manifold.

An important part of this story concerns the relation between 
the Ray-Singer torsion and Ruelle zeta functions for a compact negatively curved manifold $M$. 
Since it will play a crucial role in the proof of our main theorem,
we will spend the rest of this subsection to explain it.

Let $\Gamma$ be a torsion free co-compact subgroup of $\Isom^+ \hyp^n$, 
and let $M = \hyp^n/ \Gamma$ be the corresponding hyperbolic manifold. 
The classical Ruelle zeta function associated to $M$ is formally defined as 
\[
R(s) = \prod_{ [\gamma] \in \text{PC}(\Gamma) } \big( 1 - e^{-s \length(\gamma)} \big),
\]
where $\length(\gamma)$ is the length of the prime oriented closed geodesic defined by 
the prime conjugacy class $[\gamma]$ of $\Gamma$.
The region of convergence of $R(s)$ can be determined using the asymptotic behaviour 
of the number of closed geodesics of length less or equal than a given value.
To that end, define $P(t)$ as 
\[
P(t) = \# \Big\{ [\gamma] \in \text{PC}(\Gamma) \mid l(\gamma) \leq t \Big\}.
\]
Margulis studied the function $P(t)$ for a closed manifold 
of negative curvature in \cite{Margulis}. Among other things, 
he proved that
\[ 
\lim_{t \to \infty} \frac{P(t)} {e^{ht}/{ht}} = 1,
\]
where $h$ is the topological entropy of the geodesic flow. 
The topological entropy of a hyperbolic manifold of dimension $n$ 
is $h = n - 1$. Using Margulis' result,
the region of convergence of $R(s)$ is easily seen to be 
\[
\Big\{ s \in \cmplx \mid \Re(s) > n - 1 \Big\}.
\]
In \cite{Fried1}, Fried gave the following generalization on the definition 
of the Ruelle zeta function. 
Given an \emph{orthogonal representation} 
$\rho \colon \pi_1(M) \rightarrow O(d)$, which need not to be acyclic,
the twisted Ruelle zeta function associated to $\rho$ is defined as 
\[
R_\rho(s) = \prod_{[\gamma] \in \text{PC}(\Gamma)} \det\left( \text{Id} - \rho(\gamma) e^{-sl(\gamma)}\right).
\]
The region of convergence of $R_\rho(s)$ is the same as the one of the classical Ruelle 
zeta function (here we are using that $\rho$ is an orthogonal representation). 
In this paper, Fried proved that $R_\rho(s)$ has a meromorphic extension to the whole complex plane;
moreover, if $\rho$ is acyclic, then $R_\rho(s)$ is regular at $s=0$ and 
$|R_\rho(0)| = T(M;\rho)^2$ (if $\rho$ is not acyclic, $R_\rho(s)$ can have a pole at $s = 0$,
and $T(M;\rho)^2$ is equal to the leading term of the Laurent expansion of $R_\rho(s)$ at the origin).

In a posterior paper \cite{Fried2}, Fried proved that for a general representation
$\rho \colon \pi_1 M \rightarrow \gln(d; \cmplx)$
the twisted Ruelle zeta function $R_\rho(s)$ has also a meromorphic extension to the whole plane.
However, he was not able to prove its relationship with the Ray-Singer analytic torsion.
Nevertheless, three years later U.\;Br\"ocker proved in his thesis a similar 
result for representations of the fundamental group that are restrictions of 
finite-dimensional irreducible representations of $\Isom^+ \hyp^n \cong \text{SO}_0(n,1)$, see \cite{Brocker}.
According to M\"uller \cite{Mul}, the methods used by Br\"ocker are based on elaborate computations
which are difficult to verify. Nonetheless, this problem has been overcome by Wotzke 
in his thesis \cite{Wotzke}. The following subsection is dedicated to state Wotzke's 
Theorem in dimension $3$.

\subsection{Wotzke's Theorem}
Let $(M,\eta)$ be a connected, closed, spin-hyperbolic $3$-manifold.
If $\Gamma$ is the image of $\pi_1 (M,p)$ under the $\Hol_{(M,\eta)}$, then
\[ 
(M, \eta) = \Gamma \backslash \sln(2; \cmplx) / \textrm{SU(2)}.
\]
Let $\rho$ be a \emph{real} finite-dimensional representation of
$\sln(2; \cmplx)$, regarded as a real Lie group. 
Denote by $\theta$ the Cartan involution of $\sln(2; \cmplx)$ with respect to $\su$, and put $\rho_\theta = \rho\circ\theta$.
Let $E_\rho\rightarrow M$ be the flat vector bundle associated to $\rho$. 
Introduce some metric on $E_\rho$, and consider the Laplacians 
$\Delta^q \colon \Omega^r(M; E_\rho) \to \Omega^r(M; E_\rho)$.

\begin{theorem}[Wotzke, \cite{Wotzke}] \label{thm:Wotzke}
	With the  notation above, the following assertions hold:
	\begin{enumerate}
		\item If $\rho_\theta$ is not isomorphic to $\rho$, then $R_\rho(s)$ is regular at $s=0$ and
			$$ |R_\rho(0)| = T(M; \rho)^2. $$
		\item Assume that $\rho \circ \theta$ is isomorphic to $\rho$. If $\rho$ is not trivial, 
			then the order $h_\rho$ at $s=0$ of $R_\rho(s)$ is given by
			\[ h_\rho = 2\sum_{q=1}^3 (-1)^q q \operatorname{dim} \operatorname{ker} \Delta^q, \]
			and for the trivial representation we have $h_\rho = 4 - 2 \operatorname{ dim }\cohom^1(M; \real)$.
			The leading term of the Laurent expansion of $R_\rho (s)$ at s = 0 is given by
			\[T(M; \rho )^2 s^{h_\rho}.\]
	\end{enumerate}
	\label{thmwotzke}
\end{theorem}

\begin{rem}
	The Cartan involution of the real Lie algebra
	$\Lie{sl}(2;\cmplx)$ is given by $\theta(X) = -\overline{X}^{t}$.
	It can be checked that a \emph{complex} representation $\rho$ of $\sln(2; \cmplx)$
	is not equivalent to $\rho\circ \theta$.
\end{rem}

\subsection{M\"uller's Theorem}
Let us retain the same notation as in the previous subsection; in particular, $M$ will be assumed to be closed.
For $n>0$, let $\rho_n$ be the $n$-dimensional canonical representation of $(M,\eta)$,
\[
\rho_n \colon \pi_1 (M, p) \cong \Gamma \to \sln(n;\cmplx).	
\]

M\"uller's theorem on the equivalence of the Reidemeister torsion
and the Ray-Singer analytic torsion implies that 
\[
T(M; \rho_{n,\eta} ) = |\tor (M; \rho_n)|.
\]
Let us denote by $R_n(s)$ the Ruelle zeta function associated to the representation $\rho_n$.
Wotzke's Theorem gives
\[
|R_{\rho_n}(0)| = |\tor(M; \rho_n)|^2.
\]
Following \cite{Mul}, the Ruelle zeta function $R_{\rho_n}(s)$ can be expressed in terms of the following related Ruelle
zeta functions, 
\[ 
R_k(s) = \prod_{ [\gamma] \in \text{PC}(\Gamma) } \big(1 - \sigma_k(\gamma) e^{-s \length(\gamma)} \big),
\]
where $\sigma_k(\gamma)$ is defined by
\[
\sigma_k(\gamma) = e^{ k i \Im \lambda(\gamma)/2 } = e^{ ki \theta(\gamma)/2},
\]
with $\theta(\gamma)$ the geometric spin torsion of the closed geodesic defined by $\gamma$.
A straightforward computation then shows that 
\[
R_{\rho_n}(s) = \prod_{k = 0}^n R_{n-2k}(s - (n/2 - k)).
\]

The following theorem by M\"uller relates the Reidemeister torsion, 
Ruelle zeta functions and the volume of the manifold $M$.

\begin{rem}
	M\"uller uses in \cite{Mul} the notation $\tau_n$ to designate the representation coming from
	the nth symmetric power, so his $\tau_n$ is our $\rho_{n+1}$.
\end{rem}

\begin{theorem}[{M\"uller, \cite{Mul}}] \label{thm:MullerRuelle}
	Let $(M,\eta)$ be a closed spin-hyperbolic $3$-manifold, and for $m\geq 3$ let $\rho_m$ 
	be its $m$-dimensional canonical representation.
	Then we have the following equations,
	\begin{eqnarray*}
		\log \left| \frac{\tor (M; \rho_{2m +1}) }{\tor(M; \rho_5)} \right| & = & 
		 \sum_{k=3}^m \log{|R_{2k}(k)| } - \frac{1}{\pi} \Vol M \left( m(m+1) -6 \right), \\
		 \log \left| \frac{\tor (M, \eta; \rho_{2m}) }{\tor(M, \eta; \rho_4) }\right| & = & 
		 \sum_{k=2}^{m-1} \log{\left| R_{2k+1}\left( k +\frac{1}{2}\right) \right| } - \frac{1}{\pi} \Vol M( m^2 - 4) 
	\end{eqnarray*}
\end{theorem}
M\"uller then deduces Theorem \ref{th_mull} from the following lemma, \cite{Mul}.
\begin{lemma}
	For a closed spin-hyperbolic $3$-dimensional manifold $(M, \eta)$ 
	there exists a constant $C>0$, depending only on the manifold $M$, 
	such that for all $m\geq 3$, we have
	\[ 
	\sum_{k=3}^m \left| \log{|R_{2k}(k)| } \right| < C, \quad 
	\sum_{k=2}^{m-1} \left| \log{ \left| R_{2k+1} \left( k +\frac{1}{2} \right)\right| } \right|  < C. 
	\]
\end{lemma}

\subsection{The noncompact case}
Let $(M,\eta)$ be a compactly approximable spin-hyperbolic $3$-manifolds of finite volume.
In this subsection we want to prove that Theorem \ref{thm:asymp_main} holds for $(M,\eta)$ as well. 
We will do this by proving that Theorem \ref{thm:MullerRuelle} holds also for $(M,\eta)$.

The definition of the Ruelle zeta function $R_{\rho_n}$ for $(M, \eta)$ 
is obvious if we define it in terms of prime closed geodesics; more concretely, we define
\[
R_{\rho_n}(s) = \prod_{\varphi \in \mathcal{PC}(M)} \det\left( \text{Id} - \rho_n(\varphi) e^{-sl(\varphi)}\right).
\]
Of course, it makes sense also to define
\[ 
R_k(s) = \prod_{\varphi \in \mathcal{PC}(M)}\left(1 - \sigma_k(\varphi) e^{-s \length(\gamma)} \right).
\]
The function $R_{\rho_n}(s)$ is related to the functions $R(s, \sigma_k)$ in the same way as in the compact case.
The estimations concerning the growth of closed geodesics in $M$ imply that $R(s,\sigma_k)$ converges 
for $\Re(s) > 2$. More accurate estimations will probably allow to 
conclude that the region of convergence of $R(s,\sigma_k)$ is exactly that half-plane.
Therefore, the region of convergence of $R_{\rho_n}(s)$ contains the half-plane $\Re(s) > 2 + n/2 $. 

It is worth noticing that the following equation holds.

\begin{lemma} \label{eq:ruellelsp}
For $k \geq 3$ we have:
\begin{equation} 
	\log{\left|R_{k}\left(\frac{k}{2}\right)\right|}  =  \int_{|z| > 1} \log{|1 - z^{-k}|} d\pclsp (M,\eta)(z).
\end{equation}
\end{lemma}

With the same notation as in Section \ref{chapter:surgery}, we have the following formula.

\begin{lemma} \label{lemma:ruellecusped}
	Let $(p,q) \in \mathcal{A}_{(M, \eta)}$, and $A = \{ \pm \varphi_{p_1/q_1},\dots, \pm \varphi_{p_l/q_l}\}$ 
	be the prime oriented core geodesics in $M_{p/q}$ added in the Dehn filling.
	For an integer $m \geq 3$, we have
	\[
	\log\left | \frac{\tor(M; \rho_{2m}^{p/q} )}{\tor(M; \rho_{4}^{p/q} ) } \right| =
	-\frac{(m-2)(m+2)}{2} \sum_{i = 1}^l \length(\varphi^i_{p/q}) - \frac{1}{\pi} \Vol(M_{p/q})(m^2 - 4)  + \sum_{k= 2}^{m-1} B_{2k +1}^{p/q}, 
	\]
	where  
	\[ 
	B_j^{p/q} = 
	\sum_{\substack{ \varphi \in  \mathcal{PC}(M_{p/q}) \setminus A } }
	\log{\left|1 - e^{-j \lambda_{p/q}(\varphi)/2 } \right| }.
	\]
\end{lemma}
\begin{proof}
	For the sake of simplicity we will prove it only for one-cusped manifolds.
	The surgery formula given by Lemma \ref{lemma:evensurgeryform}, yields 
	\begin{eqnarray*}
		\log \left| \frac{ \tor(M_{p/q}; \rho_{2m}^{p/q} )} {\tor(M; \rho_{2m}^{p/q} )} \right| &=&
		\sum_{k=0}^{m-1} \log \left| \left( e^{\left( \frac{1}{2} + k\right) \lambda(\varphi_{p/q})} - 1 \right) 
		\left( e^{ - \left( \frac{1}{2} + k\right) \lambda(\varphi_{p/q})} - 1 \right) \right|.
	\end{eqnarray*}
	It follows that, 
	\begin{equation*}
		\log \left| \frac{ \tor(M_{p/q}; \rho_{2m}^{p/q} ) \tor(M; \rho_{4}^{p/q}) }{\tor(M_{p/q}; \rho_{4}^{p/q} )
		\tor(M; \rho_{2m}^{p/q} )}\right|
		  =    \sum_{k=2}^{m-1} \log \left| \left( e^{\left( \frac{1}{2} + k\right) \lambda(\varphi_{p/q})} - 1 \right) 
		\left( e^{ - \left( \frac{1}{2} + k\right) \lambda(\varphi_{p/q})} - 1 \right) \right|.
	\end{equation*}
	Since $M_{p/q}$ is compact we can apply M\"uller's Theorem \ref{thm:CuspedRuelle}.
    Denoting by $R^{p/q}_{2k+1}(s)$ the Ruelle zeta function $R_{2k+1}(s)$ attached to the manifold $M_{p/q}$, we get:
	\begin{equation*}
		\log \left| \frac{\tor(M_{p/q}; \rho^{p/q}_{2m} )} {\tor(M_{p/q}; \rho^{p/q}_{4} )} \right| =
		\sum_{k=2}^{m-1} \log{ \left| R^{p/q}_{2k+1}\left( k + \frac{1}{2}\right) \right| } - \frac{1}{\pi} \text{Vol}(M_{p/q})(m^2-4).
	\end{equation*}
	From these last two equations, we get 
	\begin{eqnarray*} 
		-\log \left| \frac{\tor(M; \rho^{p/q}_{4} )} {\tor(M; \rho^{p/q}_{2m} )} \right| & = &
		\sum_{k=2}^{m-1} \log{ \left| R^{p/q}_{2k+1}\left( k + \frac{1}{2}\right) \right| } - \frac{1}{\pi} \text{Vol}(M_{p/q})(m^2-4) \\
		&- & \sum_{k=2}^{m-1} \log \left| e^{(\frac{1}{2} + k)\lambda(\varphi_{p/q})} - 1\right| \left|e^{-(\frac{1}{2} + k)\lambda(\varphi_{p/q})} - 1\right|.
	\end{eqnarray*}
	Using the expression
	\begin{equation*} 
		\log | R^{p/q}_{2k+1}(k +\frac{1}{2})| =
		\log |1- e^{-(k+\frac{1}{2})\overline{\lambda(\varphi_{p/q})}}|^2 + B_{2k + 1}^{p/q},
	\end{equation*}
	the previous equation is written as
	\begin{eqnarray*}
		\log \left| \frac{\tor(M; \rho_{2m}^{p/q} )}{\tor(M; \rho_{4}^{p/q} )}\right| & = & 
		\sum_{k=2}^{m-1} \log \frac{ | 1 - e^{- ( k + \frac{1}{2}) \overline{ \lambda(\varphi_{p/q}) } } |^2}
			{| e^{(\frac{1}{2} + k)\lambda(\varphi_{p/q})} - 1 | | e^{-(\frac{1}{2} + k)\lambda(\varphi_{p/q})} 
			- 1 |} 	 \\
			& - & \frac{1}{\pi} \text{Vol}(M_{p/q})(m^2-4) + \sum_{k= 2}^{m-1} B_{2k +1}^{p/q}.
	\end{eqnarray*}
	We have, 
	\begin{eqnarray*}
		\frac{ | 1 - e^{- ( k + \frac{1}{2}) \overline{ \lambda(\varphi_{p/q}) } } |^2}
		{| e^{(\frac{1}{2} + k)\lambda(\varphi_{p/q})} - 1 | | e^{-(\frac{1}{2} + k)\lambda(\varphi_{p/q})} 
		- 1 |} = e^{-(\frac12 + k) \Re \lambda(\varphi_{p/q})}.
	\end{eqnarray*}
	Hence, summing up the terms, we get 
	\[
	\sum_{k=2}^{m-1} \log \left( \frac{ | 1 - e^{- ( k + \frac{1}{2}) \overline{ \lambda(\varphi_{p/q}) } } |^2}
	{| e^{(\frac{1}{2} + k)\lambda(\varphi_{p/q})} - 1 | | e^{-(\frac{1}{2} + k)\lambda(\varphi_{p/q})} - 1 |} \right) =
		- \frac{(m-2)(m+2)}{2} \length(\varphi_{p/q}),
	\]
	and the lemma follows.
\end{proof}

\begin{lemma} \label{lemma:finitesumRuelle}
	With the same notation as in the preceding lemma, for $k \geq 5$ we have
	\[
	\lim_{(p,q) \to \infty } B_k^{p/q} = \log\left| R_{k}\left(\frac{k}{2}\right)\right|.
	\]
	Moreover, the following series is absolutely convergent
	\[
	\sum_{k= 5}^\infty \log\left|R_{k}\left( \frac k2\right) \right|.
	\]
\end{lemma}
\begin{proof}
	Let $\delta$ be the length of the shortest closed geodesic in $M$.
	By Lemma \ref{lemma:shortgeocore}, for $(p,q)$ large enough, the only prime closed geodesics on $M_{p/q}$ whose lengths
	are less than $\delta/2$ are the core geodesics $A = \{ \pm \varphi_{p_1/q_1},\dots, \pm \varphi_{p_l/q_l}\}$.
	In that case,
	\[
	B_k^{p/q} = 
	\sum_{\substack{ \varphi \in  \mathcal{PC}(M_{p/q}) \setminus A } }
	\log{\left|1 - e^{-k \lambda(\varphi)/2 } \right| }
	= \int_{ |z| > e^{\delta/4} } \log{|1 - z^{-k}|} d\mu_{p/q} (z),
	\]
	where $\mu_{p/q} = \pclsp(M_{p/q},\eta_{p/q})$.
	Now we want to apply Proposition \ref{prop:spclspimprovedcont}.
	We shall show that for large $|z|$ we have
	\begin{eqnarray} \label{proof:ineqzk}
		\left| \log{|1 - z^{-k}|} \right| \leq \frac{C}{z^5}, \quad\text{ for } k\geq 5,
	\end{eqnarray}
	for some constant $C$.
	First notice that for $w \in \cmplx$ with $|w| < 1$ the following inequality holds
	\[
	\left| \log{ \left| 1 - w \right| } \right| \leq - \log{ \left| 1 - |w|\right| }.
	\]
	On the other hand, for $|w|$ small enough,
	\[
	-\log | 1 - |w| | \sim  |w|.
	\] 
	Inequality \eqref{proof:ineqzk} then follows easily  from the last two inequalities.
	Therefore, we can use Proposition \ref{prop:spclspimprovedcont} to conclude that
	\[
	\lim_{(p,q) \to \infty } B_k^{p/q} = \log\left|R_{k}\left( \frac k2\right) \right|.
	\]
	Finally, if $\mu = \pclsp(M,\eta)$, we have
	\begin{eqnarray*}
		\sum_{k= 5}^\infty \left| \log\left|R_{k}\left( \frac k2\right) \right| \right| & \leq & \sum_{k = 5}^\infty \int_{|z| > e^{\delta/2} } | \log{\big| 1 - |z|^{-k} | }| d\mu(z) \\
		& \leq & \sum_{k = 5}^\infty \int_{ |z| > e^{\delta/2} } \frac{C}{ |z|^k } d\mu(z) \\ 
		& =    & \int_{|z| > e^{\delta/2} } \frac{C}{|z|^5} \frac{1}{1- \frac{1}{|z|}} d\mu(z) \\
		& \leq & \frac{C}{1 - e^{\delta/2} } \int_{|z| > e^{\delta/2} } \frac{1}{|z|^5} d\mu(z) < \infty,
	\end{eqnarray*}
	the last integral being finite by Proposition \ref{prop:spclspimprovedcont}.
\end{proof}

Finally, letting $(p,q)$ go to infinity in the equation of Lemma \ref{lemma:ruellecusped},
using the continuity of the complex-length spectrum, the continuity of the volume,
and the fact that the lengths of the core geodesics $\varphi^i_{p/q}$ go to zero,
we deduce the following generalization of Theorem \ref{thm:MullerRuelle} for even dimensions $n$.
In the following theorem we have also included the odd dimensional case, as its proof is handled in a similar way.

\begin{theorem} \label{thm:CuspedRuelle}
	Let $M$ be a complete hyperbolic $3$-manifold of finite volume.
	Then for $m \geq 3$ 
	\[
		\log \left| \frac{\normaltor_{2m + 1} (M) }{\normaltor_5(M)} \right|  =  
		 \sum_{k=3}^m \log{|R_{2k}(k)| } - \frac{1}{\pi} \Vol M \left( m(m+1) -6 \right).
		 \]
	If in addition $M$ is enriched with an acyclic spin structure, then for $m \geq 3$ 
	\[
		 \log \left| \frac{\normaltor_{2m}(M, \eta )}{\normaltor_4(M, \eta) } \right|  =  
		 \sum_{k=2}^{m-1} \log{ \left| R_{2k+1}\left( k +\frac{1}{2}\right) \right| } - \frac{1}{\pi} \Vol M( m^2 - 4). 
	 \]
\end{theorem}

The proof of Theorem \ref{thm:asymp_main} now follows easily.
\begin{proof}[Proof of Theorem \ref{thm:asymp_main}]
	Theorem \ref{thm:CuspedRuelle} and Lemma \ref{lemma:finitesumRuelle} imply that
	\[
	\lim_{n\to\infty} \frac{ \log \left| \normaltor_{n} (M,\eta) \right|}{n^2} = -\frac{\Vol M}{4\pi}. 
	\]
\end{proof}

\section{Reidemeister torsion and length spectrum}
\label{chapter:analysis}

The results of last section, especially Theorem \ref{thm:CuspedRuelle}, show that there is a close relationship between 
the spin-complex-length spectrum of a complete, acyclic, spin-hyperbolic $3$-manifold of finite volume $(M,\eta)$
and its higher-dimensional Reidemeister torsion invariants.
In this section we want to focus on this question; 
more concretely, we want to study at what extent the sequence $\{\normaltor_n(M, \eta)\}$ determines
the spin-complex-length spectrum of the manifold. 
The equivalence between these two invariants should be regarded as a geometric interpretation 
of the information encoded in these invariants.

\begin{definition}
	We will say that two (spin-)hyperbolic $3$-manifolds are \emph{(spin-)isospectral}
	if the have the same prime (spin-)complex-length spectrum.
\end{definition}

The notion of isospectrality, as stated in this definition, 
is considered by C.\;Maclachlan and A.W.\;Reid in \cite{ReidMAC}.
They prove the following theorem.
\begin{theorem}[C.\;Maclachlan and A.W.\;Reid, \cite{ReidMAC}]\label{thm:MacR}
	For any integer $n\geq 2$, there are $n$ isospectral non-isometric closed hyperbolic $3$-manifolds.
\end{theorem}
As an immediate consequence of Theorem \ref{thm:MacR} and Wotzke's Theorem \ref{thm:Wotzke}, we get the following result.
\begin{theorem}
	For any integer $n\geq 2$, there are $n$ non-isometric, closed, hyperbolic $3$-manifolds
	$M_1,\dots,M_n$ such that for all $k>0$,
	\[
	|\tau_{2k+1}(M_i)| = |\tau_{2k+1}(M_j)|,\quad \text{for all } i,j = 1,\dots,n.
	\]
\end{theorem}

Unfortunately, we will need to weaken the notion of isospectrality, and rather consider 
isospectrality up to complex conjugation. Before giving its definition, let us make the following considerations.
Let $(M,\eta)$ be a spin-hyperbolic $3$-manifold, and let $(\overline{M}, \eta)$ be the corresponding spin manifold
with the orientation reversed 
(here we are using the canonical one-to-one correspondence between spin structures on $M$ and $\overline{M}$).
The relationship between the spin-complex-length spectra of these two manifolds is easily established.
Indeed we have:
\begin{eqnarray*}
	\pclsp (M,\eta) & = & \sum_{\varphi\in \mathcal{PC}(M)} \delta_{e^{\lambda(\varphi)/2}}, \\
	\pclsp (\overline{M},\eta) & = & \sum_{\varphi\in \mathcal{PC}(M)} \delta_{e^{\overline{\lambda(\varphi)}/2}},
\end{eqnarray*}
where $\lambda(\varphi)$ is the spin-complex-length function of $(M,\eta)$. 
Notice that $\pclsp (\overline{ M},\eta)$ is the image measure of $\pclsp (M,\eta)$ under 
the complex conjugation map.

\begin{definition}
	We will say that two complete spin-hyperbolic $3$-manifolds $(M_1,\eta_1)$ and $(M_2,\eta_2)$ 
	are \emph{spin isospectral up to complex conjugation} if they have the same spin-complex-length spectrum
	up to complex conjugation, that is, 
	\[
	\pclsp(M_1,\eta_1) + \pclsp (\overline{M_1},\eta_1) = \pclsp(M_2,\eta_2) + \pclsp (\overline{M_2},\eta_2).
	\]
	The definition for ``non-spin'' manifolds is analogous. 
\end{definition}
\begin{rem}
	The reason to consider isospectrality up to complex conjugation is essentially that 
	Wotzke's Theorem \ref{thm:Wotzke}
	is an equality between the moduli of the Ruelle zeta function and Reidemeister torsion;
	if we had also equality between the arguments, then there should be no need
	to consider isospectrality \emph{up to complex conjugation}.
\end{rem}
\begin{rem}
	If two complete spin-hyperbolic $3$-manifolds are spin-isospectral up to complex conjugation,
	then they have the same real length spectrum. The same holds true for non-spin manifolds.
\end{rem}

\begin{theorem} \label{thm:torsiondetclsp}
	Let $(M_1, \eta_1), (M_2,\eta_2)$ be two complete spin acyclic hyperbolic $3$-manifolds of finite volume.
	Assume that there exists $N\geq 4$ such that for all $n \geq N$ we have
	\[
	|\normaltor_n (M_1, \eta_1)| = |\normaltor_n (M_2, \eta_2)|.
	\]
	Then the following assertions hold:
	\begin{enumerate}
		\item The spin manifolds $(M_1, \eta_1)$ and $(M_2,\eta_2)$ are spin-isospectral 
		up to complex conjugation. In particular, they have the same real length spectrum.
		\item The equality $|\normaltor_n(M_1,\eta_1)| = |\normaltor_n (M_2, \eta_2)|$ holds 
			for all $n \geq 4$.
	\end{enumerate}
\end{theorem}

The proof of Theorem \ref{thm:torsiondetclsp} will be given in Section \ref{section:isospectrality}.
Before doing that, we need a result on complex analysis which we prove in the following subsection.

\subsection{A result on complex analysis}

The aim of this subsection is to provide a proof of the following analytical result needed to prove Theorem \ref{thm:torsiondetclsp}. We are indebted to J.\;Ortega-Cerd\`a, N.\;Makarov, and A.\;Nicolau for the proof of 
Proposition \ref{prop:measurevanish2}.

\begin{proposition} \label{prop:measurevanish}
	Let $\mu$ be a Radon complex-valued measure with compact support $\supp \mu$ 
	contained in the interior of the unit disk $D$.
	Assume that $\mu$ satisfies the following conditions:
	\begin{enumerate}
		\item $\cmplx \setminus \supp \mu$ is connected.
		\item $\supp \mu$ has zero Lebesgue measure.
		\item There exists a positive integer $N$ and a holomorphic function $\psi$ on the open unit disk with $\psi(0) = 0$, 
			$\psi'(0) = 1$ such that
			\[
			\int_{D} \frac{\psi(z^n)}{z^N} d\mu(z) = 0.
			\]
			for all $n \geq N$. 
	\end{enumerate}
	Then $\mu = 0$.
\end{proposition}

We will prove first the particular case of  Theorem~\ref{thm:torsiondetclsp} given by taking $\psi = \operatorname{Id}$.
After this, we will show that if $\psi$ is any holomorphic function on the open unit disk with 
$\psi(0) = 0$ and $\psi'(0) = 1$, then for all $N > 0$ the linear span of 
$\{ \frac{\psi(z^n)}{z^N} \}_{n\geq N}$ is dense in the space of holomorphic functions on the open unit disk 
endowed with the topology of the uniform convergence on compact sets.

A way to prove Proposition \ref{prop:measurevanish} is to use the Cauchy transform.
If $\mu$ is a Radon complex-valued measure compactly supported in the complex plane,
then its Cauchy transform is defined by
\[
\widehat \mu(\zeta) = \int_{\cmplx} \frac{d\mu(z)}{z - \zeta}. 
\]
We will need only the following properties of the Cauchy transform, see \cite{Gamelin}.

\begin{proposition} \label{prop:Cauchytrans}
	Let $\widehat \cmplx$ be the Riemann sphere.
	The Cauchy transform has the following properties:
	\begin{enumerate}
		\item $\widehat \mu (\zeta)$ is analytic on $\widehat \cmplx \setminus \supp \mu$
			and vanishes at infinity.
		\item If $\widehat \mu = 0$ Lebesgue-almost everywhere, then $\mu = 0$.
	\end{enumerate}
\end{proposition}

With this result we can prove the following particular case of Proposition \ref{prop:measurevanish} mentioned above.

\begin{proposition} \label{prop:measurevanish2}
	Let $\mu$ be a Radon complex-valued measure compactly supported in the complex plane that
	satisfies the following conditions:
	\begin{enumerate}
		\item $\cmplx \setminus \supp \mu$ is connected.
		\item $\supp \mu$ has zero Lebesgue measure.
		\item For all $n\geq 0$, $\int_\cmplx z^n d\mu(z) = 0$.
	\end{enumerate}
	Then $\mu = 0$.
\end{proposition}
\begin{proof}
	Let $\widehat \mu(\zeta)$ be the Cauchy transform of $\mu$.
	We know that $\widehat \mu(\zeta)$ is analytic on $\widehat \cmplx \setminus \supp \mu$ and vanishes at $\infty$.
	Take $|\zeta|$ large enough so that $|z/\zeta| < 1$ for all $z \in \supp \mu$. Then, we have
	\[
	\widehat \mu(\zeta) = \int_{\cmplx} \frac{d\mu(z)}{z - \zeta} = - \frac{1}{\zeta} \int_{\cmplx} \frac{d\mu(z)}{1 - \frac{z}{\zeta}} =
	- \frac{1}{\zeta} \sum_{n\geq 0} \int_{\cmplx} \frac{z^n}{\zeta^n} d\mu(\zeta) = 0.
	\]
	The last term being zero by hypothesis. Thus $\widehat \mu$ is identically zero in a neighbourhood of $\infty$, 
	and hence it must be identically zero in $\widehat \cmplx \setminus \supp \mu$, as $\cmplx \setminus \supp \mu$
	is connected. Since $\supp \mu$ has zero Lebesgue measure, we have 
	$\widehat{\mu} = 0$ Lebesgue-almost everywhere. Proposition \ref{prop:Cauchytrans} then implies that
	that $\mu = 0$, as we wanted to prove.
\end{proof}

Now, to prove Proposition \ref{prop:measurevanish}, it remains to prove the following result.

\begin{proposition} \label{prop:logdense}
	Let $H(D)$ be the space of holomorphic functions on the open unit disk  
	endowed with the topology of the uniform convergence on compact sets, 
	and let $\psi \in H(D)$ such that $\psi(0) = 0$ and $\psi'(0) = 1$.
	Then, for all $N\geq 1$, the linear span of 
	\[
	\left \{ \frac{\psi(z^k)}{z^N} \right\}_{k\geq N}
	\]
	is dense in $H(D)$.
\end{proposition}

\begin{rem}
	Since we have not been able to find this result in the literature
	we provide a proof of it.
\end{rem}

In what follows, $\psi(z)$ will denote a fixed holomorphic function in $H(D)$ 
such that $\psi(0) = 0$ and $\psi'(0) = 1$. Thus we have,
\[
\psi(z) = z + \sum_{k \geq 1} \psi_k z^k, \quad \text{for all } z \in D.
\]

The fact that the linear span of the monomials $\{z^n\}_{n\geq 0}$ is dense in $H(D)$
implies that Proposition \ref{prop:logdense} is equivalent to say that for all $n \geq 0$
there exists a sequence $\{a_k^n\}_{k \geq N}$ of complex numbers such that
\[
z^n = \sum_{k\geq N} a_k^n \frac{\psi(z^k)}{z^N},
\]
with the right hand side converging uniformly on every compact set of $D$.
Using the power series expansion of $\psi(z)$, the  equality above
yields a linear system with $\{a_k^n\}_{k \geq N}$ as unknowns.
Since $\psi(0) = 0$ and $\psi'(0) = 1$, this system is lower triangular 
with ones in the diagonal, and hence it has a unique solution.
The difficult point is to prove the convergence of the corresponding sequence.
Fortunately, we can proceed in a slightly different way.

Let us denote by $H(D_R)$ the space of holomorphic functions on the open disk of radius $R$,
\[
D_R = \{ z \in \cmplx \mid |z| < R \}.
\]
We will work with the Bergman space on $D_R$, which is defined by
\[
A^2(D_R) = \left\{ f \in H(D_R) \mid \int_{D_R} |f(z)|^2 dA(z) < \infty \right\},
\]
where $dA(z)$ is the usual area measure, see \cite{BergmanSp} for details. 
It is well known that $A^2(D_R)$ is a Hilbert space with respect to the following inner product (see \cite{BergmanSp}),
\[
\langle f, g \rangle = \int_{D_R} f(z)\overline{g(z)} dA(z).
\]
The reason to consider $A^2(D_R)$ instead of $H(D_R)$ is due to the fact that it is a Hilbert space
(thus it is \emph{a priori} easier to deal with), and to the fact that convergence in the former implies
convergence in the latter, as expressed by the following result (see \cite{BergmanSp}).

\begin{proposition}
	If a sequence of functions $\{f_n \}$ in $A^2(D_R)$ converges to $f$ in $A^2(D_R)$,
	then $\{f_n \}$ converges to $f$ uniformly on each compact set of $D_R$.
\end{proposition}

For $0 < R < 1$, consider the linear operator $A_\psi \colon A^2(D_R) \to A^2(D_R)$, with domain the 
linear space of monomials, defined as follows:
\begin{eqnarray*}
	A_\psi (1) & = & 1, \\
	A_\psi (z^n) & = & \psi(z^n) = z^n + \sum_{j\geq 2} \psi_j z^{nj} \quad \text{for } n \geq 1. 
\end{eqnarray*}
The following result shows that $A_{\psi}$ is a bounded operator.
\begin{proposition}
	For $R < 1$, let $A_\psi = I + B_\psi$. 
	Then $B_\psi \colon A^2(D_R) \to A^2(D_R)$ is Hilbert-Schmidt.
	In particular, $B_\psi$ is compact and $A_\psi$ is bounded.
\end{proposition}
\begin{proof}
	A basis of $A^2(D_R)$ is given by the following functions,
	which are just normalizations of the monomials $\{ z^k \}$,
	\[
	\phi_n (z) = \sqrt{\frac{n+1}{\pi}} \frac{z^n}{R^{n+1}}.
	\]
	To be Hilbert-Schmidt then means that 
	\[
	\sum_{n\geq 0} \left \langle B_\psi(\phi_n), B_\psi(\phi_n) \right\rangle < \infty.
	\]
	In terms of the basis $\{\phi_n\}$, $B_\psi$ is written as follows: $B_\psi(\phi_0) = 0$, and
	for $n\geq 1$,
	\begin{eqnarray*}
		B_\psi (\phi_n) & = & \sqrt{\frac{n+1}{\pi} } \frac{1}{R^{n+1}} \sum_{j\geq 2} \psi_j z^{nj} 
	= \sqrt{\frac{n+1}{\pi}} \frac{1}{R^{n+1}} \sum_{j\geq 2} \psi_j \sqrt{\frac{\pi}{nj + 1}} R^{nj +1} \phi_{nj} \\
	& = &  \sum_{j\geq 2} \psi_j \sqrt{\frac{n + 1}{nj + 1}} R^{n(j-1)} \phi_{nj}.
	\end{eqnarray*}
	Therefore, 
	\begin{eqnarray*}
		\sum_{n\geq 0} \langle B_\psi(\phi_n), B_\psi(\phi_n) \rangle & = & \sum_{n\geq 1} 
		\sum_{j\geq 2} |\psi_j|^2 {\frac{n +1 }{nj + 1}} R^{2n(j-1)} \leq \sum_{j \geq 2} \frac{2|\psi_j|^2}{j} \sum_{n\geq 1} R^{2n(j-1)}\\
		& = & \sum_{j \geq 2} \frac{2|\psi_j|^2}{j} \frac{ R^{2(j-1)}}{1-R^{2(j-1)}} 
		\leq \frac{2}{R^2 (1- R^2)} \sum_{j \geq 2} \frac{|\psi_j|^2}{j} R^{2j}.
	\end{eqnarray*}
	The last series is finite because it is exactly $\pi$ times the square of the norm in $A^2(D_R)$ of $(\psi(z) - z)/z$.
	Indeed,
	\[
	\left\| \frac{\psi(z) - z}{z} \right\|_{A^2(D_R)}^2 = \sum_{j \geq 1} |\psi_{j+1}|^2 
	\left\| z^j \right\|_{A^2(D_R)}^2 = \pi \sum_{j \geq 1} |\psi_{j+1}|^2 \frac{R^{2(j+1)}}{j +1}.
	\]
\end{proof}
\begin{corollary}
	For $R < 1$, the operator $A_\psi \colon A^2(D_R) \to A^2(D_R)$ is invertible.
\end{corollary}
\begin{proof}
	We have  $A_\psi = I + B_\psi$, with $B_\psi$ a compact operator. 
	The matrix of the operator $A_\psi$ in the basis $\{\phi_n\}$ is lower triangular, and has ones 
	in the diagonal; hence, the kernel of $A_\psi$ is trivial, and the Fredholm alternative 
	implies that $A_\psi$ is invertible.
\end{proof}

\begin{corollary}
	The linear span of $\{ 1, \psi(z), \psi(z^2), \dots \}$ is dense in $H(D)$.
\end{corollary}
\begin{proof}
	Let us fix $g(z) \in H(D)$. Let $ 0 < R < 1$. By the previous corollary,
	there exists $f_R(z) \in A^2(D_R)$ such that $A_\psi(f_R) = g$.
	Then we have 
	\[
	f_R(z) = \sum_{n\geq 0} a_n(R) z^n,
	\]
	and the series $ a_0(R) + \sum_{n\geq 1} a_n(R) \psi(z^n)$ converges to $g$ in $A^2(D_R)$, so
	it converges uniformly to $g$ in every compact contained in $D_R$. Since $f_R(z)$ also belongs to $A^2(D_{R'})$ 
	for all $ 0 < R' < R$, and is holomorphic, the coefficients $a_n(R)$ are independent of $R$, so $a_n(R) = a_n$.
	Hence $ a_0 + \sum_{n\geq 0} a_n \psi(z^n) $ converges to $g$ in every compact set contained in the unit disk,
	and this proves the result.
\end{proof}
The proof of Proposition \ref{prop:logdense} now follows easily.
\begin{proof}[Proof of Proposition \ref{prop:logdense}]
	Consider the following linear subspace of $H(D)$
	\[
	C_N = \{ \phi \in H(D) \mid \phi^{(j)}(0) = 0,\; 0 \leq j \leq N - 1 \}.
	\]
	Since the derivative is continuous in $H(D)$, $C_N$ is closed in $H(D)$.
	By the preceding corollary, it follows that $C_N$ is the closure of $\{\psi(z^N), \psi(z^{N+1}), \dots\}$. 
	On the other hand, $C_N$ is homeomorphic to $H(D)$ via the linear map
	\begin{eqnarray*}
		H(D) & \to & C_N \\
		\phi(z) & \mapsto & z^N \phi(z).
	\end{eqnarray*}
	Therefore, the closure of the linear span of $\{\frac{\psi(z^k)}{z^N}\}_{k\geq N}$ is the whole $H(D)$, 
	as we wanted to prove.
\end{proof}

\subsection{Isospectrality and torsion} \label{section:isospectrality}

We start this subsection with the proof of Theorem \ref{thm:torsiondetclsp}.

\begin{proof}[Proof of Theorem \ref{thm:torsiondetclsp}]
	We can assume that $N \geq 6$. 
	Let us put $\mu_i = \pclsp (M_i,\eta_i)$ and $\overline{\mu}_i = \pclsp (\overline{M_i},\eta_i)$, for $i = 1, 2$.
	From Theorem \ref{thm:CuspedRuelle} we deduce that for $k \geq 3$,
	\begin{eqnarray*}
		\log\left|\frac{ \normaltor_{2k + 3}(M_i)}{\normaltor_{2k + 1}(M_i)}\right| & = & 
		\log{|R_{2k + 2}^{M_i}(k + 1)|}  - \frac{2(k + 1)}{\pi} \Vol M_i \\
		\log\left|\frac{ \normaltor_{2k + 2}(M_i)}{\normaltor_{2k}( M_i)}\right|  & = & 
		\log{\left|R_{2k + 1}^{M_i}\left( k +\frac{1}{2} \right)\right|} - \frac{2k +1}{\pi}\Vol M_i.
	\end{eqnarray*}
	By hypothesis, for all $n \geq N$, $|\normaltor_n (M_1,\eta_1)| = |\normaltor_n (M_2,\eta_2) |$.
	Then, by Theorem \ref{thm:asymp_main}, we have $\Vol M_1 = \Vol M_2$. 
	On the other hand, by Lemma \ref{eq:ruellelsp}, we have:
	\begin{eqnarray*}
		\log{\left|R_{j}^{M_i}\left( \frac{j}{2} \right)\right|}  =  \int_{|z| > 1} \log{|1 - z^{-j}|} d\mu_i(z).
	\end{eqnarray*}
	Therefore, for all $n \geq N + 1$, we have
	\begin{equation} \label{proof:eqintclps}
		\int_{|z|>1} \log{|1 - z^{-n}|} d\mu_1(z) = \int_{|z|>1} \log{|1 - z^{-n}|} d\mu_2(z).
	\end{equation}
	On the other hand,
	\[
	\int_{|z|>1} 2 \log{|1 - z^{-n}|} d\mu_i(z) = \int_{|z|>1} \log{(1 - z^{-n})} d\mu_i(z) + \int_{|z|>1} \log{(1 - z^{-n})} d\overline{\mu}_i(z).
	\]
	Let $\nu_i$ be the image measure of $\mu_i + \overline{\mu}_i$ under the map $z\mapsto \frac 1z$.
	Then Equation \eqref{proof:eqintclps} is equivalent to,
	\[
	\int_{|z| < 1} \log(1 - z^n)d\nu_1(z) = \int_{|z| < 1} \log(1 - z^n)d\nu_2(z),
	\]
	for all $n \geq N + 1$. The measure $\nu_i$ is not Radon since any neighbourhood of the origin has infinite measure.
	Nevertheless, by Proposition \ref{prop:spclspimprovedcont}, $z^5 \nu_i$ is finite. 
	Hence, $\nu = z^{N + 1}(\nu_1 - \nu_2)$ is a Radon measure that satisfies
	\[
	\int_{|z| < 1} \frac{\log(1 - z^n)}{z^{N+1}} d\nu(z) = 0,\quad \text{for all } n \geq  N + 1.
	\]
	Now we can apply Proposition \ref{prop:measurevanish} with $\psi(z) = -\log(1 -z)$ to conclude
	that $\nu = 0$, which is equivalent to say that 
	\[
	\mu_1 + \overline{\mu}_1 = \mu_2 + \overline{\mu}_2.
	\]
	The first part of the theorem is then proved.
	The second part is now easily deduced by using the first part and Theorem \ref{thm:CuspedRuelle}.
\end{proof}

A similar proof shows that the following result holds.

\begin{theorem} \label{thm:torsiondetclspnonspin}
	Let $M_1$ and $M_2$ be two complete hyperbolic $3$-manifolds of finite volume.
	Assume that there exists $K\geq 2$ such that for all $k \geq K$ we have
	\[
	|\normaltor_{2k +1} (M_1)| = |\normaltor_{2k+1}(M_2)|.
	\]
	Then the following assertions hold:
	\begin{enumerate}
		\item The manifolds $M_1$ and $M_2$ are isospectral (as ``non-spin'' manifolds) 
		up to complex conjugation.
			In particular, they have the same real length spectrum.
		\item The equality $|\normaltor_{2k+1}(M_1)| = |\normaltor_{2k+1} (M_2)|$ holds 
			for all $k \geq 2$.
	\end{enumerate}
\end{theorem}
\begin{proof}[Proof of Theorem \ref{thm:torsiondetclspnonspin}]
	The proof is the same as the proof of Theorem \ref{thm:torsiondetclsp}, 
	but considering only odd dimensional representations.
\end{proof}

Wotzke's Theorem \ref{thm:Wotzke} and the previous theorem yield the following result.
\begin{theorem}
	Let $(M_1, \eta_1)$ and $(M_2,\eta_2)$ be two \emph{closed} spin-hyperbolic $3$-manifolds.
	Then the following assertions are equivalent:
	\begin{enumerate}
		\item There exists $N \geq 2$ such that for all $n \geq N$,
			\[
			|\tor_n (M_1, \eta_1)| = |\tor_n (M_2, \eta_2)|.
			\]
		\item The manifolds $(M_1, \eta_1)$ and $(M_2,\eta_2)$ are isospectral up to conjugation.
	\end{enumerate}
\end{theorem}

\begin{proof}
	If the first assertion is true, then Theorem \ref{thm:torsiondetclsp} implies
	that the two manifolds are isospectral up to complex conjugation. 
	In order to prove the converse, let us make the following observation.
	Let $(M,\eta)$ be a closed spin-hyperbolic $3$-manifold.
	By definition, the spin-complex-length spectrum of $\pclsp(M,\eta)$ determines the 
	Ruelle zeta function $R^M_{\rho_n}(s)$, and if we only know it up to conjugation,
	it determines the following function,
	\[
	F_M(s) := R^M_{\rho_n}(s) R^{\overline M}_{\rho_n}(s).
	\]
	By definition, for $\Re(s) > 2 + n/2$, we have:
	\[
	\overline{R^{\overline M}_{\rho_n}(\bar s)} = R^{M}_{\rho_n}(s).
	\]
	Since both right and left hand side of this equation are meromorphic functions,
	the equality  must hold for all $s\in \cmplx$. By Wotzke's Theorem \ref{thm:Wotzke}, we get then:
	\[
	F_M(0) = R^M_{\rho_n}(0) R^{\overline M}_{\rho_n}(0) = |R^M_{\rho_n}(0)|^2 = |\tau_n(M,\eta)|^{4}.
	\]
	The proof that Assertion 2 implies Assertion 1 is now clear.
\end{proof}

\bibliographystyle{alpha} 
\bibliography{torsion} 

\medskip

\textsc{Departament de Matem\`atiques, Universitat Aut\`onoma de Barcelona.}
\textsc{08193 Bellaterra, Spain}

{pmenal@mat.uab.cat}
{porti@mat.uab.cat}

\end{document}